\documentclass{article}
\pdfoutput=1
\usepackage{arxiv}

\usepackage[bbgreekl]{mathbbol}
\usepackage{amsfonts}
\usepackage{amsthm}
\usepackage{subcaption}
\usepackage{booktabs, caption}
\usepackage{mathtools}
\usepackage{mathrsfs}
\usepackage{enumerate}
\usepackage{algorithm, algpseudocode}
\usepackage{multirow}

\usepackage{float}
\usepackage[section]{placeins}

\usepackage{color}
\usepackage{soul,xargs}
\usepackage{array}  
\usepackage{graphicx}
\usepackage{multicol}
\usepackage[colorinlistoftodos,prependcaption,textsize=tiny]{todonotes}
\usepackage{url}

\DeclareMathOperator{\Rank}{rank}

\DeclareMathOperator{\vect}{vec}

\newcommand{\bs}[1]{\boldsymbol{#1}}

\newcommand{\argmin}{\mathop{\mathrm{arg\,min}}}

\newcommand{\half}{\frac{1}{2}}

\newcommand{\LRc}[1]{\left\{ #1 \right\}}
\newcommand{\LRp}[1]{\left( #1 \right)}
\newcommand{\LRs}[1]{\left[ #1 \right]}

\newcommand{\GM}[2]{\mc{N}\left( #1, #2 \right)}

\def\etal{{\it et al.~}}

\newcommand{\db}{\mb{d}}

\newcommand{\umap}{\ub^{\text{\normalfont{MAP}}}}
\newcommand{\umapSA}{\hat{\ub}^{\text{\normalfont{MAP}}}_\N}
\newcommand{\urma}{\ub^{\text{\normalfont{RMA}}}}
\newcommand{\urmap}{\ub^{\text{\normalfont{RMAP}}}}
\newcommand{\urs}{\ub^{\text{\normalfont{RS}}}}
\newcommand{\ursuOne}{\ub^{\text{\normalfont{RS\_U1}}}}
\newcommand{\urmarmap}{\ub^{\text{\normalfont{RMA+RMAP}}}}
\newcommand{\uenkf}{\ub^{\text{\normalfont{ENKF}}}}
\newcommand{\F}{\mc{F}}
\newcommand{\sCinv}{\sC^{-1}}
\newcommand{\Linv}{\L^{-1}}

\renewcommand{\Pr}[1]{\mathbb{P}\LRs{#1}}
\newcommand{\nor}[1]{\left\| #1 \right\|}
\newcommand{\snor}[1]{\left| #1 \right|}
\newcommand{\norm}[1]{\left \Vert #1 \right \Vert}
\newcommand{\abs}[1]{\left \vert #1 \right \vert}

\newcommand{\grad}{\nabla}
\newcommand{\bigO}{\mathcal{O}}

\newcommand{\mc}[1]{\mathcal{#1}}

\newcommand{\R}{\mathbb{R}}
\newcommand{\prob}{\mathbb{P}}
\renewcommand{\L}{{\bs{\Sigma}}}
\newcommand{\A}{\mc{A}}

\newcommand{\N}{N}
\newcommand{\M}{M}
\newcommand{\n}{n}
\newcommand{\I}{\mc{I}}
\newcommand{\J}{\mc{J}}
\renewcommand{\j}{j}

\renewcommand{\u}{u}
\newcommand{\U}{U}
\newcommand{\V}{V}
\renewcommand{\S}{S}
\newcommand{\ub}{{\bs{\u}}}
\newcommand{\zb}{{\bs{z}}}
\newcommand{\eb}{{\bs{e}}}
\newcommand{\sigb}{{\bs{\sigma}}}
\newcommand{\eps}{{{\varepsilon}}}
\newcommand{\epsb}{{\bs{\eps}}}
\newcommand{\lamb}{{\bs{\lambda}}}

\newcommand{\Ub}{{\bs{\U}}}
\newcommand{\Vb}{{\bs{\V}}}
\newcommand{\Sb}{{\bs{\S}}}

\newcommand{\Jstoch}{\tilde\J}
\newcommand{\Jsa}{\Jstoch_\N}
\newcommand{\Jexpect}{\mathscr{J}}
\newcommand{\JexpectSA}{\Jexpect_{\N}}
\newcommand{\rand}{{\bs{\xi}}}
\newcommand{\randP}{\mathcal{\pi}}
\newcommand{\randSpace}{\bs{\Xi}}

\newcommand{\del}{{\bs{\delta}}}
\newcommand{\lam}{{\bs{\lambda}}}
\newcommand{\omb}{{\bs{\omega}}}
\newcommand{\taub}{{\bs{\tau}}}
\newcommand{\epsP}{\randP_{\epsb}}
\newcommand{\delP}{\randP_{\del}}
\newcommand{\lamP}{\randP_{\lam}}
\newcommand{\sigP}{\randP_{\sigb}}
\newcommand{\omP}{\randP_{\omb}}
\newcommand{\ident}{\I}

\newcommand{\reals}{{\mathbb{R}}}

\renewcommand{\P}{\bs{P}}

\newcommand{\Sn}{\bs{S}_{\N}}

\newcommand{\Ln}{\bs{L}_{\N}}

\newcommand{\sigbarN}{\bar{\sigb}_N}
\newcommand{\delbarN}{\bar{\del}_N}
\newcommand{\meanIN}{\frac{1}{\N} \sum_{i=1}^{\N}}

\newcommand{\tol}{\beta}

\newcommand{\lip}{\mathcal{L}}

\newcommand{\likeb}{{\pi_{\text{like}}\LRp{\db|\ub}}}
\newcommand{\post}{{\pi\LRp{\ub | \db}}}

\newcommand{\prior}{{\pi_{\text{prior}}\LRp{\ub}}}

\newcommand{\Ex}{{\mathbb{E}}}
\newcommand{\expect}{\Ex}

\newcommand{\piprior}{{\pi_{\text{prior}}}}

\newcommand{\asconv}{\stackrel{a.s.}{\rightarrow}}

\newcommand{\figlab}[1]{\label{fig:#1}}
\newcommand{\eqnlab}[1]{\label{eq:#1}}
\newcommand{\theolab}[1]{\label{theo:#1}}
\newcommand{\corolab}[1]{\label{coro:#1}}
\newcommand{\propolab}[1]{\label{propo:#1}}
\newcommand{\lemlab}[1]{\label{lem:#1}}

\newcommand{\figref}[1]{\ref{fig:#1}}
\newcommand{\theoref}[1]{\ref{theo:#1}}

\newcommand{\cororef}[1]{\ref{coro:#1}}
\newcommand{\proporef}[1]{\ref{propo:#1}}
\newcommand{\lemref}[1]{\ref{lem:#1}}
\newcommand{\eqnref}[1]{\eqref{eq:#1}}

\newcommand{\seclab}[1]{\label{sect:#1}}
\newcommand{\secref}[1]{\ref{sect:#1}}

\algblock{ParFor}{EndParFor}
\algnewcommand\algorithmicparfor{\textbf{For all}}
\algnewcommand\algorithmicpardo{\textbf{do in parallel}}
\algnewcommand\algorithmicendparfor{\textbf{end\ For all}}
\algrenewtext{ParFor}[1]{\algorithmicparfor\ #1\ \algorithmicpardo}
\algrenewtext{EndParFor}{\algorithmicendparfor}

\def\jmpu{{\lbrack\!\lbrack \widetilde{\gamma}> \\
  <g,\widetilde{\Balpha}> \\
<g,\widetilde{\Bbeta}> \end{bmatrix}$, whereu\rbrack\!\rbrack}}

\def\sC{{\bs{\Gamma}}}

\def\Balpha{\mbox{\boldmath$\alpha$}}
\def\Bbeta{\mbox{\boldmath$\beta$}}

\newcommand{\mb}[1]{\mathbf{#1}}

\def\bd{\mb{d}}

\def\by{\mb{y}}

\newcommand{\TheTitle}{On Unifying  Randomized Methods for Inverse 
Problems}

\makeatletter
\newcommand{\specificthanks}[1]{\@fnsymbol{#1}}
\makeatother

\title{{\TheTitle}}

\author{Jonathan Wittmer \\ Oden Institute \\ University of Texas at Austin \\ Austin, TX 78712  
        \And C G Krishnanunni \\ Department of Aerospace Engineering and Engineering Mechanics \\ University of Texas at Austin \\ Austin, TX 78712 
        \And Hai V. Nguyen \\ Department of Aerospace Engineering and Engineering Mechanics \\ University of Texas at Austin \\ Austin, TX 78712   
        \And  Tan Bui-Thanh \\ Oden Institute \\ Department of Aerospace Engineering and Engineering Mechanics \\ University of Texas at Austin \\ Austin, TX 78712 } 

\newcounter{theoremcount}
\newtheorem{theorem}[theoremcount]{Theorem}
\newtheorem{proposition}[theoremcount]{Proposition}
\newtheorem{lemma}[theoremcount]{Lemma}
\newtheorem{corollary}[theoremcount]{Corollary}
\newtheorem{remark}[theoremcount]{Remark}

\begin{document}

\bibliographystyle{vancouver}
\maketitle

\begin{abstract}  
This work unifies the analysis of various randomized methods for solving linear and nonlinear inverse problems
by framing the problem in a stochastic optimization setting. By doing so, 
we show that many randomized methods are variants of a sample average approximation.
More importantly, we are able to prove a single theoretical
result that guarantees the asymptotic convergence for a variety of randomized methods.
Additionally, viewing randomized methods as a sample average approximation enables us to 
prove, for the first time, a single non-asymptotic error result that holds for randomized methods under consideration. Another important consequence of our unified framework is that it allows us to discover new randomization methods. We present various numerical
results for linear, nonlinear, algebraic, and PDE-constrained inverse problems that verify the theoretical convergence results and provide a discussion on the apparently different convergence rates and the behavior for
various randomized methods. 
\end{abstract}
\keywords{Randomization, Bayesian Inversion, Ensemble Kalman Filter, randomized maximum {\textit a posteriori}}

\section{Introduction}
Solving large-scale ill-posed inverse problems that are governed by
partial differential equations (PDEs), though tremendously
challenging, is of great practical
importance in science and engineering. Classical deterministic inverse methodologies, which provide
point estimates of the solution, are not capable of 
accounting for the uncertainty in the inverse solution in a principled way. The Bayesian
formulation provides a systematic quantification of uncertainty by
posing the inverse problem as one of 
statistical inference. 
The Bayesian framework for inverse problems proceeds as follows: given
observational data $\db \in \R^k$ and their uncertainty, the governing forward
problem and its uncertainty, and a prior probability density function
describing uncertainty in the parameters $\ub \in
\R^n$, the solution of the inverse problems is the posterior
probability distribution $\post$ over the
parameters. Bayes' Theorem explicitly gives the posterior density as
\begin{align*}
\post \propto \likeb \times \prior
\end{align*}
which updates the prior knowledge $\prior$ using the
likelihood $\likeb$. The prior encodes any
knowledge or assumptions about the parameter space that we may wish to
impose before any data are observed, while the likelihood explicitly represents
the probability that a given set of parameters $\ub$ might give
rise to the observed data $\db$.

Even when the prior and noise probability distributions are Gaussian,
the posterior need not be Gaussian, due to possible nonlinearity embedded
in the likelihood.  For large-scale inverse problems, exploring
non-Gaussian posteriors in high dimensions to compute statistics is a
grand challenge since evaluating the posterior at each point
in the parameter space requires a solution of the parameter-to-observable map, 
including a potentially expensive forward model solve.
Using numerical quadrature to compute the mean and covariance matrix,
for example, is generally infeasible in high dimensions. 
Usually the method of choice for
computing statistics is Markov chain Monte Carlo (MCMC), which
judiciously samples the posterior distribution, so that sample
statistics can be used to approximate the exact ones.

The Metropolis-Hastings (MH) algorithm, first developed by
Metropolis \etal \cite{MetropolisRosenbluthRosenbluthEtAl53} and then
generalized by Hastings \cite{Hastings70}, is perhaps the most popular MCMC
method. Its popularity and attractiveness come from the ease of
implementation and minimal requirements on the target density and the
proposal density \cite{HaarioLaineMiraveteEtAl06, RobertCasella05}.
The problem, however, is
that standard MCMC methods often require millions of samples for
convergence; since each sample requires an evaluation of the
parameter-to-observable map, this could entail millions of expensive
forward PDE simulations\textemdash a prohibitive proposition.
On one hand, with the rapid development of parallel computing,
parallel MCMC methods \cite{Wang14, Byrd10, Wilkinson05, Brockwell06, Strid10}
are studied to accelerate the computation.
While parallelization allows MCMC algorithms to produce more samples in a
shorter time with multiple processors, such accelerations typically do not
improve the mixing and convergence of MCMC algorithms.
More sophisticated MCMC methods that exploit 
the gradient and higher derivatives of the log posterior (and hence
the parameter-to-observable map) \cite{DuaneKennedyPendletonEtAl87,
  Neal10, GirolamiCalderhead11, BeskosPinskiSanz-SernaEtAl11,
  Bui-ThanhGirolami14, CuiMartinMarzoukEtAl14, CuiLawMarzouk16,
  MartinWilcoxBursteddeEtAl12, Bui-ThanhGhattas12d,
  PetraMartinStadlerEtAl14} can, on the other hand, improve the mixing, acceptance rate,
and convergence of MCMC. Another sample-based family of approaches that
provide uncertainty quantification and are well-suited for 
parallelization on large clusters is the various forms of 
Stein variational gradient descent 
\cite{liu2016,han2018,chenghattas2020,zhuo2018}. Of related interest 
are particle filter methods such as those found in 
\cite{carpenter1999,vandermerwe2000,yang2013,soto2005} that evolve particles through a dynamical system 
over time, updating both an estimate of the state 
and uncertainty.

One approach to addressing the computational challenge in high-dimensional 
statistical inverse problems pose is to use randomization, either to 
reduce the dimension of the optimization problem used in estimating 
the maximum a posteriori (MAP) point \cite{LeEtAl2017}, or to aid in sampling from the posterior distribution 
\cite{WangBui-ThanhGhattas18}. Several methods have been proposed which utilize randomization 
to accelerate the solution of inverse problems
\cite{LeEtAl2017,WangBui-ThanhGhattas18,chen2020,avron2013,wang2017sketching,iglesias2013ensemble}. 
{\em As the main contribution of this paper}, we derive unified results 
of randomized inverse approaches that apply to a broad 
class of linear and nonlinear inverse problems
not only in the asymptotic regime,
but also for the non-asymptotic setting.
The asymptotic convergence and a non-asymptotic error bound of various existing
methods follows immediately as special cases of the general result.
\section{A unified analysis of randomized inverse problems through a sample average approximation lens}
\seclab{theory}
For the remainder of this paper, we will use lower case letters for 
scalar quantities ($\alpha$), boldface lower case letters for vectors ($\ub$) 
and boldface upper case letters for matrices ($\bs{A}$).
Further, we will use superscript lower case letters to denote 
sample index and subscript uppercase letters to denote the total 
number of samples, i.e. $\lamb^i$ is the $i^{\text{th}}$ sample
and $\ub_\N$ is a quantity depending on $\N$ samples. Lastly, 
descriptions or method names will be in uppercase superscripts, such 
as $\umap$ which is $\ub$ at the MAP point, for example. 
This should be clear from the context.

Therefore, let $\ub,
\ub_0 \in \R^\n$. 
The posterior measure $\nu$ in this case has the
density $\post$ with respect to the Lebesgue measure:
\begin{equation}
  \eqnlab{posterior}
\post \propto \likeb \times \prior,
\end{equation}
where the likelihood is given by $\likeb \propto \exp\LRp{-\Phi\LRp{\ub,\db}} = \exp\LRp{-\half\norm{\db -
    \F\LRp{\ub}}^2_{\L^{-1}}}$ and the prior by $\piprior \propto \exp\LRp{-\half \norm{\ub - \ub_0}^2_{\sC^{-1}}}$. Here, $\F\LRp{\ub}$ is known as the parameter-to-observable (PtO) map, an evaluation of which typically requires a solution of the forward model (e.g. partial differential equations) governing the underlying physics.
The {\em maximum a posteriori} (MAP) problem reads
\begin{equation}
\eqnlab{MAPfinite}
\umap := \arg\min_\ub \J\LRp{\ub; \ub_0, \db} :=\half\norm{\db -
    \F\LRp{\ub}}^2_{\L^{-1}} + \half \norm{\ub - \ub_0}^2_{\sC^{-1}},
\end{equation}
where $\sC \in \R^{\n\times \n}$ is the prior covariance matrix 
and $\L \in \R^{k \times k}$ is the noise covariance matrix.

To the end of the paper, we denote by $\Ex$ the expectation with subscript as the random variable with respect to which the expectation is taken. When the random variable is clear from the context we simply omit the subscript for brevity. 
Let $\sigb, \epsb, \del,$ and $\lam$ be finite dimensional independent 
  random variables with bounded second moments
  such that: 
  \begin{equation}
    \expect \LRs{\sigb} = 0, \quad
    \expect \LRs{\epsb \epsb^T} = \L^{-1}, \quad
    \expect \LRs{\del} = 0, \quad
    \expect \LRs{\lam \lam^T} = \sC^{-1}.
    \eqnlab{boundednessSigC}
  \end{equation}
  
  Let us define
  $ \rand = \LRs{\sigb, \epsb, \del, \lam}^T \in \randSpace$
  with joint probability distribution
  $\randP = \sigP \times \epsP \times \delP \times \lamP$. 
  Consider the following stochastic cost function:
  \begin{align}
    \eqnlab{stochastic_cost}
    \begin{split}
      \Jstoch\LRp{\ub; \ub_0, \db, \rand}:=
      &\half \norm{\epsb^T \LRp{\db + \sigb - \F\LRp{\ub}}}^2_2 + \half \norm{\lam^T \LRp{\ub - \ub_0 - \del}}^2_2\\
      =&\half \LRp{\db + \sigb - \F\LRp{\ub}}^T \epsb \epsb^T 
      \LRp{\db + \sigb - \F(\ub)} \\
      + &\half \LRp{\ub - \ub_0 - \del}^T \lam \lam^T 
      \LRp{\ub - \ub_0 - \del}.
    \end{split}   
  \end{align}
  Define 
  \begin{equation*}
    \eqnlab{expect_cost}
    \Jexpect \LRp{\ub; \ub_0, \db} := 
    \expect_\randP \LRs{\Jstoch \LRp{\ub; \ub_0, \db, \rand}},
  \end{equation*}
  and the sample average approximation (SAA) of $\Jexpect$ to be
  \begin{equation}
    \JexpectSA :=
    \frac{1}{\N} \sum_{\j=1}^\N \Jstoch \LRp{\ub; \ub_0, \db, \rand^\j}
    \eqnlab{SA_cost}
  \end{equation}
  where $\rand^\j$ are i.i.d. samples from $\randP$.
  Assume that both $\J$ and $\JexpectSA$ have a minimum and
  let us define
  \begin{equation}
    \eqnlab{MAPsols}
    \umap := \argmin_\ub \J, \text{ and }
    \umapSA := \argmin_\ub \JexpectSA.
  \end{equation}
Below we study  asymptotic and non-asymptotic convergence of $ \umapSA$ to $ \umap$.

\subsection{Asymptotic convergence analysis for inverse problems}
\seclab{asymptotic_theory}

\begin{theorem}[Asymptotic convergence of randomized nonlinear inverse problems]
  \theolab{asymptotic_convergence}
  Assume that $\F\LRp{\ub}$ is such that $\Jstoch$ is a convex, twice continuously differentiable function in $\ub$ for almost every $\rand$, and measurable\footnote{Here, measurable is with respect to the the $\sigma$-algebra given by the product $\sigma$-algebras of the deterministic $\LRp{\ub,\ub_0,\db}$ and random variables $\rand$.}. 
  Then the following hold true:
  \begin{enumerate}[i)]
  \item Minimizing $\Jexpect$ is equivalent to minimizing $\J$ in the sense: $\argmin_\ub \J = \argmin_\ub \Jexpect$.
  \item $\umapSA 
    \xrightarrow[N \to \infty]{a.s.}
    \umap$.
  \end{enumerate}
  \theolab{asymptotic}
\end{theorem}

\begin{proof}
  For the first assertion, consider only the first term of $\Jexpect$ as the second term follows analogously. We have
  \begin{align}
    \half \expect_{\randP} &\LRs{\LRp{\db + \sigb - \F(\ub)}^T \epsb \epsb^T 
                             \LRp{\db + \sigb - \F(\ub)}} \nonumber \\ 
                           &= \half \expect_{\sigP \times \epsP} \LRs{\LRp{\db + \sigb - \F(\ub)}^T \epsb \epsb^T 
                             \LRp{\db + \sigb - \F(\ub)}} \nonumber\\
                           &=\half \expect_{\sigP} \LRs{\LRp{\db + \sigb - \F(\ub)}^T \expect_{\epsP}
                             \LRs{\epsb \epsb^T} 
                             \LRp{\db + \sigb - \F(\ub)}} \nonumber \\
                           &=\half \expect_{\sigP} \LRs{\LRp{\db + \sigb - \F(\ub)}^T \L^{-1}
                             \LRp{\db + \sigb - \F(\ub)}} \nonumber \\
                           &= \half \LRp{\db  - \F(\ub)}^T \L^{-1}\LRp{\db - \F(\ub)} + 
                             \expect_{\sigP} \LRs{\sigb^T \L^{-1} \sigb}.
                             \eqnlab{expect_final}
  \end{align}
  
  The final term in \eqnref{expect_final} is constant with respect to $\ub$ 
  and can be ignored, leaving only the first term of  
  $\J$. Applying the same procedure to the second term of $\Jexpect$ shows that 
  minimizing $\Jexpect$ is equivalent to minimizing $\J$.
  
  We invoke \cite[Theorem 5.4]{ShapiroDentchevaRuszczynski09} to prove the second assertion. It is sufficient to verify the
  following conditions:

  \begin{enumerate}[(i)]
  \item $\Jstoch\LRp{\ub; \ub_0, \db, \rand}$ is random lower semicontinuous,
    
  \item  for almost every $\rand \in \randSpace$, 
    $\Jstoch\LRp{\ub; \ub_0, \db, \rand}$ is convex in $\ub$,
    
  \item 
  $\Jexpect\LRp{ \ub; \ub_0, \db}$ is lower semicontinuous in $\ub$ and there
    exists a point $\bar \ub \in \reals^\n$ such that 
    $\Jexpect \LRp{ \ub; \ub_0, \db} < \infty$ for all $\ub$ in a neighborhood 
    of $\bar \ub$; 
    
  \item the set of optimal solutions of the true problem 
    is nonempty and bounded; and
    
  \item the law of large numbers (LLN) \cite{Feller71,Durrett19} holds pointwise for $\JexpectSA$. 
  \end{enumerate}
  
  Clearly, $\Jstoch$ is a continuous function for every $\rand$, thus random lower semicontinuous as well.
  By assumption, $\Jstoch$ is also convex for almost every $\rand$. 
  Due to the boundedness assumptions \eqnref{boundednessSigC} and the fact that $\J$ is a continuous and convex function,
  $\Jexpect$ is also a continuous and convex function. 
  Furthermore, taking, for example, $\bar \ub = \ub_1$ in \eqnref{MAPsolution} it is straightforward to see that
  $\Jexpect \LRp{ \ub; \ub_0, \db} < \infty$ for any ball with finite radius centered at $\bar \ub$. The last two conditions are clear.
\end{proof}

An important special case of this theorem
occurs when we consider an inverse problem with a linear parameter-to-observable map. 
When the forward map $\F\LRp{\ub}$ is linear, the convexity and continuous differentiability assumptions are satisfied.
While requiring convexity is a strong assumption in general, this is not an insurmountable issue for regularized inverse problems.
Note that the Hessian of $\J$ is given by
\begin{equation*}
  \grad^2_\ub \J = \grad^2_\ub \F\LRp{\ub} \L^{-1} \LRp{\F\LRp{\ub} - \db} + \grad_\ub \F\LRp{\ub}^T \L^{-1} \grad_\ub \F\LRp{\ub} + \sC^{-1}. 
\end{equation*}
Thus the prior covariance matrix can be chosen such that $\grad^2_\ub \J$ is semi-positive definite. Indeed, this is the major
role that the prior covariance plays in regularizing the ill-posed inverse problem. 

Lastly, note that instead of treating all of the random variables 
$\sigb, \epsb, \lamb, $ and $\del$ as a single random variable, sampling them 
together as in \eqnref{SA_cost}, we could have chosen to approximate each random variable 
separately, applying \cite[Theorem 5.4]{ShapiroDentchevaRuszczynski09} 
four times to complete the asymptotic proof. 
\begin{equation}
    \eqnlab{independent_SA_sums}
    \tilde{\ub}^{\text{\normalfont{MAP}}}_N := \argmin_\ub \frac{1}{\N_1\N_2\N_3\N_4}\sum_{i=1}^{\N_1}
    \sum_{j=1}^{\N_2}\sum_{k=1}^{\N_3}\sum_{l=1}^{\N_4} 
    \Jsa \LRp{\ub; \ub_0, \db, \LRs{\sigb^i; \epsb^j; 
    \lamb^k; \del^l}}
\end{equation}
This flexibility in deciding how samples will be drawn will aid in both 
non-asymptotic convergence analysis and will provide great freedom in designing 
a variety of randomized methods to solve inverse problems in section 
\ref{section:numerical_results}. 

To alleviate some notational burden, let us define a few new quantities. 
\begin{subequations}
\eqnlab{SA_definitions}
\begin{align}
  \Sn &:= \meanIN \epsb^i (\epsb^i)^T, \quad \quad 
  \Ln := \meanIN \lam^i (\lam^i)^T, \\
  \sigbarN &:= \meanIN \sigb^i, \quad \quad \quad \quad 
  \delbarN := \meanIN \del^i.
\end{align}
\end{subequations}
Written in terms of norms, equation \eqnref{independent_SA_sums} becomes
\begin{align}
\eqnlab{independent_SA_loss}
\tilde{\ub}^{\text{MAP}}_N = \argmin_{\ub} \half \norm{\db + \sigbarN - \F \LRp{\ub}}_{\Sn}^2
 + \half \norm{\ub - \ub_0 - \delbarN}_{\Ln}^2.
 \end{align}
Note that \eqnref{independent_SA_loss} is equivalent to 
\eqnref{SA_cost} when at most 
one of $\epsb$ and $\sigb$ are randomized and at most one of 
$\lamb$ and $\del$ are randomized. This is because the only difference 
between the two cost functions is how samples of $\epsb$ interact
with samples of $\sigb$ and likewise, how samples of $\lamb$ interact with 
samples of $\del$.

\subsection{Non-asymptotic error analysis for nonlinear inverse problems}
\seclab{nonasymptotic_theory}
In addition to proving a general asymptotic convergence of randomized inverse problems,
it is also possible to derive a general non-asymptotic error bound with the 
slightly stronger assumption that the random variables are subgaussian
\cite{vershynin_2018}.
The general form of this non-asymptotic bound is useful in that it is easy
to identify the key components that go into forming the bound --- giving insight into the performance
of various methods by enabling easy simplification in the case that certain quantities are not randomized.
The derived bound gives a probabilistic worst-case for finite sample size $\N$ when all
of the above randomizations are implemented at the same time. By fixing some of the quantities,
the given bound can be simplified in a straightforward manner --- yielding a more insightful bound.

Additionally, we follow the standard vector norm convention of $\norm{\ub}_\infty := \max \LRp{\snor{\ub_1}, ..., \snor{\ub_n}}$ and $\norm{\ub}_1 := \sum_{i=1}^n \abs{\ub_i}$ for a vector $\ub \in \reals^n$. Matrix norms 
are understood to be induced norms \cite{trefethen1997numerical}.
Due to the equivalence of norms in finite dimensional spaces, all the results are also valid for other norms, albeit with different constants that possibly depend on the dimension. We use w.p. as the abbreviation for {\em with probability}.

\begin{proposition}[Convergence of mean-zero subgaussian random vector]
  \propolab{zero_mean}
    Let $\del^i$, for $i = 1, ..., \N$, be independent subgaussian random vectors in $\reals^n$ such that $\expect \LRs{\del^i} = \mb{0}$,
    $\expect \LRs{\del^i (\del^i)^T} = \Gamma$,
    and $\expect \LRs{(\del^i)^T\del^i} < \infty$. 
    Denote the empirical mean $\bar{\del} := \frac{1}{\N} \sum_{i=1}^{\N} \del^i$.
    Further, let
    $\zeta\LRp{\N, \tol} := \exp\LRp{-c \tol^2 \N}$
    for some $\tol > 0$ and $c$ is a constant possibly depending on the dimension $n$ but not on $\N$.
    Then 
    \begin{equation}
        \norm{\bar{\del}}_{\infty} \leq \tol \norm{\Gamma^\half}_\infty \quad \text{w.p. at least } 
        1 - \zeta \LRp{\N,\tol}.
    \end{equation}
\end{proposition}
 \begin{proof}
Define $\del^i = \Gamma^\half \taub^i$, where $\taub^i \sim \GM{0}{\ident}$. 
Thus, $\overline{\taub} = \frac{1}{N}\sum_{i=1}^N\taub^i \sim \GM{0}{\frac{\ident}{N}}$.
First from\footnote{While \cite[Theorem 1]{gao2022tail} is derived for Gaussian random matrices, it also applies to subgaussian random 
matrices because subgaussian random variables have the same bound for the expected value of 
their moment generating function (see \cite[Proposition 2.5.2]{vershynin_2018} for the details).
} \cite[Theorem 1]{gao2022tail} we have 
      \[
      \Pr{\nor{\taub}_{{\infty}} > \tol_1 } \;\; \leq \;\; \Pr{\nor{\taub}_{1} > \tol_1 } \;\; \leq \;\; \exp\LRp{-\frac{\tol_1^2}{4cn}},
      \]
      where $c$ is an absolute positive constant. Therefore, abusing notation to consolidate the constant terms, we have
    \[
      \Pr{\norm{\overline{\taub}}_\infty \; > \; \frac{\tol_1}{\sqrt{N}} } \;\; \le \;\; \exp\LRp{-c\tol_1^2}
      \]
      since $\sqrt{\N} \cdot \overline{\taub}$ has identity covariance. Next, set $\tol = \frac{\tol_1}{\sqrt{N}} $ and note that
      \[
      \Pr{\norm{\overline{\del}}_\infty \; \le \; \tol \norm{\Gamma^\half}_\infty} \;\; \ge \;\; \Pr{\norm{\overline{\taub}}_\infty \; \le \; \tol}
      \;\; \ge \;\;  1 - \exp\LRp{-c \tol^2 \N},
      \]
      and this concludes the proof.
\end{proof}

\begin{proposition}[Convergence of subgaussian random vector outer product]
  \propolab{outer_product}
    Let $\omb^i$ be a subgaussian random vector in $\reals^n$ such that 
    $\expect \LRs{\omb^i (\omb^i)^T} = \Gamma$ for $i = 1, ..., N$. 
    Define $\Omega_\N$ to be the random matrix formed by stacking $\omega^i$
    in the columns and scaling by $1/\sqrt{\N}$. 
    It follows that 
    \begin{equation}
        \norm{\Omega_\N \Omega_\N^T - \Gamma}_\infty \leq \tol \norm{\Gamma}_\infty
        \text{ w.p. at least } 1 - 2 \zeta\LRp{\N, \tol},
    \end{equation}  
    where $c$ is a constant depending on $n$ but not on $\N$.
\end{proposition}
\begin{proof}
    This result follows from straightforward algebraic manipulation of 
    \cite[Theorem 4.6.1]{vershynin_2018}.
\end{proof}

An additional fact needed to prove a non-asymptotic bound is that the product of 
three subgaussian random variables is $\alpha$-subexponential  with $\alpha = 2/3$. The following discussion on $\alpha-$subexponential random variables is based on \cite{sambale2020some}. 
For a more compete treatment of the topic, \cite{sambale2020some} can be consulted. 

It has been established that the product of two subgaussian random variables is subexponential
\cite[Lemma 2.7.7]{vershynin_2018}. 
A centered random variable $X$ is said to be  $\alpha$-subexponential 
(or sub-Weibull \cite{vladimirova2020sub,zhang2022sharper})
if it satisfies  
\begin{equation*}
    \prob[\abs{X} \geq \tol] \leq 2 \exp \LRp{-c\tol^{\alpha}}, 
\end{equation*}
for any $\tol > 0$, $\alpha \in (0, 2]$, and some positive constant $c$.
To show that a random variable satisfies this condition, it is sufficient to show 
that the following Orlicz (quasi-) norm \cite{sambale2020some,vershynin_2018} is finite: 
\begin{equation}
    \eqnlab{alpha_subexponential_prob}
    \norm{X}_{\psi_{\alpha}} := \inf \LRs{\tol > 0 : \expect \exp \LRp{\LRp{\abs{X} / \tol}^\alpha} \leq 2} < \infty.
\end{equation}
When $\alpha < 1$, this is a quasi-norm since it does not satisfy the triangle inequality. 

\begin{proposition}[$\alpha$-subexponential from  product of three subgaussian random variables]
    \propolab{three_subgaussians}
    Let $X_1, X_2, X_3$ be subgaussian random variables. 
    Then $Y = X_1 X_2 X_3$ is an $\alpha$-subexponential random variable with $\alpha = 2/3$.
\end{proposition}
\begin{proof}
    It suffices to show that $\expect \exp\LRp{\LRp{\abs{Y}}^{2/3}}\leq 2$.
    Without loss of generality, assume $\norm{X_i}_{\psi_2} = 1$. Then $\expect \exp\LRp{X_i^2} \leq 2$ and we have
    \begin{align*}
        \expect \exp\LRp{\LRp{\abs{Y}}^{2/3}} &= \expect \exp\LRp{\LRp{\abs{X_1}^{2/3}\abs{X_2}^{2/3}\abs{X_3}^{2/3}}}\\
        &\leq \expect \exp\LRp{ 
            \frac{{X_1}^{2}}{3} +
            \frac{{X_2}^{2}}{3} +
            \frac{{X_3}^{2}}{3}
        }  \quad \LRp{\text{Young's inequality for 3 variables \cite{Youngs3Variables}}}\\
        &= \expect 
            \exp\LRp{\frac{{X_1}^{2}}{3}}
            \exp\LRp{\frac{{X_2}^{2}}{3}}
            \exp\LRp{\frac{{X_3}^{2}}{3}} \\
        &\leq \expect \frac{1}{3}\LRs{
            \exp\LRp{{X_1}^{2}} +
            \exp\LRp{{X_2}^{2}} +
            \exp\LRp{{X_3}^{2}}
        }  \quad \LRp{\text{Young's inequality again}}\\
        &\leq 2.
    \end{align*}
\end{proof}

\begin{corollary}
    \corolab{alpha_subexponential_error}
    Let $\del \in \R^n$ be a zero-mean random vector such that $\expect[\del \del^T] = \sC$ 
    with  $\alpha$-subexponential entries. Define 
    $\zeta_\alpha\LRp{N, \tol} := \exp\LRp{-c\tol^\alpha N^{\alpha / 2}}$.
    Then 
    \begin{equation}
        \prob \LRs{\norm{\frac{1}{N}\sum_{i=1}^n \del^i}_\infty \leq \tol \norm{\sC^{\half}}_\infty}  
        \geq 1 - 2n\zeta_\alpha \LRp{N, \tol}.
    \label{eq:subExpBound}
    \end{equation}
\end{corollary}

Note that while inequality \eqref{eq:subExpBound} has the dimension $n$ in front of the exponential, 
this is fixed and the probability of committing an error greater than some tolerance 
still decreases exponentially in the number of 
samples $N$. Though the decaying is at the slower rate $\propto\exp\LRp{-N^{\alpha / 2}}$ compared to 
the subgaussian error rate $\propto\exp\LRp{-N}$, it is not surprising because 
subgaussian is a special case of $\alpha$-subexponential when $\alpha = 2$ \cite{sambale2020some}.
These results can be combined to derive a probabilistic non-asymptotic error bound for
nonlinear inverse problems.

\begin{lemma}[Non-asymptotic error analysis for randomized nonlinear inverse problems]
\lemlab{nonasymptotic_bounds}
Let $\vect\LRp{\L^{-1}}$ denote a vectorization of  a matrix $\L^{-1}$. Define 
\[ 
\P := \LRs{\vect\LRp{\L^{-1}};\; \vect\LRp{\sC^{-1}};\; \eb ;\; \zb}
\]
 as a vector concatenating all  four vectors $\vect\LRp{\L^{-1}}, \vect\LRp{\sC^{-1}}, \eb$ and $ \zb$, where $\L \in \R^{k\times k}, \sC \in \R^{n\times n}, \eb \in \R^k$, and $\zb\in \R^n$. 
Define the function 
\begin{equation*}
   g\LRp{\P;\ub} 
    := \grad_\ub \F\LRp{\ub} \LRs{\L^{-1} \LRp{\F \LRp{\ub} - \db} - \eb}
    + \sC^{-1} \LRp{\ub - \ub_0} - \zb.
\end{equation*}
Assume that the problem $g\LRp{\P;\ub}  = 0$, with $\P$ as parameters and $\ub$ as solution, is Lipschitz well-posed \cite{latz2020well}
with Lipschitz constant $\lip$, and we define $\mathcal{G}\LRp{\P}$ as the solution $\ub$. 
Let
\begin{align*}
\P^{\text{\normalfont{MAP}}} \,\; &:= \LRs{\vect\LRp{\L^{-1}}; \;\vect\LRp{\sC^{-1}};\; 0;\; 0},\\
\tilde\P_{N} \quad\; &:= \LRs{\vect\LRp{\Sn};\; \vect\LRp{\Ln};\; 
\Sn\sigbarN;\; \Ln\delbarN},\\
\hat\P_{N}\quad\; &:= \LRs{\vect\LRp{\Sn};\; \vect\LRp{\Ln};\; 
\frac{1}{N}\sum_{i=1}^N \epsb^i \LRp{\epsb^i}^T \sigb^i;\;
\frac{1}{N}\sum_{i=1}^N \lamb^i \LRp{\lamb^i}^T \del^i},
\end{align*}
where $\expect[\sigb \sigb^T] = \L$ and $\expect[\del \del^T] = \sC$.
Then
\begin{align}
   \eqnlab{independent_SA_nonasymptotic}
   \norm{\umap - \tilde{\ub}^{\text{\normalfont{MAP}}}_N}_\infty &\leq \tol \lip 
   \LRp{\norm{\L^{-1}}_\infty + \norm{\sC^{-1}}_\infty 
   + \LRp{1 + \tol}\norm{\L^{-\half}}_\infty +
   \LRp{1 + \tol}\norm{\sC^{-\half}}_\infty} \\
          &\text{ w.p. at least  } 1 - 10 \zeta\LRp{\N, \tol}, \nonumber
\end{align}
    and 
\begin{align}
    \eqnlab{SA_nonasymptotic}
    \norm{\umap - \umapSA}_\infty &\leq \tol \lip \LRp{\norm{\L^{-1}}_\infty + \norm{\sC^{-1}}_\infty +            \norm{\L^{-\half}}_\infty +
     \norm{\sC^{-\half}}_\infty}  \\
      &\text{ w.p. at least  } 1 - 4 \zeta\LRp{\N, \tol} \nonumber - 2k \zeta_{2/3} \LRp{\N, \tol} 
     - 2n \zeta_{2/3} \LRp{\N, \tol}. \nonumber
\end{align}
\end{lemma}
\begin{proof}
    Noting that $\umap = \mathcal{G}\LRp{\P^{\text{\normalfont{MAP}}}}$ and $\tilde{\ub} = \mathcal{G}\LRp{\tilde \P_\N}$,
    we have by the Lipschitz well-posedness assumption,
    \begin{align*}
        &\norm{\umap - \tilde{\ub}^{\text{\normalfont{MAP}}}_N}_\infty 
         = \norm{\mathcal{G}\LRp{\P^{\text{\normalfont{MAP}}}} - \mathcal{G}\LRp{\tilde \P_\N}}_\infty \leq \lip \norm{\P - \tilde \P_\N}_\infty \\
                            &\leq \lip \big (\norm{\L^{-1} - \Sn}_\infty + \norm{\sC^{-1} - \Ln}_\infty  + \norm{\Sn \sigbarN}_\infty + 
                            \norm{\Ln \delbarN}_\infty\big ). 
    \end{align*}
    We can bound  $\Sn\sigbarN$ (and similarly for $\Ln\delbarN$) as follows
    \begin{align*}
        \norm{\Sn \sigbarN}_\infty &= \norm{\L^{-\half}\LRp{\L^{\half}\Sn \L^{\half}- \ident + \ident}\L^{-\half} \sigbarN}_{\infty}\\
        &\leq \norm{\L^{-\half}\LRp{\L^{\half}\Sn\L^{\half} - \ident}}_\infty 
        \norm{\L^{-\half}\sigbarN}_\infty + 
        \norm{\L^{-\half}}_{\infty}\norm{\L^{-\half}\sigbarN}_\infty.
    \end{align*}
    Note that $\L^{-\half} \sigbarN$ is the sample average of a mean-zero subgaussian random variable 
    with identity covariance and  $\expect[\L^{\half}\Sn \L^{\half}] = \ident$. 
    Therefore, applying Proposition \proporef{outer_product} to 
    $\norm{\L^{\half}\Sn \L^{\half} - \ident}_\infty$ and Proposition 
    \proporef{zero_mean} to $\norm{\L^{-\half}\sigbarN}_\infty$ along with the union bound, we obtain
    \begin{equation*}
        \norm{\Sn \sigbarN}_\infty \leq \LRp{\tol^2 + \tol} \norm{\L^{-\half}}_\infty  
        \quad \text{w.p. at least } 1 - 3\zeta \LRp{\N, \tol}.
    \end{equation*}
    The proof of the bound $\norm{\Ln\delbarN}_\infty \le \LRp{\tol^2 + \tol} \norm{\sC^{-\half}}_\infty$ 
    $\text{w.p. at least } 1 - 3\zeta \LRp{\N, \tol}$ follows analogously.
    Applying Proposition \proporef{outer_product} to the terms $\norm{\Sn - \L^{-1}}_{\infty}$ and   
    $\norm{\Ln - \sC^{-1}}_\infty$ along with the union bound, inequality 
    \eqnref{independent_SA_nonasymptotic} follows immediately.
      
    To prove the bound in \eqnref{SA_nonasymptotic} for $\norm{\umap - \umapSA}_{\infty}$, it is sufficient 
    to bound $\norm{\eb}_\infty$  and $\norm{\zb}_\infty$ according to the definition of $\hat\P_N$ 
    since the bounds for $\norm{\Sn - \L^{-1}}_{\infty}$ and   
    $\norm{\Ln - \sC^{-1}}_\infty$ have already been proven.
    By Proposition \proporef{three_subgaussians}, $\epsb^i \LRp{\epsb^i}^T \sigb^i$ is a zero-mean,
    $2/3$-subexponential random variable with covariance $\L^{-1}$. 
    Then by Corollary \cororef{alpha_subexponential_error}, 
    \begin{equation*}
        \prob \LRs{\norm{\eb}_\infty \leq \tol \norm{\L^{-\half}}_\infty} \geq 1 - 2k\zeta_{2/3} \LRp{N, \tol}.
    \end{equation*}
    Similarly, applying Corollary \cororef{alpha_subexponential_error} to $\norm{\zb}_\infty$ yields 
    \begin{equation*}
        \prob \LRs{\norm{\zb}_\infty \leq \tol \norm{\sC^{-\half}}_\infty} \geq 1 - 2n\zeta_{2/3} \LRp{N, \tol}.
    \end{equation*}
    Using the union bound to combine the bounds on $\norm{\eb}_\infty$  and $\norm{\zb}_\infty$ with  
    the bounds previously $\norm{\Sn - \L^{-1}}_\infty$ and 
    $\norm{\Ln - \sC^{-1}}_\infty$ , inequality \eqnref{SA_nonasymptotic} follows. 
\end{proof}

Lemma \lemref{nonasymptotic_bounds} provides a general strategy for analyzing the convergence of various randomized methods. When a variable is not randomized, we simply
drop the corresponding terms in equations \eqnref{independent_SA_nonasymptotic} and \eqnref{SA_nonasymptotic}, and adjust the probability accordingly.
Additionally, as will be discussed later, the Lipschitz constant may be affected by the choice of randomization strategy. 
By choosing to randomize only some variables, the inverse solution may have lower worst-case sensitivity (Lipschitz constant).
Note that since we are mainly concerned with small deviations, 
it is sufficient to consider well-posed inverse problems where the 
solution depends continuously on $\db$ everywhere and 
depends on $\P$ in a locally Lipschitz manner in a neighborhood of 
$\LRs{\vect\LRp{\L^{-1}}; \;\vect\LRp{\sC^{-1}};\; 0;\; 0}$.
Note that it is unlikely for the solution to depend continuously on $\P$ everywhere since $\sC^{-1}$ 
acts as regularization. Otherwise, the original problem would not be ill-posed.
In particular, the regularizing role that $\sC^{-1}$ plays may cause the solution of the inverse problem to be
especially sensitive to perturbations of $\sC^{-1}$, greatly increasing the (local) Lipschitz constant 
compared to the case where $\sC^{-1}$ is not randomized. 
This is clear from the linear case where 
the condition number of the problem takes the place of the Lipschitz constant. It is well-known that 
the choice of $\sC^{-1}$ has a significant impact on the condition number \cite{diao2016structured,chu2011condition}.
Lastly, note that in the linear case, standard perturbation theory for linear systems can be used to find an explicit
bound for the relative error in terms of the condition number of $\A^T \L^{-1}\A + \sC^{-1}$.

\section{Rediscovery of randomized inverse methods}
In this section we derive several different known randomization 
schemes as special cases of the more general randomization 
scheme proposed in \eqnref{stochastic_cost}. For the rest of this section,
we will explicitly write each of the random variables (subset of $\rand = \LRs{\sigb, \epsb, \del, \lam}^T$) that the stochastic
cost function depends on. To keep notation clean, we will use 
$\Jstoch$ to represent the stochastic cost function for all randomization 
schemes, where the form of the cost function will be clear from the context. 
If $\sigb$ does not appear as an argument of $\Jstoch$, we replace it with ${0}$ in \eqnref{stochastic_cost}. Likewise, we replace $\epsb\epsb^T$ with $\L^{-1}$, $\del$ with $0$, and $\lamb\lamb^T$ with $\sC^{-1}$ 
in \eqnref{stochastic_cost} when the corresponding random variable does not appear as an argument of $\Jstoch$.
Additionally, we will present each method in the general nonlinear setting, but we will also explicitly
write the sample average solution in the linear case as closed form solutions are available, yielding additional insights.
To that end, we write two equivalent 
formulations of the MAP estimate for linear inverse problems.

\begin{proposition}
  \propolab{MAPsolution}
  When the PtO map is linear, i.e. $\F(\ub) = \A \ub$, the solution of the MAP problem \eqnref{MAPfinite} is given by either
  \begin{subequations}
    \eqnlab{MAPsolution}
    \begin{align}
      \eqnlab{MAPsolutiona}
  \ub_1 &= \LRp{\A^T\L^{-1}\A + \sC^{-1}}^{-1}\LRp{\A^T\L^{-1}\db + \sC^{-1}\ub_0}
  \intertext{or}
        \eqnlab{MAPsolutionb}
   \ub_2 &= \ub_0 +\sC\A^T\LRp{\L + \A\sC\A^T}^{-1}\LRp{\db -\A\ub_0}.
   \end{align}
   \end{subequations}
\end{proposition}

\begin{proof}
    The first of these identities is derived directly from the optimality
    condition of \eqnref{MAPfinite}. Specifically, 
    \begin{equation}
        \grad \J  = \LRp{\A^T \L^{-1} \A + \sC^{-1}}  \ub - 
        \A^T \L^{-1} \db - \sC^{-1} \ub_0 = 0.
    \end{equation}
    The second formulation can be derived from $\ub_1$ 
    using the Sherman-Morrison-Woodbury formula  \cite{deng2011generalization}
    under the conditions that $\sC^{-1}$ and $\L^{-1}$ are invertible. 
\end{proof}

\begin{remark}
    This last assumption concerning the invertibility 
    of $\sC^{-1}$, while seemingly trivial in light 
    of the fact that $\sC^{-1}$ is written as the inverse of a matrix, will be important 
    in the following discussion of randomization. 
\end{remark}

Additionally, let us define new random variables:
\begin{equation}
  \eqnlab{rand_means}
  \db^i := \db + \sigb^i \quad \text{and} \quad  \ub_0^i := \ub_0 + \del^i,
\end{equation}
where $\sigb^i$ and $\del^i$ are the first and the third components of $\rand^i$ defined in Theorem \theoref{asymptotic}.
These quantities will be useful in the following discussion. 
By the LLN we have
\[
  \frac{1}{\N} \sum_{i=1}^\N \db^i 
  \xrightarrow[N \to \infty]{a.s.}
  \db \quad
  \text{and} \quad \frac{1}{\N} \sum_{i=1}^\N \ub_0^i \xrightarrow[N \to \infty]{a.s.}
  \ub_0.
\]

\subsection{Randomized MAP approach}
Assuming that the order of minimization and expectation can be
interchanged\footnote{The conditions under which the interchange is
valid can be consulted in \cite[Theorem 14.60]{RockafellarWetts98}.},
we can write
\begin{equation}
\eqnlab{SP}
 \argmin_\ub\Ex_{\sigP \times \delP}\LRs{\Jstoch\LRp{\ub;  \ub_0, \db, \sigb, \del}} = 
\Ex_{\sigP \times \delP}\LRs{\argmin_\ub \Jstoch\LRp{\ub; \ub_0, \db, \sigb, \del}}.
\end{equation}
The sample average approximation of the RHS can then be written 
\begin{equation}
  \eqnlab{rmap_saa}
  \urmap_\N := \frac{1}{\N}\sum_{i=1}^\N \argmin_\ub 
  \Jstoch \LRp{\ub; \ub_0, \db, \sigb^i, \del^i}.
\end{equation}
This randomization approach coincides with the
randomized MAP approach \cite{WangBui-ThanhGhattas18} when $\expect_{\sigP} \LRs{\sigb \sigb^T} = \L$ and 
$\expect_{\delP} \LRs{\del \del^T} = \sC$
(also known as the randomized maximum likelihood
\cite{Kitanidis95, OliverReynoldsLiu08, BardsleySolonenHaarioEtAl13}).
In the linear case, we can write
\[
\urmap_\N = \frac{1}{\N}\sum_{i=1}^\N \LRp{\ub^{\text{RMAP}}}^i, 
\]
where, thanks to Proposition \proporef{MAPsolution},
\begin{subequations}
\begin{align*}
  \eqnlab{rMAPsampleN}
  \LRp{\urmap}^i &= \LRp{\A^T\L^{-1}\A + \sC^{-1}}^{-1}
                        \LRs{\A^T\L^{-1}\LRp{\db + \sigb^i} +
                        \sC^{-1}\LRp{\ub_0+\del^i}} \\
                      &=\LRp{\A^T\L^{-1}\A + \sC^{-1}}^{-1}
                        \LRs{\A^T\L^{-1}\db^i +
                        \sC^{-1}\ub_0^i}.
\end{align*}
\end{subequations}
Since the sample average approximation \eqnref{rmap_saa} of the
right hand side of \eqnref{SP}
converges to its expectation, the analysis from Section \secref{theory}
applies. 
We have the following result for nonlinear PtO map.
\begin{corollary}[Asymptotic convergence of RMAP]
  \corolab{rMAP}
  Suppose that the nonlinear PtO map $\F$ satisfies the assumptions of Theorem \theoref{asymptotic}. Let $\sigb^i \sim \sigP$ and $\del^i \sim \delP$ with bounded covariances where
  $\expect_{\sigP} \LRs{\sigb} = 0$ and $\expect_{\delP} \LRs{\del} = 0$. 
  Then
  \begin{equation*}
    \urmap_\N \asconv \umap \quad \text{as } \N \to \infty.
  \end{equation*}
\end{corollary}
  
Note that if \textit{only} the MAP point $\umap$ is needed, then the randomized MAP approach is not useful: in fact very expensive while only giving an approximate solution for $\umap$. However, the approach could be appealing for Bayesian settings. Indeed,
by choosing $\expect_{\sigP} \LRs{\sigb \sigb^T} = \L$ and 
$\expect_{\delP} \LRs{\del \del^T} = \sC$, each solution $\LRp{\urmap}^i$ is a bona fide sample of the posterior
distribution in the linear case. For nonlinear cases, $\LRp{\urmap}^i$ are biased samples of 
the posterior \cite{WangBui-ThanhGhattas18}, but can be corrected via Metropolization \cite{bui2016fem,chen2020fast}.
Note that for linear inverse
problems, the RMAP approach is the same as the randomize-then-optimize
approach in \cite{BardsleySolonenHaarioEtAl13} (see an explanation
from \cite{WangBui-ThanhGhattas18}).
This randomized MAP method is embarrassingly parallel and is well-suited
for implementation on distributed computing systems.
While we could randomize the data and prior mean without exchanging
expectation and optimization and convergence would be maintained,
such a method would be of little use because we would obtain an inaccurate approximation of $\umap$ while having the same cost.

\subsection{Randomized misfit approach (left sketching)}
In this section we show that the randomized misfit approach (RMA) \cite{LeEtAl2017} is a special case of our randomization in \eqnref{stochastic_cost}.
Indeed, if we let $\epsb \sim \epsP$ where $\expect_{\epsP} \LRs{\epsb} = 0$ and $\expect_{\epsP} \LRs{\epsb \epsb^T} = \L^{-1}$, then
\begin{subequations}
    \begin{align*}
    \urma :&= \argmin_{\ub} 
    \expect_{\epsP} \LRs{\Jstoch \LRp{\ub; \ub_0, \db, \epsb}} \\
    &= \argmin_\ub \expect_{\epsP} \LRs{\half \norm{\epsb^T \LRp{\db - \F(\ub)}}^2_2
    + \half \norm{\ub - \ub_0}^2_{\sC^{-1}}}. 
    \end{align*}
\end{subequations}
The SAA of $\urma$ can be written as
\begin{subequations}
  \eqnlab{RMA_saa}
  \begin{align}
    \urma_\N :&= \argmin_{\ub} \frac{1}{\N} \sum_{i=1}^\N 
                   \Jstoch \LRp{\ub; \ub_0, \db, \epsb^i} \\
                 &= \argmin_{\ub} \frac{1}{\N} \sum_{i=1}^\N 
                    \half \norm{\tilde{\db^i} - \tilde{\F^i}(\ub)}^2_2 + \half \norm{\ub - \ub_0}^2_{\sC^{-1}},
  \end{align}
\end{subequations}
where
\begin{equation*}
  \tilde{\F^i} := {\eps^i}^T \F \quad \text{and} \quad
  \tilde{\db^i} := {\eps^i}^T \db.
\end{equation*}
That is, the random samples, which can be combined into a random matrix, sketch the PtO map and the data from the left. 
Random sketching has been used extensively to reduce the cost of solving inverse problems \cite{chen2020,Liu2018,ClarksonWoodruff2017}. The following is a direct consequence of Theorem \theoref{asymptotic_convergence}.
\begin{corollary}[Asymptotic convergence of RMA]
    Let $\urma_\N$ be as defined above. Then 
    \begin{equation*}
        \urma_\N \asconv \ub^{MAP} \quad \text{as } \N \to \infty.
    \end{equation*}
\end{corollary}
Calculating the optimality condition in the linear case results in
\begin{equation}
  \eqnlab{ls_opt}
  \urma =  \LRp{\A^T \expect_{\epsP} \LRs{\epsb \epsb^T} \A +
      \sC^{-1}}^{-1} \LRp{\A^T \expect_{\epsP} \LRs{\epsb \epsb^T} \db + \sC^{-1} \ub_0}.
\end{equation}
By letting
\begin{equation*}
  \tilde{\A} := \eps^T \A \quad \text{and} \quad
  \tilde{\db} := \eps^T \db,
\end{equation*}
we can rewrite \eqnref{ls_opt} as
\begin{equation*}
  \urma =  \LRp{\expect_{\epsP} \LRs{\tilde{\A}^T \tilde{\A}} +
      \sC^{-1}}^{-1} \LRp{\expect_{\epsP} \LRs{\tilde{\A}^T \tilde{\db}} + \sC^{-1} \ub_0}
\end{equation*}

If we combine the RMA and randomized MAP approaches 
into a single stochastic optimization problem, we discover a new method which we will denote RMA+RMAP.
Specifically, consider the problem directly arising from randomization of \eqnref{stochastic_cost} and 
define the solution using the RMA+RMAP method to be

\begin{subequations}
\begin{align*}
  \urmarmap :&= \argmin_\ub \expect_{\sigP \times \epsP \times \delP}
               \LRs{\Jstoch \LRp{\ub; \ub_0, \db, \sigb, \epsb, \del}}\\
             &= \argmin_\ub \expect_{\sigP \times \epsP \times \delP}
               \LRs{\half \norm{\epsb^T \LRp{\db + \sigb - \F(\ub)}}^2_2
               + \half \norm{\ub - \ub_0 - \del}^2_{\sC^{-1}}},
\end{align*}
\end{subequations}
and the corresponding SAA solution
\begin{subequations}
  \begin{align*}
  \urmarmap_\N :&= \argmin_\ub \frac{1}{\N} \sum_{i=1}^\N
                  \Jstoch \LRp{\ub; \ub_0, \db, \sigb^i, \epsb^i, \del^i} \\
                &=\argmin_\ub \frac{1}{\N} \sum_{i=1}^\N
                  \LRs{\half \norm{{\epsb^i}^T \LRp{\db + \sigb^i - \F(\ub)}}^2_2
                  + \half \norm{\ub - \ub_0 - \del^i}^2_{\sC^{-1}}}.
\end{align*}
\end{subequations}
The following result is immediate from Theorem \theoref{asymptotic_convergence}.
\begin{corollary}[Asymptotic convergence of RMA+RMAP]
  Let $\urmarmap_\N$ be as defined above. Then 
  \begin{equation*}
    \urmarmap_\N \asconv \umap \quad \text{as } \N \to \infty.
  \end{equation*}
\end{corollary}

Allowing for the interchange of optimization and expectation as in the RMAP approach 
along with independent sample average approximations of each random variable,
the following variant sequences also converge.
\begin{equation}
  \eqnlab{rma_rmap_opt_1}
  \ub^{\text{\normalfont{RMA+RMAP$_1$}}}_\N =
  \frac{1}{\N} \sum_{i=1}^{\N} \argmin_{\ub}
  \LRs{\half \norm{{\epsb^i}^T \LRp{\db + \sigb^i - \F(\ub)}}^2_2
                  + \half \norm{\ub - \ub_0 - \del^i}^2_{\sC^{-1}}}.
\end{equation}
\begin{equation}
  \eqnlab{rma_rmap_opt_2}
  \ub^{\text{\normalfont{RMA+RMAP$_2$}}}_{\N,M} = \frac{1}{\M} \sum_{i=1}^{\M} \argmin_{\ub}
  \frac{1}{\N} \sum_{j=1}^N \LRs{\half \norm{{\epsb^j}^T \LRp{\db + \sigb^i - \F(\ub)}}^2_2
                  + \half \norm{\ub - \ub_0 - \del^i}^2_{\sC^{-1}}}.
\end{equation}
We would like to point out \eqnref{rma_rmap_opt_2} is perhaps the most intuitive 
way to combine RMA and RMAP. Randomization of the
noise covariance matrix acts as a random projection (left sketching) while the 
randomized prior mean and data aid in sampling from the posterior. Note that \eqnref{rma_rmap_opt_2} arises as a variant of the loss function defined in equation \eqnref{independent_SA_loss} by
exchanging the optimization and expectation of only $\sigb$ and $\del$.
On the other hand, \eqnref{rma_rmap_opt_1} would likely yield inaccurate results as 
it is the sum of solutions where the PtO map and data have been projected 
onto a one dimensional subspace: thus the prior dominates each solution. 
In the linear case, $\ub^{\text{\normalfont{RMA+RMAP}}_2}_{\N,M}$ can be written as
\begin{align*}
  \ub^{\text{\normalfont{RMA+RMAP}}_2}_{\N,M} = 
  \frac{1}{M} \sum_{i=1}^{M}
  \LRp{\frac{1}{N} \sum_{j=1}^N \LRp{\tilde{\A}^j}^T \tilde{\A}^j  + \sC^{-1}}^{-1} 
  \LRp{\frac{1}{\N} \sum_{j=1}^{\N}\LRp{\tilde{\A}^j}^T \LRp{\epsb^j}^T \db^i + \sC^{-1}\ub_0^i}.
\end{align*}

Clearly we also have convergence to the MAP point of other combinations, 
such as randomizing only one of the data or prior mean.
To avoid a combinatorial explosion in the number of corollaries, 
we omit all the possibilities here. 

\subsection{Randomized prior}
Here we propose a randomization scheme based on randomizing $\sCinv$ though $\lamb \sim \lamP$ 
where $\expect_{\lamP} \LRs{\lamb} = 0$ and $\expect_{\lamP} \LRs{\lamb \lamb^T} = \sCinv$. Let
\begin{align}
    \eqnlab{randomized_prior}
    \ursuOne :&= \argmin_{\ub} 
    \expect_{\lamP} \LRs{\Jstoch \LRp{\ub; \ub_0, \db, \lamb}} \\
    &= \argmin_\ub \expect_{\lamP} \LRs{\half \norm{\db - \F(\ub)}^2_{\Linv}
    + \half \norm{\lamb^T\LRp{\ub - \ub_0}}^2_2}. \nonumber
\end{align}
The reason for designating this method ``RS\_U1'' is due to its relationship with the right sketching approach (see Section \secref{EnKF}).
Then the SAA reads
\begin{align}
  \eqnlab{RS_U1}
    \ursuOne_\N :&= \argmin_{\ub} \frac{1}{\N} \sum_{i=1}^\N 
                        \Jstoch \LRp{\ub; \ub_0, \db, \epsb^i} \\
                     &= \argmin_{\ub} \frac{1}{\N} \sum_{i=1}^\N 
                         \half \norm{\db - \F(\ub)}^2_{\Linv} + 
                         \half \norm{\LRp{\lamb^i}^T\LRp{\ub - \ub_0}}^2_2. \nonumber
\end{align}

\begin{corollary}[Asymptotic convergence of randomized prior]
  Let $\ursuOne_\N$ be as defined above. Then 
  \begin{equation*}
    \ursuOne_\N \asconv \umap \quad \text{as } \N \to \infty.
  \end{equation*}
\end{corollary}
In the linear case, the optimality condition yields the following solution
\begin{equation*}
    \ursuOne = \LRp{\A^T \Linv \A + \expect_{\lamP} \LRs{\lamb \lamb^T}}^{-1}
    \LRp{\A^T \Linv \db + \expect_{\lamP} \LRs{\lamb \lamb^T} \ub_0}.
\end{equation*}
This approach can be thought of as sketching the prior from the left. 

\section{Optimize, transform, then randomize}
An important observation about 
the methods discussed so far
is that the linear settings are
solved using $\ub_1$ given in \eqnref{MAPsolutiona}.
That is, to show the equivalence of the solution of the randomized cost function
to the solution of the corresponding method in the literature,
one only needs to consider the optimal solution
\eqnref{MAPsolutiona}.
Additionally, the sample average approximation of the cost function is 
exactly the same as replacing the expectations in form $\ub_1$ with their 
respective sample average approximations. 
The next methods require form $\ub_2$ in \eqnref{MAPsolutionb}
to be used to see the equivalence of the randomized solution
and the corresponding method given in the literature -- where the Sherman-Morrison-Woodbury
formula is applied to the optimality condition before making sample average approximations.
{\em As $\ub_2$ is only equivalent to $\ub_1$ in the linear case, we
will restrict the following discussion to linear inverse problems. }

\subsection{Right sketching and the Ensemble Kalman filter}
\seclab{EnKF}
Since we are now considering schemes derived from randomizing $\ub_2$, 
we introduce a new random variable, $\omb$, defined such that
$\expect_{\omP} \LRs{\omb} = 0$ and $\expect_{\omP}\LRs{\omb\omb^T} = \sC$. 
By taking advantage of the asymptotic convergence of the SAA of
$\expect_{\omP}\LRs{\omb\omb^T}$, we have
\begin{subequations}
  \eqnlab{CALLN}
  \begin{align}
    \frac{1}{\N}\sum_{i=1}^\N \omb^i\LRp{\A\omb^i}^T
    & 
    \xrightarrow[N \to \infty]{a.s.}
      \Ex_{\omb}\LRs{\omb\omb^T\A^T} = \sC\A^T, \\
    \frac{1}{\N}\sum_{i=1}^\N \LRp{\A\omb^i}\LRp{\A\omb^i}^T
    & 
    \xrightarrow[N \to \infty]{a.s.}
      \Ex_{\omb}\LRs{\A\omb\omb^T\A^T} = \A\sC\A^T.
  \end{align}
\end{subequations}
Combining \eqnref{MAPsolutionb} and \eqnref{CALLN} gives
\begin{equation}
  \eqnlab{rightSketch}
  \urs_\N := \ub_0 +\LRp{\frac{1}{\N}\sum_{i=1}^\N
    \omb^i\LRp{\A\omb^i}^T}\LRp{\L + \frac{1}{\N}\sum_{i=1}^\N
    \LRp{\A\omb^i}\LRp{\A\omb^i}^T}^{-1}\LRp{\db -\A\ub_0},
\end{equation}
which is the same as sketching the PtO map $\A$ from the
right or sketching the transpose of the PtO map from the left.

\begin{lemma}[Asymptotic convergence of right sketching]
  \lemlab{RS}
  Let $\urs_\N$ be defined in \eqnref{rightSketch} and assume that
  $\expect_{\lamP} \LRs{\lam \lam^T}$ (where $\lam$ is defined in 
  Section \secref{theory}) is invertible.  Then 
  \begin{equation*}
    \urs_\N \asconv \umap \quad \text{as } \N \to \infty.
  \end{equation*}
\end{lemma}
\begin{proof}
  Beginning with equation \eqnref{expect_cost} and randomizing
  only $\sC^{-1}$ through  $\lam$, the optimality condition is 
  \begin{equation*}
    \ub^* = \LRp{\A \L^{-1} \A + \expect_{\lamP} \LRs{\lam \lam^T}}^{-1}
    \LRp{\A^T \L^{-1} \db +  \expect_{\lamP} \LRs{\lam \lam^T}\ub_0}. 
  \end{equation*}
  Since $\expect_{\lamP} \LRs{\lam \lam^T}$ is assumed to be invertible, this can
  be rewritten using the Sherman-Morrison-Woodbury formula in the form
  $\ub_2$ as
  \begin{align*}
    \ub^* &= \ub_0 + \LRp{\expect_{\lamP} \LRs{\lam \lam^T}}^{-1}
            \A^T\LRp{\L + \A \LRp{\expect_{\lamP} \LRs{\lam \lam^T}}^{-1}
            \A^T}^{-1}\LRp{\db -\A\ub_0}.
  \end{align*}
  Before making a sample average approximation, note that
  \begin{equation*}
    \LRp{\expect_{\lamP} \LRs{\lam \lam^T}}^{-1} = \LRp{\sC^{-1}}^{-1}
    = \sC = \expect_{\omP} \LRs{\omb \omb^T}.
  \end{equation*}
  Then,
  \begin{align*}
    \ub^* = \urs = \ub_0 + \expect_{\omP} \LRs{\omb \omb^T}
            \A^T\LRp{\L + \A \expect_{\omP} \LRs{\omb \omb^T}
            \A^T}^{-1}\LRp{\db -\A\ub_0}.
  \end{align*}
  since matrix multiplication and matrix inversion are continuous functions, 
  \begin{equation*}
    \urs_\N \asconv \urs = \ub^*.
  \end{equation*}
  by the continuous mapping theorem \cite[Theorem 2.3]{van2000asymptotic}.
\end{proof}

The key step here is recognizing that, asymptotically, sampling from $\lamP$
and solving using form $\ub_1$ \eqnref{MAPsolutiona} 
gives the same results as sampling from $\omP$ 
and solving using
form $\ub_2$ \eqnref{MAPsolutionb}. However, Lemma \lemref{RS} does
not imply that $\urs_\N$ is equivalent to sampling from $\lamP$ and
solving using form $\ub_1$ for a finite $\N$. Indeed, when $\N < \dim \LRp{\ub}$,
$\urs_\N$  cannot be rewritten in the form $\ub_1$ since
\begin{equation*}
  \text{rank} \LRp{\frac{1}{\N} \sum_{j=1}^\N \omb^j \LRp{\omb^j}^T}
  \leq \N < \dim \LRp{\ub}.
\end{equation*}
This implies that the sample average of $\expect_{\omP} \LRs{\omb \omb^T}$ is not
invertible, breaking an assumption of Lemma \lemref{RS} and
showing that $\urs_\N$ does not satisfy the optimality condition of \eqnref{MAPfinite}. 
Here, it is important that $\sCinv$ is invertible, otherwise
$\ub_1 \neq \ub_2$ and the randomized schemes discussed here have no hope of
converging to $\ub^*$.

If in addition to right sketching we use \eqnref{rand_means}
to randomize $\db$ and $\ub_0$ in \eqnref{rightSketch}, and define
\begin{equation}
  \LRp{\uenkf_\N}^i := \ub_0^i +\LRp{\frac{1}{N}\sum_{j=1}^\N
    \omb^j\LRp{\A\omb^j}^T}\LRp{\L + \frac{1}{\N}\sum_{j=1}^\N
    \LRp{\A\omb^j}\LRp{\A\omb^j}^T}^{-1}\LRp{\db^i -\A\ub_0^i},
    \eqnlab{ensemble_kal}
\end{equation}
we rediscover the well-known ensemble Kalman filter (EnKF) update formula
for a single member of the ensemble \cite{Evensen03}.
Notice here that the sketching of $\A$ from the right is
fixed for each random sample $\db^i$ and $\ub_0^i$.
In the language of the EnKF, the sample prior covariance
matrix is fixed for all members of the ensemble.

\begin{corollary}[Asymptotic convergence of EnKF]
  Let
  \begin{equation}
    \eqnlab{enkf_saa}
    \uenkf_{\M,\N} := \frac{1}{\M}
    \sum_{i=1}^\M \LRp{\uenkf_\N}^i. 
  \end{equation}
  Then
  \begin{equation*}
    \uenkf_{\M, \N} \asconv \umap \quad
    \text{as} \quad \M, \N \to \infty.
  \end{equation*}
\end{corollary}
\begin{proof}
  The result follows immediately from Theorem \theoref{asymptotic_convergence},
  Lemma \lemref{RS}, and the interchange of expectation and integration used in
  Corollary \cororef{rMAP}. 
\end{proof}

As with right sketching, this result only holds asymptotically, 
with special sensitivity to $\N$, since the validity of $\LRp{\uenkf_\N}^i$ as an optimal solution
of \eqnref{MAPfinite} requires $\N$ to be large enough to ensure invertibility
of all matrices involved. It should be noted that the EnKF is often used as part of
an iterative method for solving inverse problems rather than used directly 
\cite{iglesias2013ensemble}. The reason for this can be understood by investigating what right sketching
(and thus the EnKF) is doing to the prior covariance matrix and viewing this
through the lens of regularization.

\subsubsection{Right sketching from the left as randomized regularization}
\seclab{randomized_regularization}
Consider again the form $\ub_1$ given in \eqnref{MAPsolutiona}:
\begin{equation*}
    \ub_1 = \LRp{\A^T\L^{-1}\A + \sC^{-1}}^{-1}\LRp{\A^T\L^{-1}\db + \sC^{-1}\ub_0}.
\end{equation*}
While the randomized prior \eqnref{RS_U1}, 
right sketching \eqnref{rightSketch} 
and ensemble Kalman filter \eqnref{ensemble_kal} 
methods still fall under the asymptotic analysis given in Section
\secref{asymptotic_theory} for linear inverse problems,
a practical and theoretical issue arises due to the regularizing role that $\sCinv$ plays.
The inverse of the prior covariance, $\sC^{-1}$, can be considered to be 
a regularization operator when viewed through the lens of deterministic 
inverse problems and is indeed equivalent to a Tikhonov regularization 
strategy \cite{EnglKunischNeubauer89a}. In the deterministic setting, the role of regularization
is often to ``damp out'' highly oscillatory modes caused by the rapidly 
decaying spectrum of $\A$ --- modes that are highly polluted by noise.
While asymptotic analysis (see Theorem \theoref{asymptotic_convergence}) establishes the convergence of these aforementioned methods, it is incapable of explaining why these methods could fail for finite sample size $\N$. This is where non-asymptotic analysis shines.  
Indeed, Lemma \lemref{nonasymptotic_bounds} shows that 
the successful (small error) probability requires quite a large number of samples.
According to Remark 4.7.2 of \cite{vershynin_2018}, the number of samples
required to accurately estimate the covariance matrix is proportional to $n/\tol^2 $ where
$n$ is the dimension of the matrix and $\tol$ is the tolerance.
This is not surprising from a regularization point of view as the sample covariance needs to closely approximate the true covariance
in order to adequately perform its role as a regularizer. 

As a concrete example, consider the simple case where $\sCinv = \alpha \ident$
with $\alpha > 0$.
Letting $\L^{-\half}\A = \Ub \Sb \Vb^T$ be the SVD of the whitened PtO map, 
the first term of $\ub_1$ can be written 
\begin{equation*}
    \LRp{\Vb\Sb^2\Vb^T + \sC^{-1}}^{-1}
    = \LRp{\Vb \LRp{\Sb^2 + \alpha \ident}\Vb^T}^{-1}
    = \Vb \mb{D} \Vb^T,
\end{equation*}
where $\mb{D}$ is the diagonal matrix with the $i$th diagonal element given by $\mb{D}_{ii} =\frac{1}{\Sb_{ii}^2 + \alpha}$.
Comparing to the case of no regularization, we can see that the inverse of the prior 
covariance shifts the spectrum of $\A^T \L^{-1} \A$ upward by the constant $\alpha$. 
Furthermore, upon inverting, $\alpha > 0$ ensures that the denominator of
$\frac{1}{\Sb_{ii}^2 + \alpha}$ is not too close to 0, keeping the 
inverse solution from blowing up as $\Sb_{ii}^2 \to 0$. 
Now, consider the RS\_U1 randomization of $\sCinv$ proposed in \eqnref{RS_U1} --- 
the same randomization as right sketching
when viewed in the $\ub_1$ form (sketching the prior from the left):
\begin{equation*}
    \sC^{-1} = \expect_{\lamP} \LRs{\lam \lam^T} 
    \approx \frac{1}{\N} \sum_{i=1}^{\N} \LRp{\lam^i} \LRp{\lam^i}^T.
\end{equation*} 
As we saw before, this randomization converges as $\N \to \infty$, 
but the convergence rate $\bigO (1 / \sqrt{\N})$ of a 
SAA is notoriously slow. So how 
does this slow convergence affect the regularization strategy?
Clearly when $\N < \dim \LRp{\ub}$, the regularization is not full rank and 
there may be $\dim \LRp{\ub} - \N$ modes of $\A^T \L^{-1} \A$
left unregularized, assuming the random matrix has linearly independent columns.
Even in the case when $\N \geq \dim \LRp{\ub}$, 
slow convergence of the SAA leaves modes underregularized leading to oscillatory solutions as seen in 
Figures \figref{1D_Decon_RS} and \figref{1D_Decon_ENKF} for the 1D deconvolution problem. This can also be
seen explicitly in Figure \figref{saa_covariance_convergence_identity} where the spectrum of the sample average
inverse covariance is plotted against the spectrum of the true prior inverse covariance for various $\N$.

In Figure \figref{saa_covariance_convergence_identity}, we consider the case
where $\sC = \ident$ and $\dim (\ub) = 1000$. 
\begin{figure}[h!t!b!]
  \centering
  \begin{tabular}{c}
    \textbf{\small{Spectrum of sample inverse covariance $\lam \lam^T$}}\\
    \begin{subfigure}{0.45\textwidth}
      \centering
      \includegraphics[trim=1.5cm 0.5cm 1.5cm 0.68cm,clip=true,width=\textwidth]{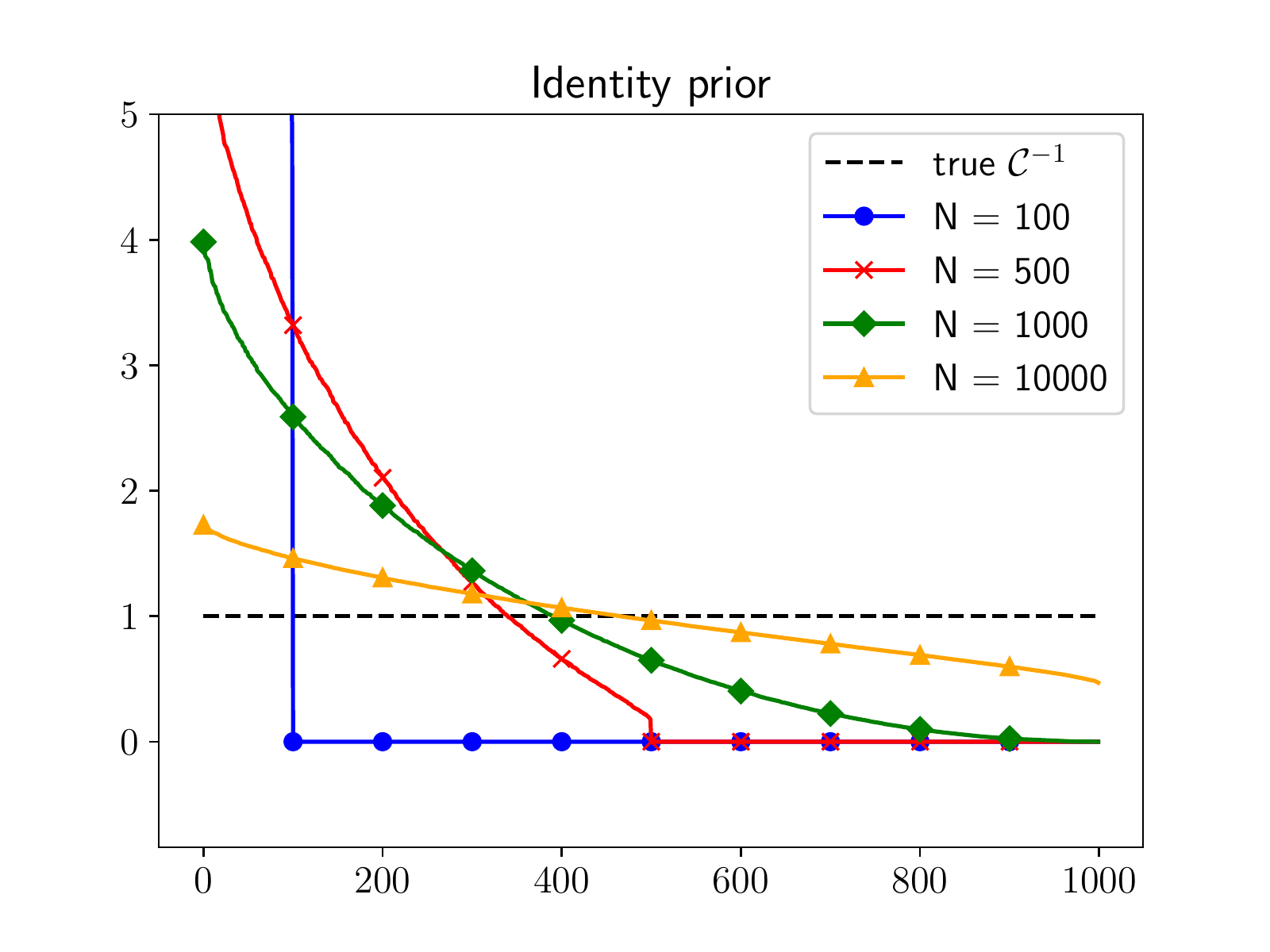}
      \caption{}
      \figlab{saa_covariance_convergence_identity}
    \end{subfigure}%
    \hfill
    \hspace{0.5cm}
    \begin{subfigure}{0.45\textwidth}
      \centering
      \includegraphics[trim=1.5cm 0.5cm 1.5cm 0.68cm,clip=true,width=\textwidth]{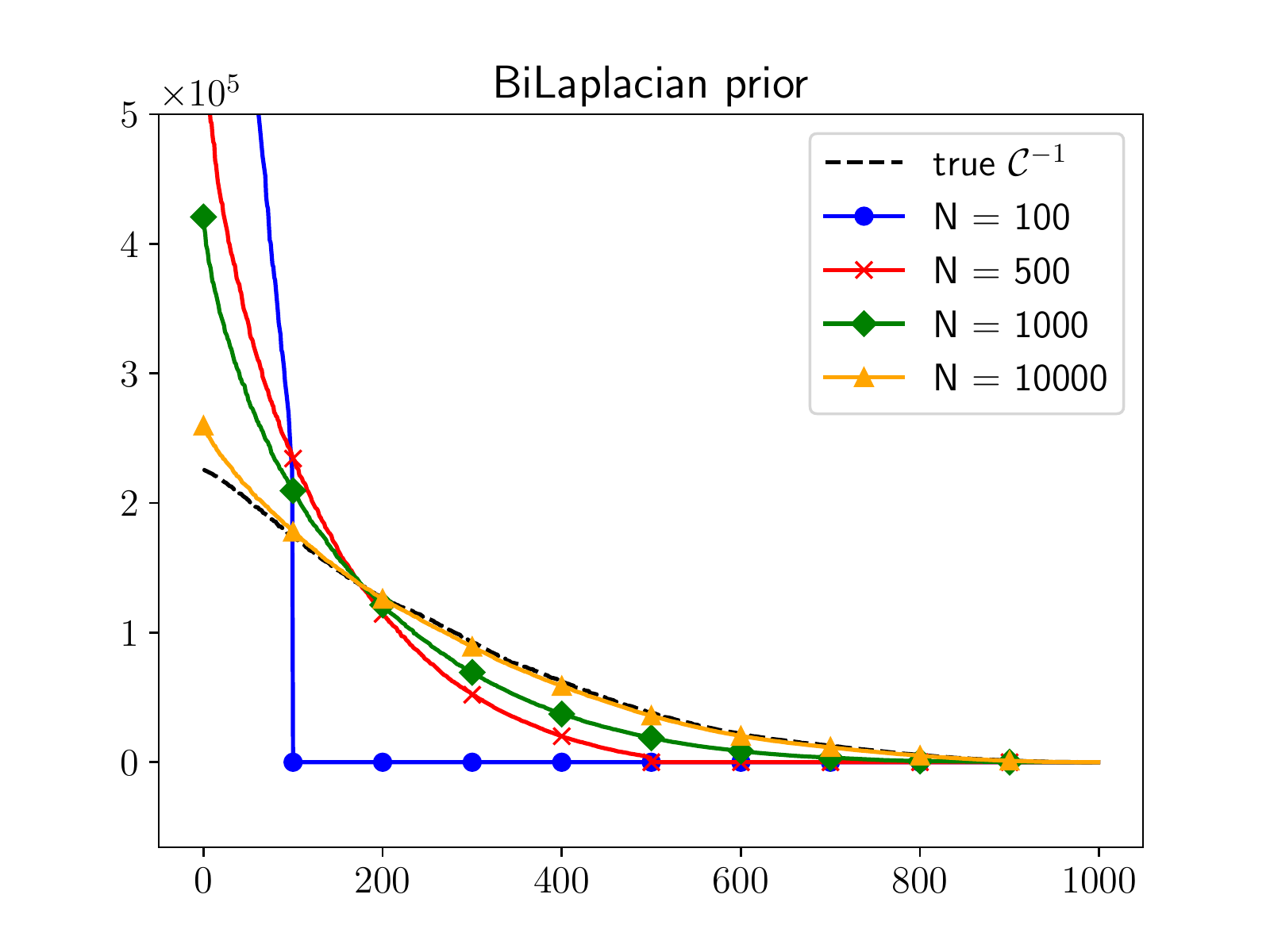}
      \caption{}
      \figlab{saa_covariance_convergence_bilaplacian}
    \end{subfigure}
  \end{tabular}
  \caption{Convergence of spectrum of the sample average approximation of the
    inverse prior covariance for the case where $\sC = \ident \in \reals^{1000 \times 1000}$ (a) and $\sC$ is the BiLaplacian (b). When $N = 100$, there
    are fewer samples than the dimension of the parameter and some modes are left completely unregularized. Even when there are more samples than the dimension of the parameter, this does not
    guarantee acceptable convergence for SAA. This shows that the sample
    average of the inverse prior covariance converges slowly to the true inverse prior covariance.
    However, when the spectrum of the inverse prior covariance decays, the sample average approximation more closely
    matches the true inverse prior covariance with fewer samples. }
    \figlab{saa_covariance_convergence}
\end{figure}
Because randomizing the inverse of the prior covariance results 
in a poor performing regularizer, solutions using right sketching 
or a single step of the ensemble Kalman filter exhibit highly oscillatory behavior when 
choosing $\N$ to be of reasonable size, at least for the identity prior.
In problems where a decaying prior spectrum is desirable,
randomization of the prior
has a less pronounced effect on the quality of the inverse solution.
For example,
the advection-diffusion PDE constrained inverse problem detailed in Section \secref{hippylib}
with the BiLaplacian prior shows good results with right sketching. The similarity of the
sample average spectrum to the spectrum of the true prior inverse covariance can be seen
in Figure \figref{saa_covariance_convergence_bilaplacian} for the BiLaplacian prior. Additionally, problems where the PtO map has a slowly decaying 
spectrum as in the X-ray tomography problem (Section \secref{xray_tomography})
may also be less sensitive to inaccurate approximations to $\sC^{-1}$.
\section{Numerical results}
\label{section:numerical_results}
In this section we show numerical results for a variety of inverse problems demonstrating the asymptotic convergence of
various methods. As the possible number of randomized variants would be unnecessarily burdensome to enumerate,
we will focus on a few key methods: 
randomized misfit approach \eqnref{RMA_saa}, 
randomized MAP \eqnref{rmap_saa}, 
the combination of RMA and RMAP \eqnref{rma_rmap_opt_2},
right sketching \eqnref{rightSketch},  
the ensemble Kalman filter \eqnref{enkf_saa}, 
and randomizing everything \eqnref{independent_SA_loss} (listed as ALL). 
It is important to keep in mind
that we are not advocating for or against the use of any particular method --- this section is to serve as
numerical validation of the asymptotic convergence of each method. 
Additionally we discuss the differing convergence behavior of each method for different problems. 
In particular, we find empirically that
methods randomizing $\sCinv$
such as right sketching, the EnKF, and ALL generally have very poor performance and require
many more samples than the dimension of the problem in order to provide suitable 
results for several  problems. 
The reason for this has been discussed at length in Section \secref{randomized_regularization}.
These methods do however exhibit asymptotic convergence to the MAP solution as predicted by our theoretical results. 

To explore the performance and convergence of the various methods, we consider a variety of
prototype problems with different characteristics. The 1D deconvolution problem 
with scaled identity prior covariance is a relatively simple inverse problem 
that provides easily digestible visualizations of the convergence
for each method. X-ray tomography is a mildly ill-posed two dimensional imaging problem with fewer
observations than parameters. The fact that it is only mildly ill-posed exposes interesting effects
in the context of randomization.
We also show the convergence of each method for a linear time dependent PDE-constrained
inverse problem with PDE-based prior covariance on a domain with a hole. Finally, we conclude with an
example demonstrating convergence on a non-linear elliptic PDE-constrained inverse problem.

In problems with more than one randomization, such as EnKF and RMA+RMAP,
each expectation can be approximated
by a separate sample average. However, exploring the effect of choosing a different number of samples
for each random variable is outside the scope of this paper and serves only to obscure
the asymptotic convergence property that we aim to show in this section. Therefore, all methods
assume that the number of random samples is the same for all random variables. To be more concrete, we set
$\uenkf_{M,N} = \uenkf_{N,N}$  in \eqnref{enkf_saa}.
In addition, the relative errors presented are with respect to $\umap$, not the true solution,
emphasizing the errors induced by randomization rather than errors due to other effects.
This is due to the fact that the theory presented shows convergence to $\umap$.

\subsection{1D Deconvolution problem}
Deconvolution, the inverse problem associated with the convolution process, finds enormous application
in the signal and image processing domains \cite{kundur1996blind,swedlow2013quantitative,ryan2007free}.
For demonstration, we consider the 1D deconvolution problem with a 1-periodic function given by: 

\[f(x) = \sin (2\pi x) + \cos (2\pi x) \quad x \in [0,1].\]
The domain is divided into $n = 1000$ sub-intervals. The kernel is constructed as \cite{mueller2012linear}:

\[\Psi(x) = C_a \LRp{x+a}^2 \LRp{x-a}^2,\]
where $a = 0.235$ and the constant $C_a$ is chosen to enforce the normalization condition \cite{mueller2012linear}. Synthetic observations are generated with
5\%  additive Gaussian noise. We choose to randomize using the Achlioptas distribution \cite{LeEtAl2017,achlioptas2003}, 
an example of an $l-$percent sparse random variable with $l = 2/3$ and entries in $\LRc{-1, 0, 1}$ with equal probability.
The reconstructed functions obtained by different randomization approaches are shown in
Figure  \figref{1D_Decon} and 
the relative errors are given in Table \ref{Table:1D_Deconvolution}.
It can be seen that the right sketching and EnKF methods give the least accurate results
as evidenced in Table \ref{Table:1D_Deconvolution}. 
This is because randomizing the inverse of the prior covariance results in poor performance
as a regularizer, providing numerical confirmation of the discussion in Section \secref{randomized_regularization}. Other methods perform reasonably well.
While not all methods perform equally well, all methods converge as more samples are
taken and this is consistent with our asymptotic convergence results. 

\begin{table}[h!]
\centering
\begin{tabular}{ |l|l|l|l|l|l| }
    \hline
    \multirow{2}{*}{Method} & \multicolumn{5}{c|}{Relative error ($\%$)}   \\ \cline{2-6}
      & N = 10& N = 100& N = 1000& N = 10000 & N = 100000\\ \cline{2-6}
    \hline
    \text{RMAP}  & 7.74 & 2.39 & 0.80 & 0.25 & 0.07  \\ \cline{1-6}
    \text{RMA}  & 35.5 & 4.2 & 1.8 & 0.42 & 0.18  \\ \cline{1-6}
    \text{RMA+RMAP}  & 41.7 & 9.0 & 1.7 & 0.80 & 0.15  \\ \cline{1-6}
    \text{RS}  & 917 & 295 & 106 & 31.4 & 9.8  \\ \cline{1-6}
    \text{ENKF}  & 896 & 352 & 100 & 31.9 & 9.7  \\ \cline{1-6}
    \text{ALL}  & 158 & 33875 & 844 & 33.3 & 9.3  \\ \cline{1-6}
  \end{tabular}
  \vspace*{0.25cm}
  \caption{Relative error for various randomized methods compared to the $\umap$
  solution for 1D deconvolution. }
  \label{Table:1D_Deconvolution}
\end{table}
 
\begin{figure}[htb!]
    \centering
    \begin{tabular}{c}
        \hspace{-1.5cm}\textbf{1D Deconvolution}\\
        \rotatebox{90}{\hspace{2cm}Relative error}
        \includegraphics[width=0.9\textwidth, trim=1.0cm 0.9cm 0.0cm 0.8cm,clip=true,]{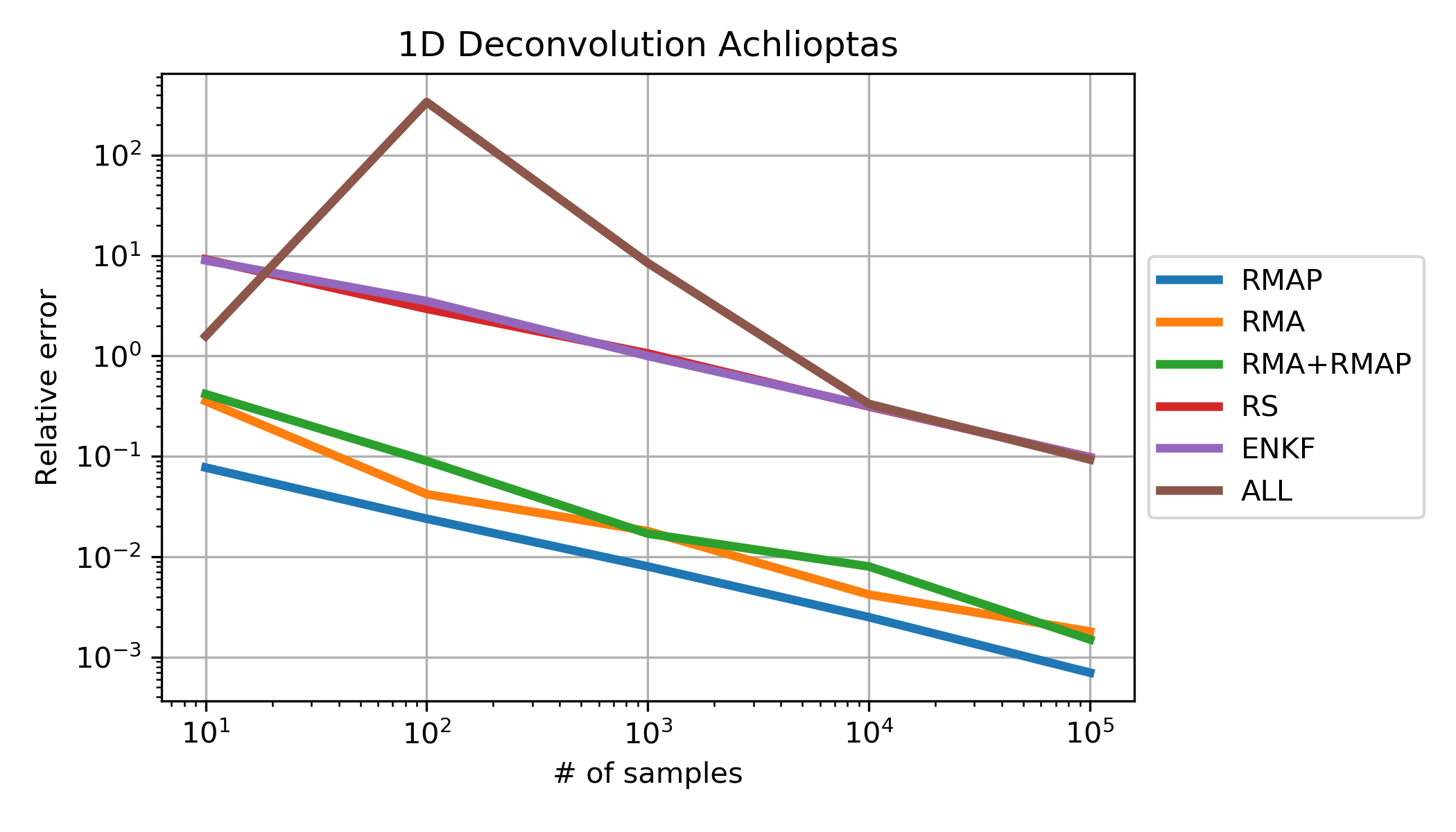} \\
        \hspace{-1.0cm}$N$
    \end{tabular}
    \caption{Relative error plot for 1D Deconvolution problem with Achlioptas random variable. }
\end{figure}

\subsection{X-ray tomography}
\seclab{xray_tomography}
\newcommand{\xraySize}{64}
In x-ray tomographic imaging, X-ray projections of an object are captured at multiple
angles and the inverse problem is to recover the internal structure of the object from the
projection data \cite{mueller2012linear}.
We consider the canonical \textit{phantom} image of size $\xraySize \times \xraySize$ pixels 
with $45$ measurement angles uniformly divided over the range $[0, \pi]$. With this number of measurement
angles, the PtO map has shape $\xraySize \times 45$ by $\xraySize^2$ resulting in fewer observations than parameters
(pixels). 
A scaled identity prior covariance is once again considered.
Measurements are corrupted with 1\% additive Gaussian noise.

\begin{table}[h!]
\centering
\begin{tabular}{ |l|l|l|l|l|l| }
    \hline
    \multirow{2}{*}{Method} & \multicolumn{5}{c|}{Relative error ($\%$)}   \\ \cline{2-6}
      & N = 10& N = 100& N = 1000& N = 10000& N = 50000\\ \cline{2-6}
    \hline
    \text{RMAP}  & 6.44 & 6.44 & 6.44 & 6.44 & 0.04  \\ \cline{1-6}
    \text{RMA}  & 94.17 & 74.95 & 39.42 & 31.07 & 4.02  \\ \cline{1-6}
    \text{RMA+RMAP}  & 96.64 & 77.94 & 40.35 & 30.89 & 4.02  \\ \cline{1-6}
    \text{RS}  & 191.87 & 373.07 & 176.02 & 51.87 & 22.43  \\ \cline{1-6}
    \text{ENKF}  & 324.27 & 352.58 & 178.58 & 52.25 & 21.97  \\ \cline{1-6}
    \text{ALL}  & 95.12 & 71.30 & 80.50 & 60.36 & 23.60  \\ \cline{1-6}
  \end{tabular}
  \vspace*{0.25cm}
  \caption{Relative error for various randomized methods compared to the $\umap$ solution for 
  the X-ray tomography problem. }
  \label{Table:XRAY}
\end{table}

\begin{figure}[htb!]
    \centering
    \begin{tabular}{c}
        \hspace{-1.5cm}\textbf{X-ray tomography}\\
        \rotatebox{90}{\hspace{2cm}Relative error}
        \includegraphics[width=0.9\textwidth, trim=1.0cm 0.9cm 0.0cm 0.8cm,clip=true,]{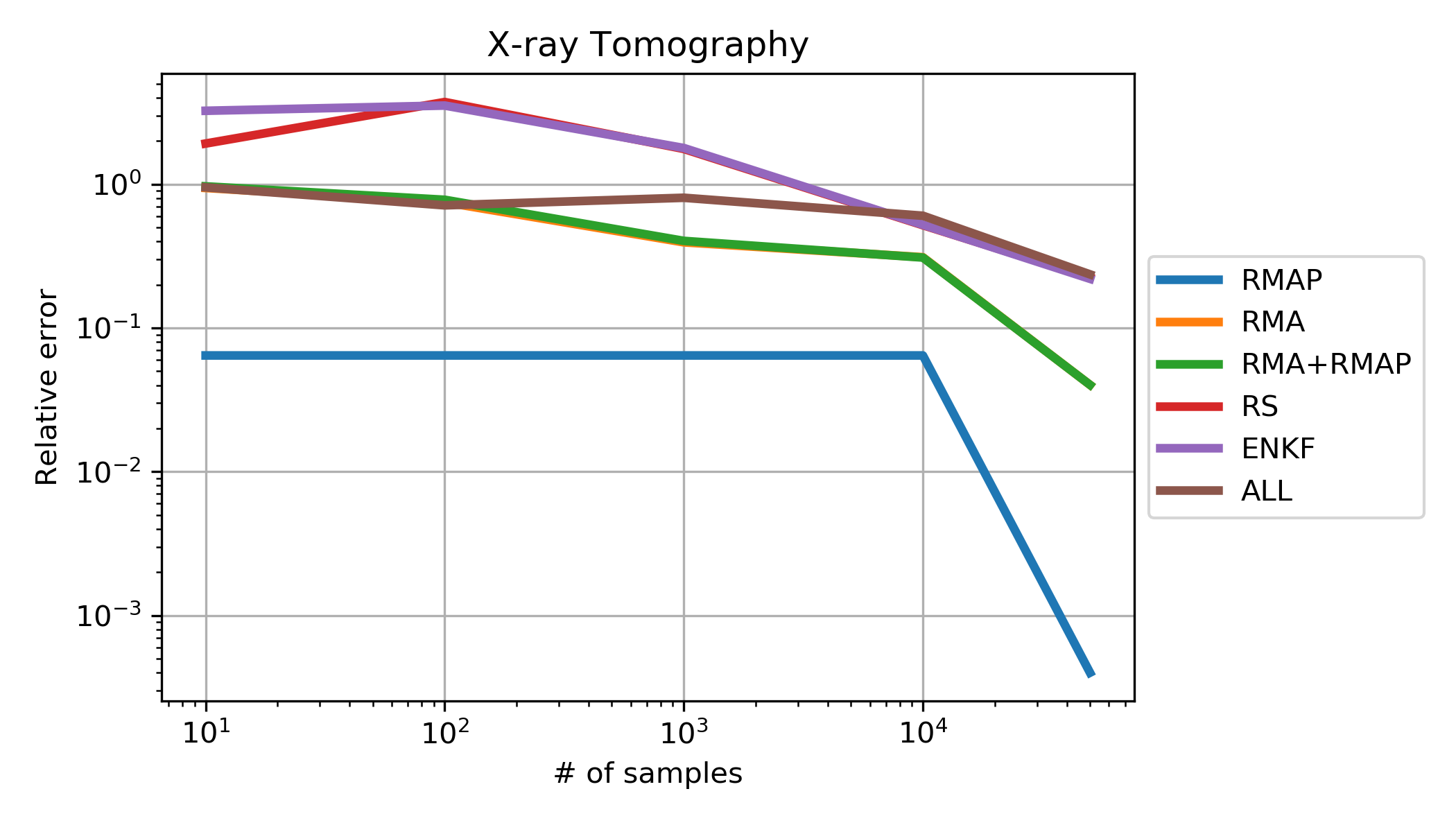} \\
        \hspace{-1.0cm}$N$
    \end{tabular}
    \caption{Relative error plot for X-ray tomography problem.}
\end{figure}

The results are shown in Figure \ref{X_ray} and 
Table \ref{Table:XRAY} shows the relative error for different methods.
Two observations are in order.
First, results show asymptotic convergence of all methods, though convergence is noticeably slower for
RMA and RMA+RMAP than in previous problems.  
This occurs because X-ray tomography is only a mildly ill-posed inverse problem with the spectrum of the PtO map
decaying slowly after an initial fast decay (Figure \figref{xray_spectrum}). This means that the effective rank of the PtO map
is close to the dimension of the data in the case presented. While mildly ill-posed problems
are usually easier to work with, this can present a challenge for randomized methods, particularly
methods such as RMA that randomize the misfit term. Recall that for any two matrices $\bs{X}$ and $\bs{Y}$,
$\Rank(\bs{XY}) \leq \Rank(\bs{X})$. By projecting the misfit term onto a lower dimensional
subspace, important information is lost in the case where $\A$ is has effective
rank close to the dimension of the data. This indicates that such a method is better suited for problems
that are severely ill-posed.

\begin{figure}
\centering 
\hspace*{-2cm}
    \begin{tabular}{c}
        \hspace{0.5cm}\textbf{Spectrum of the PTO map for X-ray tomography}\\
        \rotatebox{90}{\hspace{2cm}singular value}
        \includegraphics[width=0.7\textwidth, trim=0.8cm 0.9cm 0.0cm 0.0cm,clip=true,]{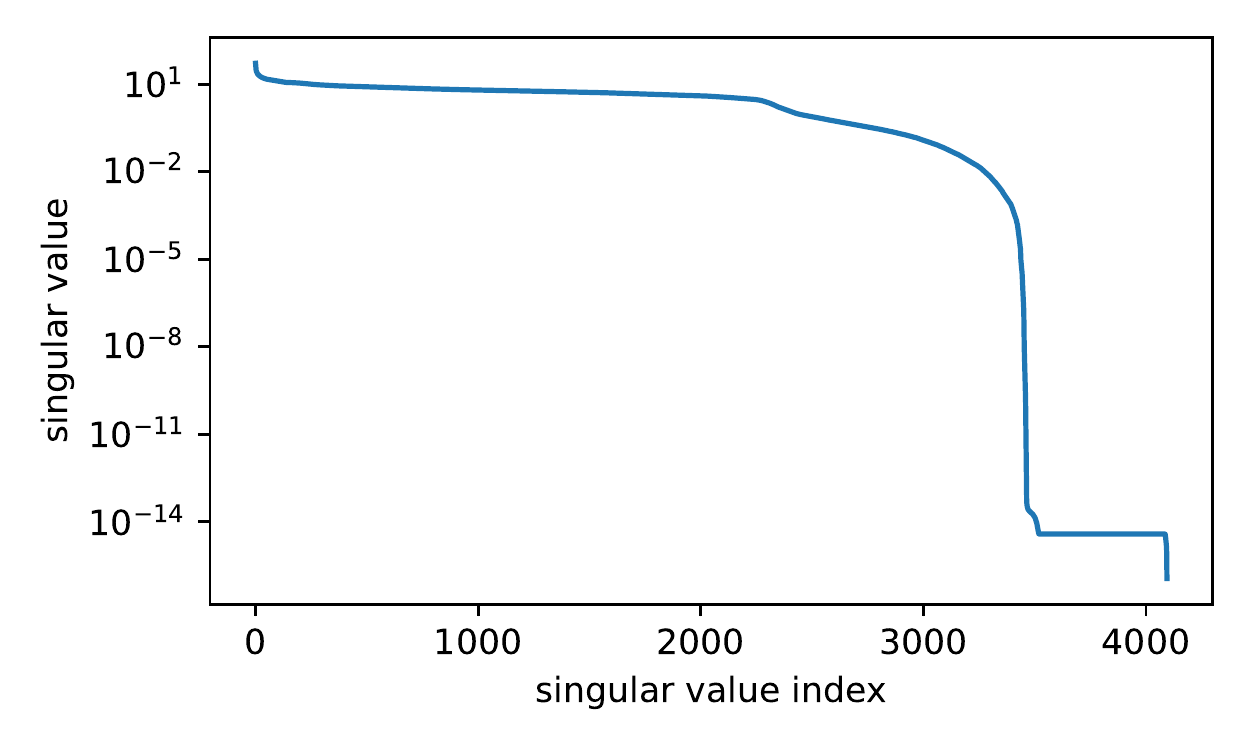} \\
        \hspace{0.8cm}singular value index
    \end{tabular}
\caption{The singular values of the parameter-to-observable map for an X-ray tomography problem 
decay rapidly at first and then slowly until the last few singular vectors. This shows that the 
effective rank of the PtO map is close to the minimum dimension. }
\figlab{xray_spectrum}
\end{figure}

The second observation is that the error for randomized prior methods initially increases
then decreases as the number of samples increases.
In previous problems, we have used a direct linear solver to find the solution to the stochastic optimization problem.
In this problem, we use an iterative conjugate gradient (CG) solver to showcase how solver choice
interacts with the randomized approaches. The main difference that can be seen here
between a direct solver and an iterative solver such as CG is that a direct solver will
invert possibly tiny eigenvalues of an ill-conditioned matrix while an iterative solver will often
stop early, depending on the convergence parameters set, acting as an iterative regularizer 
\cite{PiccolominiZama99,LandiPiccolominiTomba16,HankeNagy96}.
This effect is particularly pronounced on the randomized prior methods such as RS and EnKF where
the low rank randomized prior causes the iterative solver to stop earlier with fewer samples.
This causes the error of RS and EnKF to increase initially until the regularization has sufficient rank, 
then the methods converge asymptotically to $\umap$. 
In the case of X-ray tomography, full-rank regularization 
is not needed due to the mildly ill-posed nature
of the problem. 

\subsection{Initial condition inversion in an advection-diffusion problem}
\seclab{hippylib}
We now consider a linear inverse problem governed by a parabolic PDE
based on the method used in \cite{AustinMerced2017}.
The parameter to observable map (advection-diffusion equation) maps an initial condition
$\ub \in L^2(\Omega)$ to pointwise spatio-temporal observations of the concentration field
$\by(x,t)$. The advection-diffusion equation is given by:

\begin{equation}
    \begin{split}
        \by_t - \kappa\Delta \by + {\vec v} \cdot \nabla \by &= 0 \quad \text{in } \Omega \times(0,T),\\
        \by(.,0)&=\ub\quad \text{in } \Omega,\\
        \kappa \nabla \by \cdot \bf{n}&=0\quad \text{on } \partial \Omega \times (0,T),
    \end{split} 
    \eqnlab{advection_diffusion}
\end{equation}
where, $\Omega\subset \R^2$ is a bounded domain, $\kappa>0$ is the diffusion coefficient,
$T>0$ is the final time. The velocity field $\vec{v}$ is computed by solving the following
steady-state Navier-Stokes equation with the side walls driving the flow:

\begin{equation}
   \begin{aligned}
        - \frac{1}{\operatorname{Re}} \Delta {\vec v} + \nabla q + {\vec v} \cdot \nabla {\vec v} &= 0 &\quad&\text{ in }\Omega,\\
        \nabla \cdot {\vec v} &= 0 &&\text{ in }\Omega,\\
        {\vec v} &= {\vec g} &&\text{ on } \partial\Omega.
    \end{aligned} 
    \eqnlab{velocity_field}
\end{equation}
where $q$ is the pressure, and $\operatorname{Re}$ is the Reynolds number.
The Dirichlet boundary condition $\vec{g}\in \R^2$ is prescribed as $\vec{g}=[0, 1]$ on the left side of the domain,
and $\vec{g}=[0, 0]$ elsewhere. Velocity boundary conditions are not prescribed on the right side of the boundary. 
The values of the forward solution $\by$ on a set of locations $\{ x_1, x_2, ..., x_m\}$ at the final
time $T$ are extracted
and used as the observation vector $\bd\in \R^k$ for solving the initial condition inverse problem.
Synthetic observations are generated by corrupting this observation vector with 1\% additive Gaussian noise. 
The observation data and the velocity profile used in the study are shown in Figure \ref{velocity_profile}.
Upon discretization, the operator $\A$ maps the initial condition  $\ub \in \R^n$ to the observation $\bd\in \R^k$.
\begin{figure}[h!t!b!]
  \centering
        \begin{tabular}{c}
            \begin{subfigure}[b]{0.4\textwidth}
                \begin{tabular}{c}
                    \textbf{{ Velocity profile} }\\
                    \begin{subfigure}{\textwidth}
                        \centering
                        \includegraphics[trim=2.0cm 1.5cm 5.5cm 0.8cm,clip=true,width=\textwidth]{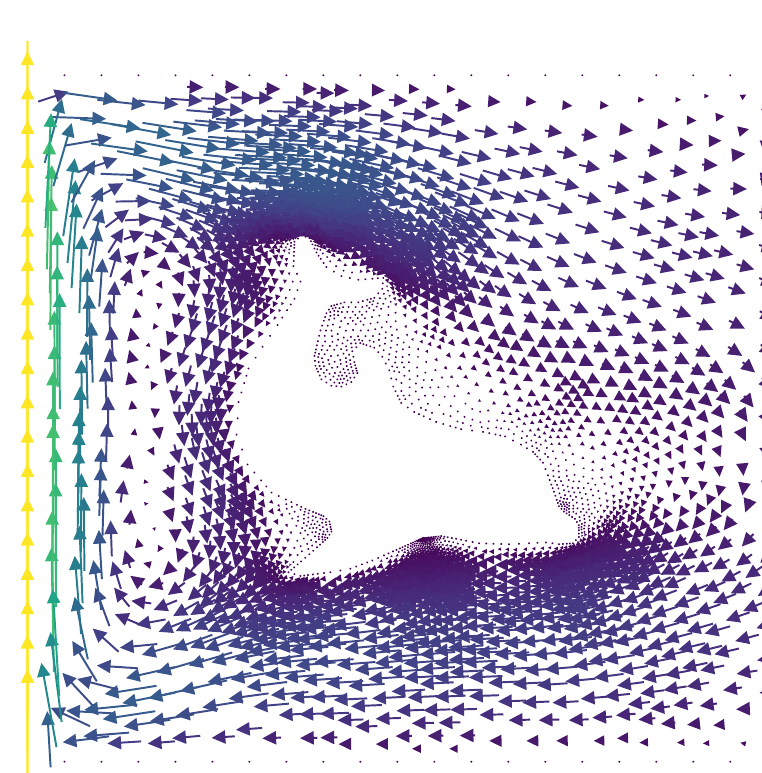}
                    \end{subfigure}
                \end{tabular}
            \end{subfigure}%
            
            \begin{subfigure}[b]{0.1\textwidth}
                \begin{tabular}{c}
                    $\;$ \\
                    $\;$ \\
                    \begin{subfigure}[b]{\textwidth}
                        \centering
                        \includegraphics[trim=11.5cm 1.0cm 1.5cm 0.8cm,clip=true,width=\textwidth]{pdf_figures/Linear_PDE/velocity.pdf}
                    \end{subfigure}
                \end{tabular}
            \end{subfigure}%

            \begin{subfigure}[b]{0.4\textwidth}
                \begin{tabular}{c}
                \textbf{Observation at T=3s}\\
                    \begin{subfigure}{\textwidth}
                        \centering
                        \includegraphics[trim=2.0cm 1.5cm 7.0cm 0.8cm,clip=true, width=\textwidth]{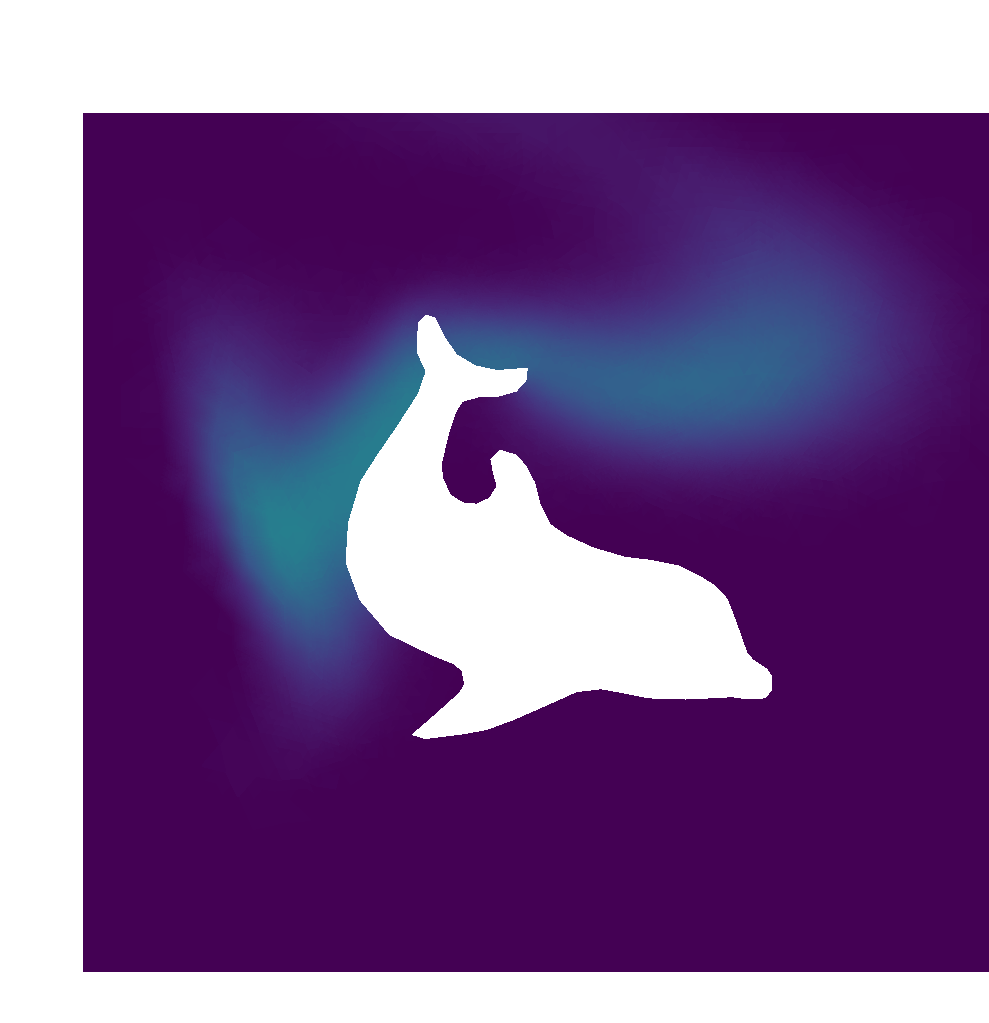}
                    \end{subfigure}
                \end{tabular}
            \end{subfigure}
            
            \begin{subfigure}[b]{0.1\textwidth}
                \begin{tabular}{c}
                 $\quad\quad$\\
                 $\;$ \\
                    \begin{subfigure}{\textwidth}
                        \centering
                        \includegraphics[trim=14.8cm 1.0cm 1.5cm 0.8cm,clip=true, width=\textwidth]{pdf_figures/Linear_PDE/observ.pdf}
                    \end{subfigure}
                \end{tabular}
            \end{subfigure}
      \end{tabular}
    \caption{ The velocity profile and observation data used for inversion}
    \label{velocity_profile}
\end{figure}

In addition, we define the prior covariance matrix to be the PDE-based BiLaplacian prior defined as:

\begin{equation}
  \sC=\LRp{\delta I+\gamma \nabla \cdot (\theta \nabla)}^{-2},  
  \eqnlab{Bi_Laplacian}
\end{equation}
where, $\delta$ governs the variance of the samples, while the ratio $\frac{\gamma}{\delta}$ governs the correlation length.
$\theta$ is a symmetric positive definite tensor to introduce anisotropy in the correlation length.

Following \cite{AustinMerced2017}, a mixed formulation employing $P2$ Lagrange elements for approximating the velocity field and $P1$
elements for pressure is adopted
for solving \eqnref{velocity_field} to obtain the velocity field.
The computed velocity field is then used to solve the advection-diffusion equation, \eqnref{advection_diffusion}.
$P1$ Lagrange elements are used for the variational formulation of the advection-diffusion equation. 

The observation vector $\db$ is computed at time $t=3s$ with $m=200$ observation points.
For this problem, there are $n = 2868$ degrees of freedom. The diffusion coefficient is $\kappa=0.001$
and the parameters of the BiLaplacian prior \eqnref{Bi_Laplacian} are $\delta=8$, $\gamma=1$, 
and $\Theta = \ident$.

\begin{table}[h!]
\centering
\begin{tabular}{ |l|l|l|l|l| }
    \hline
    \multirow{2}{*}{Method} & \multicolumn{4}{c|}{Relative error ($\%$)}   \\ \cline{2-5}
      & N = 10& N = 100& N = 1000& N = 10000\\ \cline{2-5}
    \hline
    \text{RMAP}  & 5.02 & 1.96 & 0.49 & 0.15  \\ \cline{1-5}
    \text{RMA}  & 53.38 & 15.16 & 5.30 & 1.15  \\ \cline{1-5}
    \text{RMA+RMAP}  & 80.14 & 14.48 & 7.05 & 1.30  \\ \cline{1-5}
    \text{RS}  & 59.58 & 26.78 & 7.33 & 5.06  \\ \cline{1-5}
    \text{ENKF}  & 74.09 & 18.76 & 7.10 & 5.43  \\ \cline{1-5}
    \text{ALL}  & 91.58 & 197.66 & 193.25 & 9.66  \\ \cline{1-5}
  \end{tabular}
  \vspace*{0.25cm}
  \caption{Relative error for various randomized methods compared to the $\umap$
  solution for 2D linear advection-diffusion initial condition inverse problem.}
  \label{Table:Linear_PDE}
\end{table}

\begin{figure}[htb!]
    \centering
    \begin{tabular}{c}
        \hspace{-1.5cm}\textbf{2D Advection-Diffusion}\\
        \rotatebox{90}{\hspace{2cm}Relative error}
        \includegraphics[width=0.85\textwidth, trim=1.0cm 0.9cm 0.0cm 0.8cm,clip=true,]{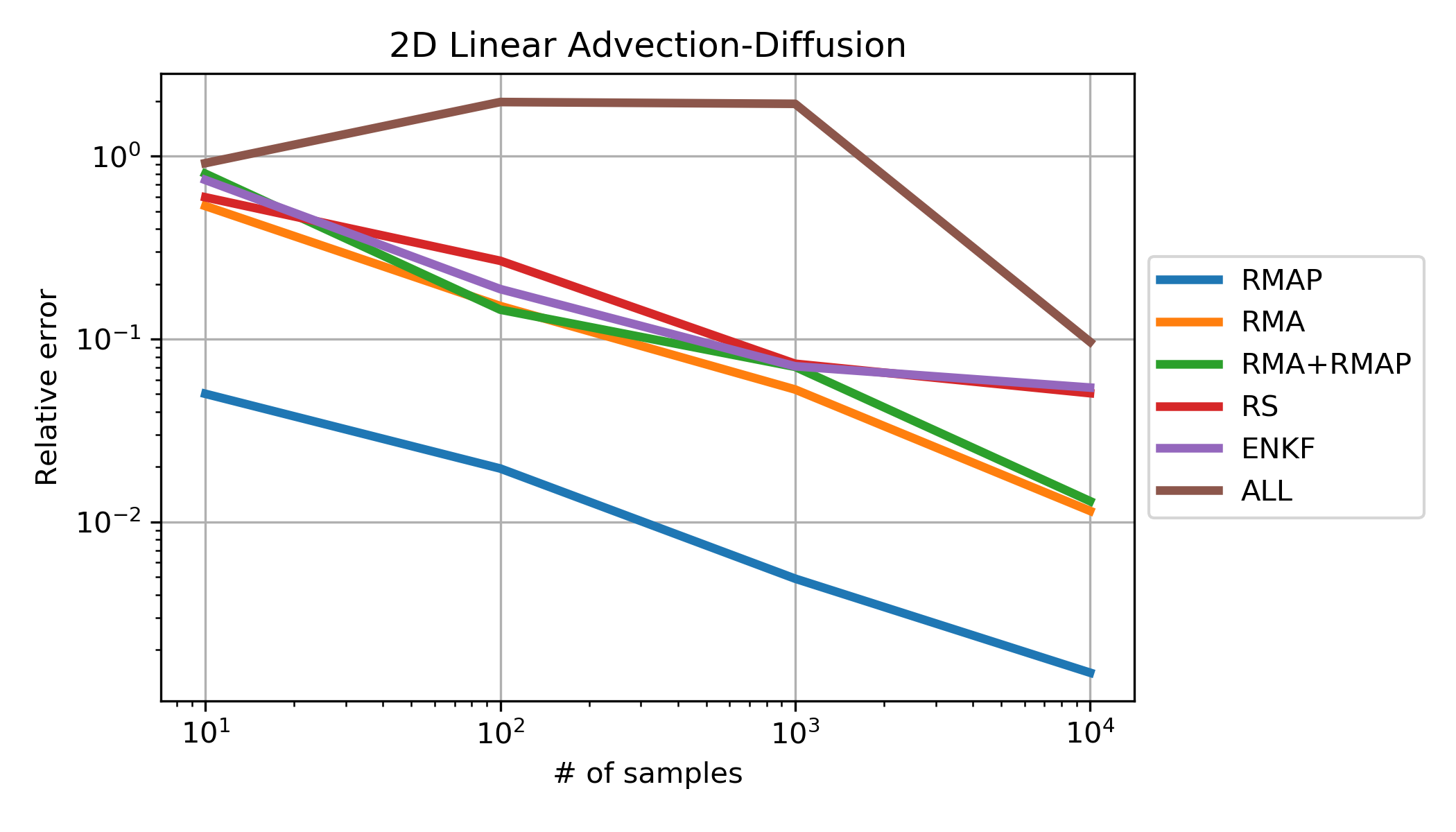} \\
        \hspace{-1.0cm}$N$
    \end{tabular}
    \caption{Relative error plot for 2D linear advection-diffusion initial condition inversion.}
\end{figure}

The MAP solution $\umap$ is shown in Figure \ref{non_u1_xtrue}.
The condition number of $\sC^{-1}$ is of the order of $10^{6}$.
The results of different randomization schemes are shown in Figure
\ref{Linear_PDE} in the appendix. 
Table \ref{Table:Linear_PDE} gives the relative error with respect to $\umap$.
As expected, randomized MAP gives the most accurate results followed by left sketching and the randomized misfit approach.
In contrast to the previous examples considered, 
the right sketching and EnKF approaches gives reasonably good results as evident from Table \ref{Table:Linear_PDE} and Figure \ref{Linear_PDE} in the appendix.

\begin{figure}[h!t!b!]      
\centering
  \begin{tabular}{c}
    \begin{subfigure}[b]{0.31\textwidth}
      \begin{tabular}{c}
        \textbf{True solution}\\
        \begin{subfigure}{\textwidth}
          \centering
          \includegraphics[trim=2.5cm 2.5cm 7.0cm 1.5cm,clip=true, width=\textwidth]{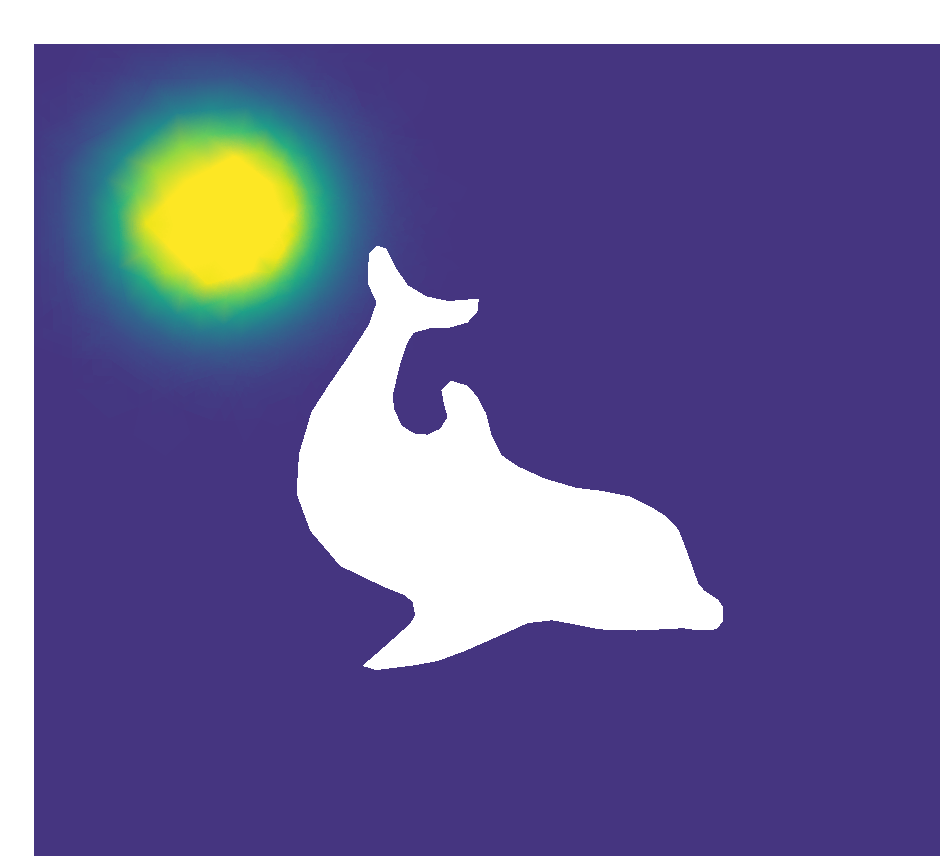}
        \end{subfigure}
      \end{tabular}
    \end{subfigure}%
    \begin{subfigure}[b]{0.31\textwidth}
      \begin{tabular}{c}
        \textbf{$\umap$ solution }\\
        \begin{subfigure}{\textwidth}
          \centering
          \includegraphics[trim=2.5cm 2.5cm 7.0cm 1.5cm,clip=true,width=\textwidth]{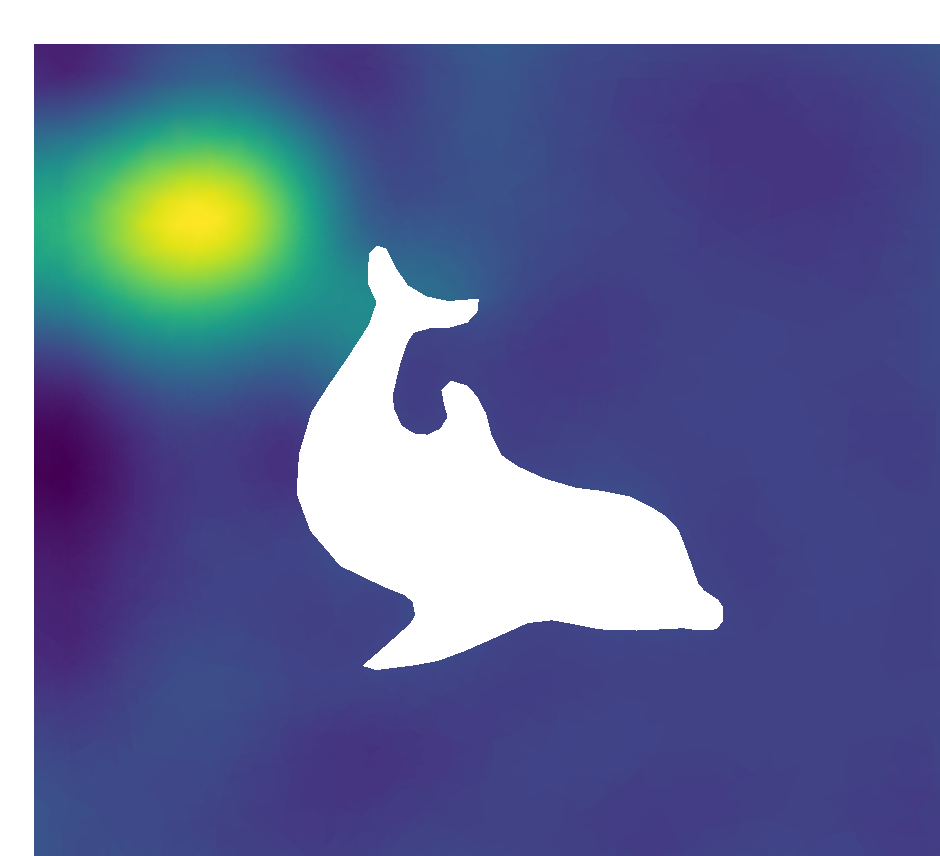}
        \end{subfigure}
      \end{tabular}
    \end{subfigure}%
    
    \begin{subfigure}[b]{0.31\textwidth}
      \begin{tabular}{c}
        \textbf{RS, N = 100 }\\
        \begin{subfigure}{\textwidth}
          \centering
          \includegraphics[trim=2.5cm 2.5cm 7.0cm 1.5cm,clip=true,width=\textwidth]{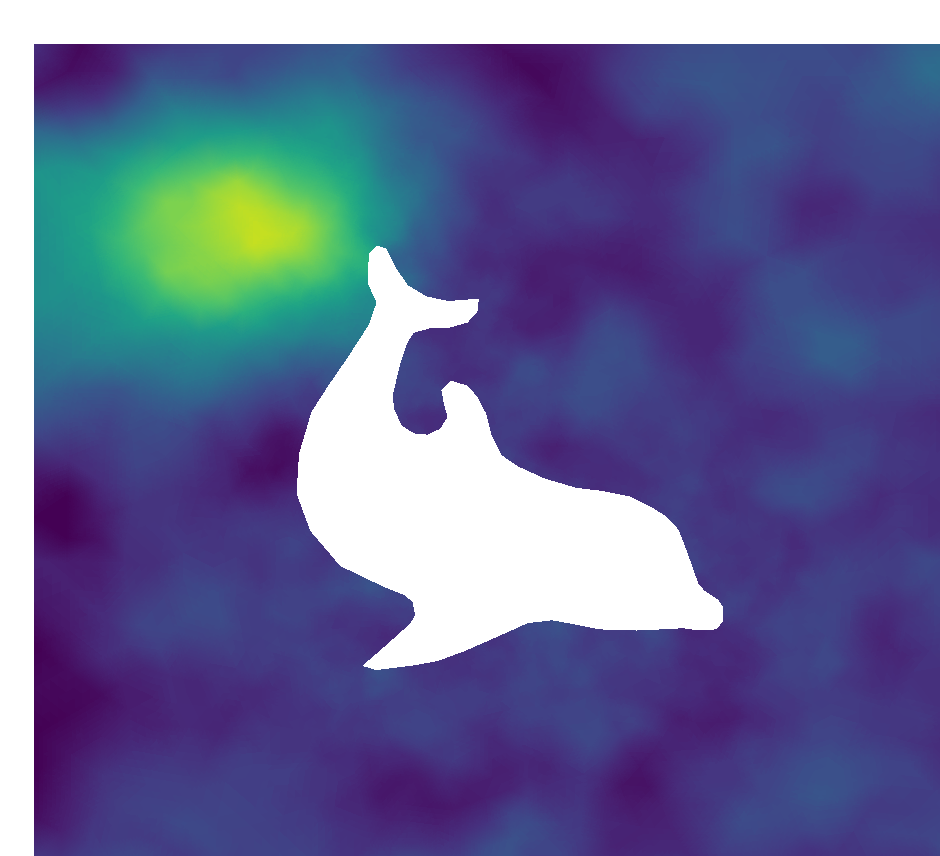}
        \end{subfigure}
      \end{tabular}
    \end{subfigure}%
  \end{tabular}
  \caption{From left to right are true solution,  $\umap$ solution and right sketching solution for linear 
  advection-diffusion initial condition inverse problem. With only 
  100 random samples, right sketching can obtain a reasonably good initial condition reconstruction
  (26\% relative error). 
  This behavior indicates the fast convergence of the 
  randomly sampled prior inverse covariance to the true prior inverse covariance.}
  \label{non_u1_xtrue}
\end{figure}     

This is due to the faster convergence of the randomized prior covariance
to the true prior covariance for the BiLaplacian prior. 
This also points to the fact that care should be exercised when choosing 
a randomized method for a particular problem. For inverse problems 
where a prior with a decaying spectrum is desirable, the EnKF or RS approach
may perform well.

\subsection{Nonlinear parameter inversion in a steady-state heat equation}
To show the convergence of various methods for nonlinear inverse problems, we consider 
a nonlinear PDE constrained parameter inversion problem. In the previous section, we 
considered an initial condition problem where the data depended linearly on the parameter, 
even though the statement of the problem itself was rather involved. Now
we consider a simple to state but extremely ill-posed nonlinear inverse problem. 
Given a steady-state temperature distribution $\bs{T}\LRp{x, y}$ and boundary conditions, 
invert for the conductivity everywhere in the domain. 
The governing equations are given by 
\begin{align*}
    \grad \cdot \LRp{e^\kappa \grad \bs{T}} &= f \quad \text{in } \Omega,\\
    \bs{T} \LRp{x, 0} &= 2 \LRp{1-x},  \\
    \bs{T} \LRp{x, 1} &= 2x,  \\
    \grad \bs{T} \cdot \bs{n} &= 0 \quad \text{on } \partial \Omega \setminus \LRc{y = 0, y = 1}.
\end{align*}
While this equation is linear in the temperature distribution $\bs{T}$, the (log-) thermal conductivity that we are
inverting for, $\kappa$, appears non-linearly. That is, the parameter-to-observable map is nonlinear. 
The heat equation makes for an excellent test problem for inverse solvers since the dependence of the 
steady-state temperature distribution on the conductivity is rather weak. 

\begin{figure}[!htb]
\begin{tabular}{cccc}
\begin{subfigure}[b]{0.35\textwidth}
    \begin{tabular}{c}
    \textbf{True Parameter}\\
    \begin{subfigure}[]{\textwidth}
        \includegraphics[width=\textwidth,trim={1.2cm 1cm 5.5cm 2cm},clip=true]{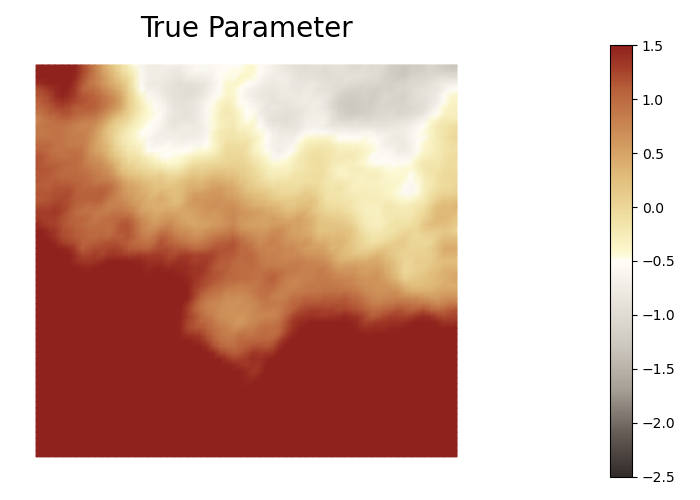}
    \end{subfigure}
    \end{tabular}
    \caption{}
\end{subfigure}&
\begin{subfigure}[b]{0.085\textwidth}
    \includegraphics[width=\textwidth,trim={15.15cm 0cm 0cm 0cm},clip=true]{pdf_figures/Nonlinear/true_parameter.png}
\end{subfigure}&
\begin{subfigure}[b]{0.35\textwidth}
    \begin{tabular}{c}
    \textbf{Observations}\\
    \begin{subfigure}[]{\textwidth}
        \includegraphics[width=\textwidth,trim={0cm 0cm 5cm 0.95cm},clip=true]{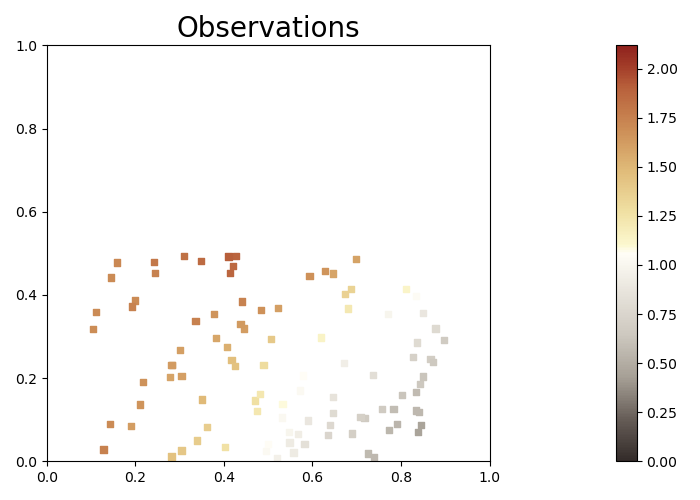}
    \end{subfigure}
    \end{tabular}
    \caption{}
\end{subfigure}&
\begin{subfigure}[b]{0.085\textwidth}
    \includegraphics[width=\textwidth,trim={15.3cm 0.75cm 0cm 1cm},clip=true]{pdf_figures/Nonlinear/observations.png}
\end{subfigure}
\end{tabular}
\caption{The true log-conductivity ($\kappa$) and 100 sparse observations. Observing the 
temperature distribution only in the lower half of the 
domain  makes inverting for the log-conductivity in the entire domain a more difficult task. }
\figlab{nonlinear_setup}
\end{figure}

We again follow a mixed formulation where the temperature distribution is modeled using $P2$ Lagrange 
elements and the parameter is modeled with $P1$ Lagrange elements. With a mesh size of $64 \times 64$ elements, 
this results in discrete variables $\bs{y} \in \R^{16,641}$ and $\kappa \in \R^{4,225}$. 
The BiLaplacian prior defined in \eqnref{Bi_Laplacian} is also used here with $\delta = 0.5$, $\gamma = 0.1$,
and the anisotropic diffusion tensor 
\begin{equation*}
    \theta = 
    \begin{bmatrix}
        \theta_1 \sin^2 \alpha & \LRp{\theta_1 - \theta_2} \sin \alpha \cos \alpha\\
        \LRp{\theta_1 - \theta_2} \sin \alpha \cos \alpha & \theta_2 \cos^2 \alpha
    \end{bmatrix},
\end{equation*}
where $\theta_1 = 2.0$, $\theta_2 = 0.5$ and $\alpha = \pi / 4$. 
Lastly, we consider an inhomogeneous case where 
\begin{equation*}
    f = 50  \sin^2 \LRp{\pi x} \cos^2\LRp{\pi y}.
\end{equation*}

\begin{table}[h!]
\centering
\begin{tabular}{ |l|l|l|l|l| }
  \hline
  \multirow{2}{*}{Method} & \multicolumn{4}{c|}{Relative error ($\%$)}   \\ \cline{2-5}
                          & N = 10& N = 100& N = 1000& N = 5000\\ \cline{2-5}
  \hline
  \text{RMAP}  & 19.50 & 14.66 & 16.39 & 16.17  \\ \cline{1-5}
  \text{RMA}  & 24.97 & 8.12 & 4.19 & 3.08  \\ \cline{1-5}
  \text{RMA+RMAP}  & 17.51 & 17.82 & 16.27 & 16.49  \\ \cline{1-5}
  \text{RS}  & 47.06 & 110.88 & 107.46 & 14.44  \\ \cline{1-5}
  \text{ENKF}  & 20.82 & 18.26 & 17.54 & 19.55  \\ \cline{1-5}
  \text{ALL}  & 26.13 & 15.63 & 19.45 & 19.62  \\ \cline{1-5}
  \end{tabular}
  \vspace*{0.25cm}
  \caption{Relative error of MAP solution for various randomization schemes compared to the $\umap$ solution.}
  \label{Table:Nonlinear}
\end{table}

\begin{figure}[htb!]
    \centering
    \begin{tabular}{c}
        \hspace{-1.5cm}\textbf{2D Nonlinear Elliptic}\\
        \rotatebox{90}{\hspace{2cm}Relative error}
        \includegraphics[width=0.9\textwidth, trim=1.0cm 0.9cm 0.0cm 0.8cm,clip=true,]{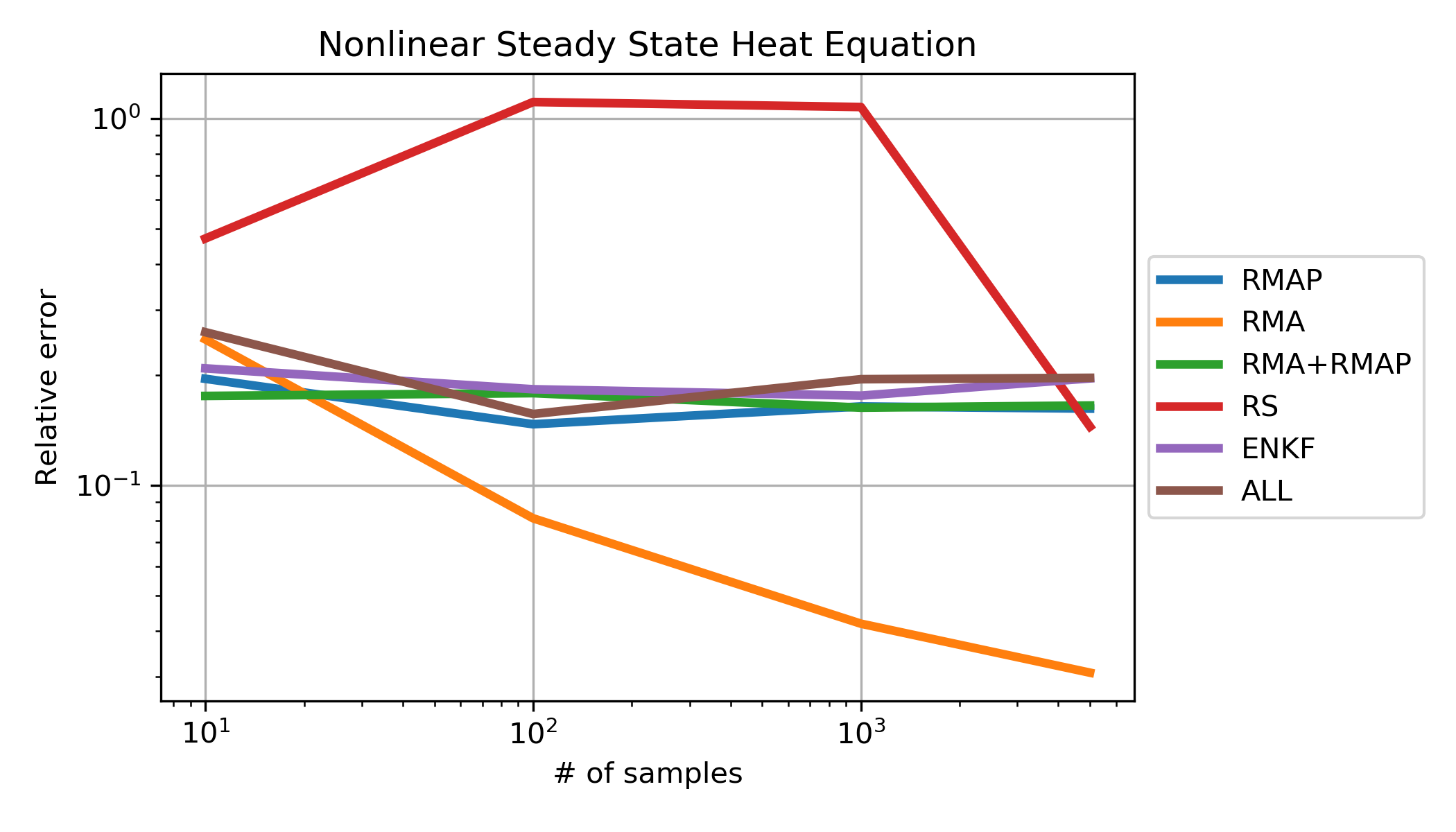} \\
        \hspace{-1.0cm}$N$
    \end{tabular}
    \caption{Relative error plot for nonlinear elliptic parameter inverse problem.}
\end{figure}

A few remarks are in order to understand the rather unimpressive
results in Table \ref{Table:Nonlinear}. 
First, recall that the results 
shown here are for an extremely difficult problem. The inverse of the diffusion 
equation is notoriously ill-posed, as it amounts to the inverse of a compact operator.
That is, small perturbations in the data can lead to drastically 
different inversion results. Indeed, we show an even more difficult 
problem where the task is to infer 4,225 parameters from only 
100 measurements which are recorded in only half of the domain as 
shown in Figure \figref{nonlinear_setup}. 
With so few measurements, adding noise to the data as in 
RMAP, RMA+RMAP, ENKF, and ALL may not lead to desirable results. 

Secondly, while there is still $\sim$20\% error for these methods, 
the MAP estimate for each of these is still reasonably good visibly as seen in Figure \figref{Nonlinear_eyeball}. 
The high relative error is in part due to the small norm of the 
$\umap$ solution. Simply shifting the parameter up by a constant 
changes the norm of the denominator in the relative error formula 
\begin{equation*}
    \text{relative error} := \norm{\ub_N^{\langle method \rangle} - \umap} / \norm{\umap}.
\end{equation*}
In addition to the numerical relative error, it is important to 
consider the ``eyeball norm''. Figure \figref{Nonlinear_eyeball} shows that the 
estimated solutions are still quite close to the $\umap$ solution, 
especially given the extreme ill-posedness. 
Despite this shortcoming of data-randomization methods, recall that the main advantage of additively randomizing the data 
and prior mean is to \textit{aid in sampling} from the posterior. 
In other words, 
we are most interested in accelerating uncertainty quantification, not 
getting the SAA MAP estimate error down to machine precision. 
These randomization schemes may still find use in such applications. 

\begin{figure}[h!t!b!]
    \begin{tabular}[h!]{c}
        \begin{subfigure}[b]{0.31\textwidth}
            \begin{tabular}{c}
            \\\textbf{$\umap$}\\
                \begin{subfigure}{\textwidth}
                    \centering
                    \includegraphics[trim=2.0cm 1.3cm 6.0cm 1.5cm,clip=true,width=\textwidth]{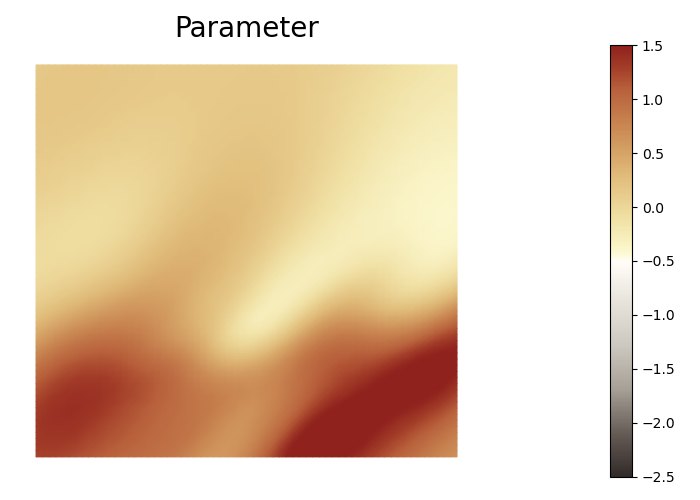}
                \end{subfigure}
            \end{tabular}
        \end{subfigure}

        \begin{subfigure}[b]{0.31\textwidth}
            \begin{tabular}{c}
            \textbf{RMAP}\\ \textbf{N = 5000}\\
                \begin{subfigure}{\textwidth}
                    \centering
                    \includegraphics[trim=2.0cm 1.3cm 6.0cm 1.5cm,clip=true,width=\textwidth]{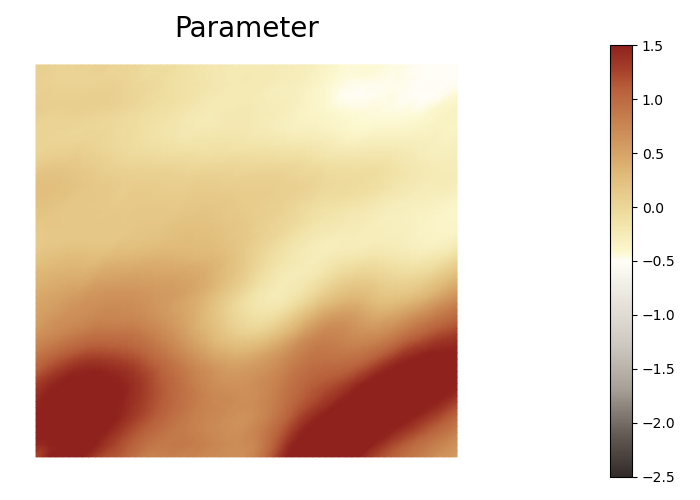}
                \end{subfigure}
            \end{tabular}
        \end{subfigure}
        
        \begin{subfigure}[b]{0.31\textwidth}
            \begin{tabular}{c}
            \textbf{ALL}\\\textbf{N = 5000}\\
                \begin{subfigure}{\textwidth}
                    \centering
                    \includegraphics[trim=2.0cm 1.3cm 6.0cm 1.5cm,clip=true,width=\textwidth]{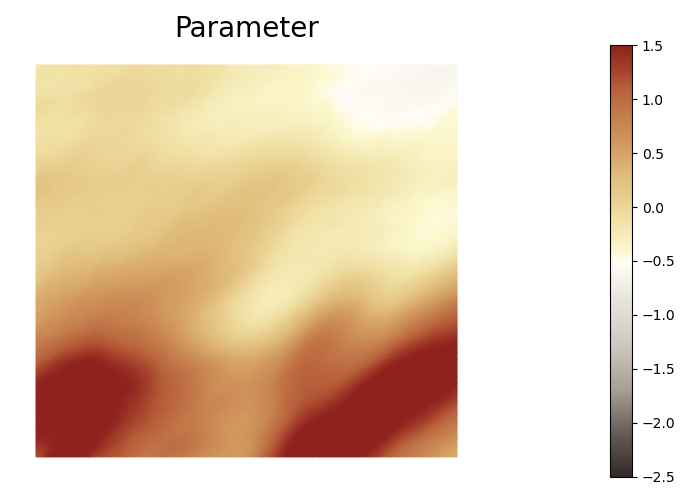}
                \end{subfigure}
            \end{tabular}
        \end{subfigure}
  \end{tabular}
    \caption{Solutions to the nonlinear diffusion inverse problem for 
    two different methods. Visually, these methods give nearly identical 
    results compared to the $\umap$ solution even though numerically, they 
    have relative error $\sim$20\%. }
    \figlab{Nonlinear_eyeball}
\end{figure}

\section{Conclusions}
By viewing the randomized solution of inverse problems through the lens of stochastic programming and
the sample average approximation, we developed a unified framework though which we can analyze the asymptotic
convergence of randomized solutions of linear inverse problems to the solution obtained with its deterministic counterpart.
This framework allowed us to prove the asymptotic and non-asymptotic convergence of the minimizer of a general stochastic cost function 
to the minimizer of the expected value of the stochastic cost function.
Several well-known methods for introducing randomness into linear and nonlinear inverse problems were recovered as special cases of
this general framework. 
Viewing the solution to randomized inverse problems through the lens 
of the sample average approximation also allowed us to prove a novel 
non-asymptotic error analysis that applies to all randomized methods discussed.
We also show that while all of the methods presented converge asymptotically, the results
can be quite poor if an insufficient number of samples are drawn. 
While this observation is easily understood through our non-asymptotic 
error analysis, it is not possible from an asymptotic view point.
In particular, we showed that randomizing the prior
covariance matrix may not be a good idea for certain priors due to the regularizing role that the prior plays
in the solution of inverse problems.  
This is due to the potentially slow convergence of random matrices to their expected value,
depending on the spectrum of the expected value matrix. 
The convergence of all schemes was shown numerically for a variety of linear and nonlinear inverse problems,
including 1D and 2D problems governed by algebraic and PDE constraints.

\section{Acknowledgements}
We would like to thank the Texas Advanced Computing Center (TACC) at the 
University of Texas at Austin for providing 
HPC resources that contributed to the results presented in this work. 
URL: http://www.tacc.utexas.edu 

\appendix
\newpage

\section{Figures}
\begin{figure}[h!t!b!]  
  \begin{subfigure}[b]{\textwidth}
    \begin{tabular}{c}
      \textbf{1D Deconvolution Problem}\\
      \begin{subfigure}[b]{0.45\textwidth}
        \begin{tabular}{c}
          \begin{subfigure}{\textwidth}
            \centering
            \includegraphics[trim=0.0cm 0.0cm 0.0cm 0.0cm,clip=true,width=\textwidth]{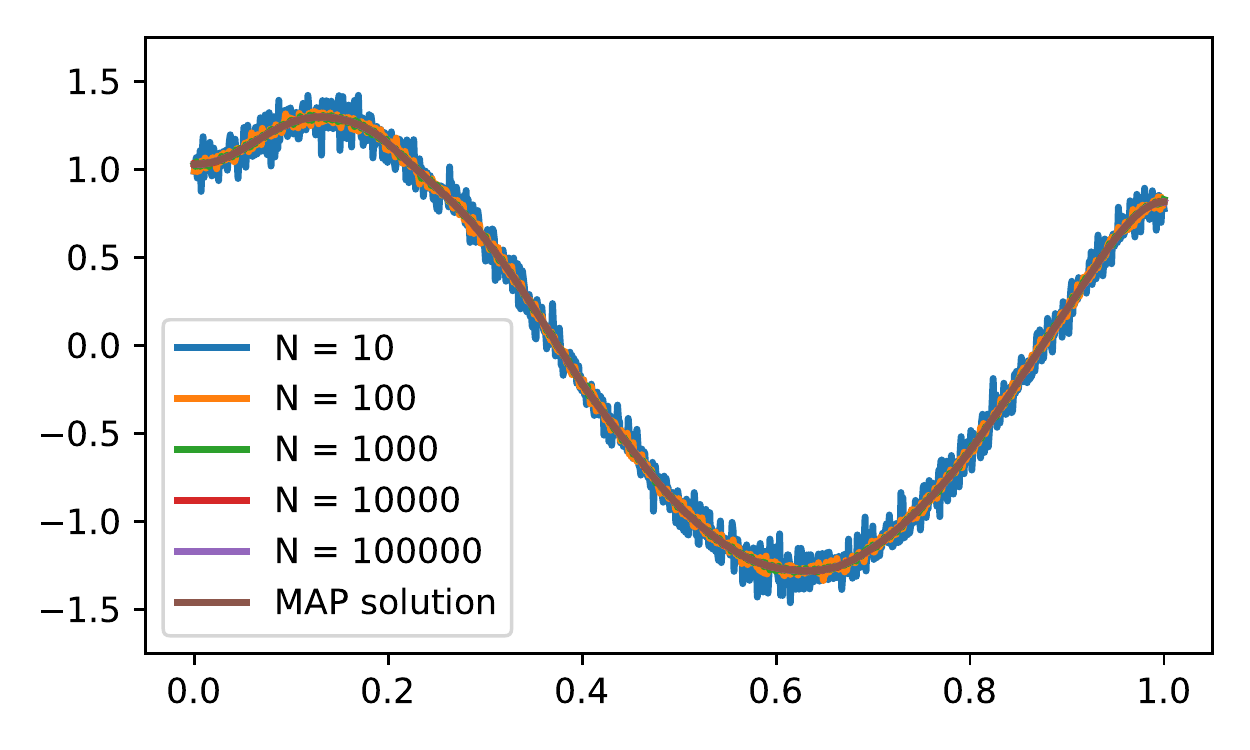}
            \caption{Randomized MAP}
            \figlab{1D_Decon_RMAP}
          \end{subfigure}
        \end{tabular}
      \end{subfigure}%
      \hfill%
      \hspace{1cm}
      \begin{subfigure}[b]{0.45\textwidth}
        \begin{tabular}{c}
          \begin{subfigure}{\textwidth}
            \centering
            \includegraphics[trim=0.0cm 0.0cm 0.0cm 0.0cm,clip=true, width=\textwidth]{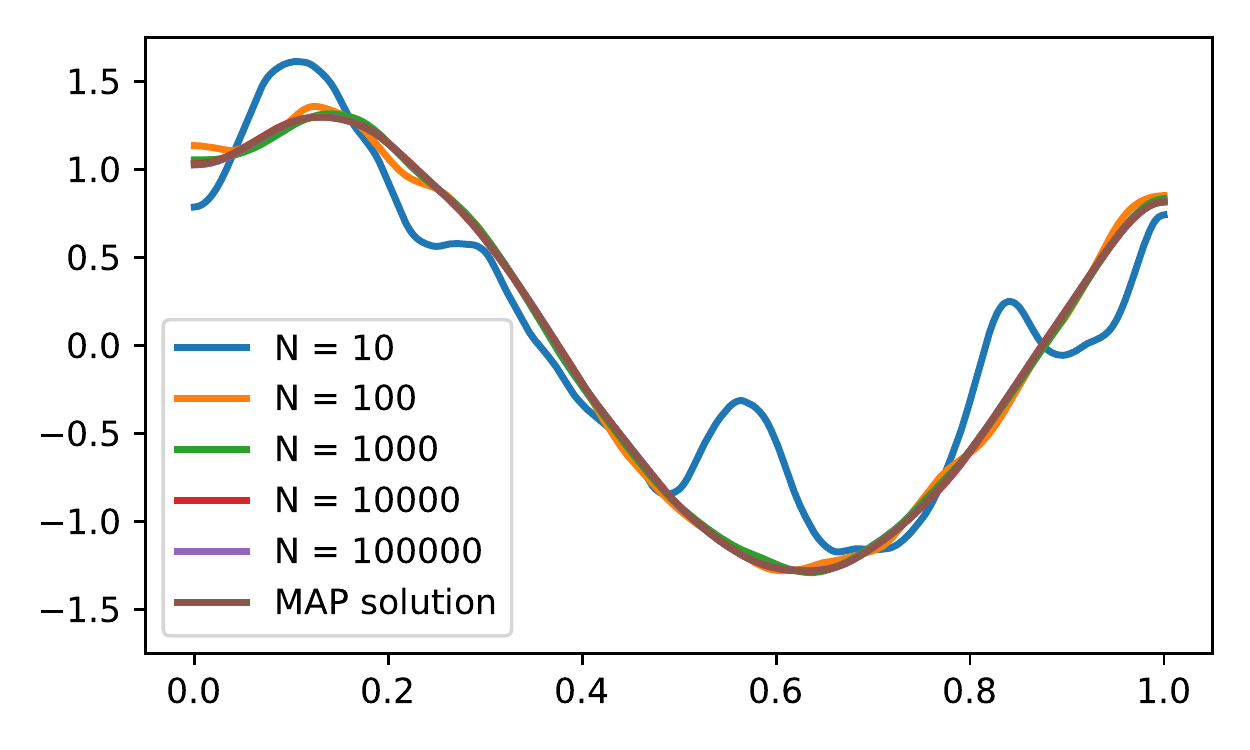}
            \caption{Randomized misfit approach}
            \figlab{1D_Decon_LS}
          \end{subfigure}
        \end{tabular}
      \end{subfigure}\\
      \begin{subfigure}[b]{0.45\textwidth}
        \begin{tabular}{c}
          \begin{subfigure}{\textwidth}
            \centering
            \includegraphics[trim=0.0cm 0.0cm 0.0cm 0.0cm,clip=true,width=\textwidth]{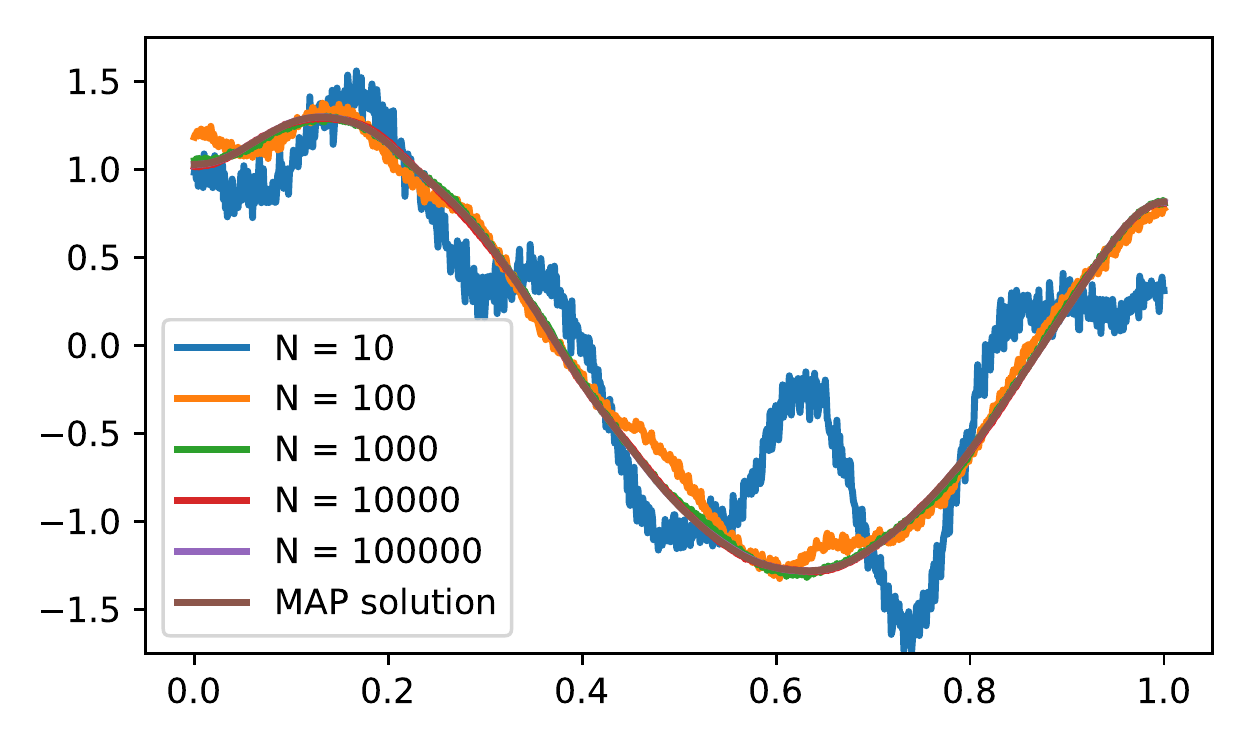}
            \caption{RMA + RMAP}
            \figlab{1D_Decon_RMA}
          \end{subfigure}
        \end{tabular}
      \end{subfigure}%
      \hfill%
      \hspace{1cm}
      \begin{subfigure}[b]{0.45\textwidth}
        \begin{tabular}{c}
          \begin{subfigure}{\textwidth}
            \centering
            \includegraphics[trim=0.0cm 0.0cm 0.0cm 0.0cm,clip=true, width=\textwidth]{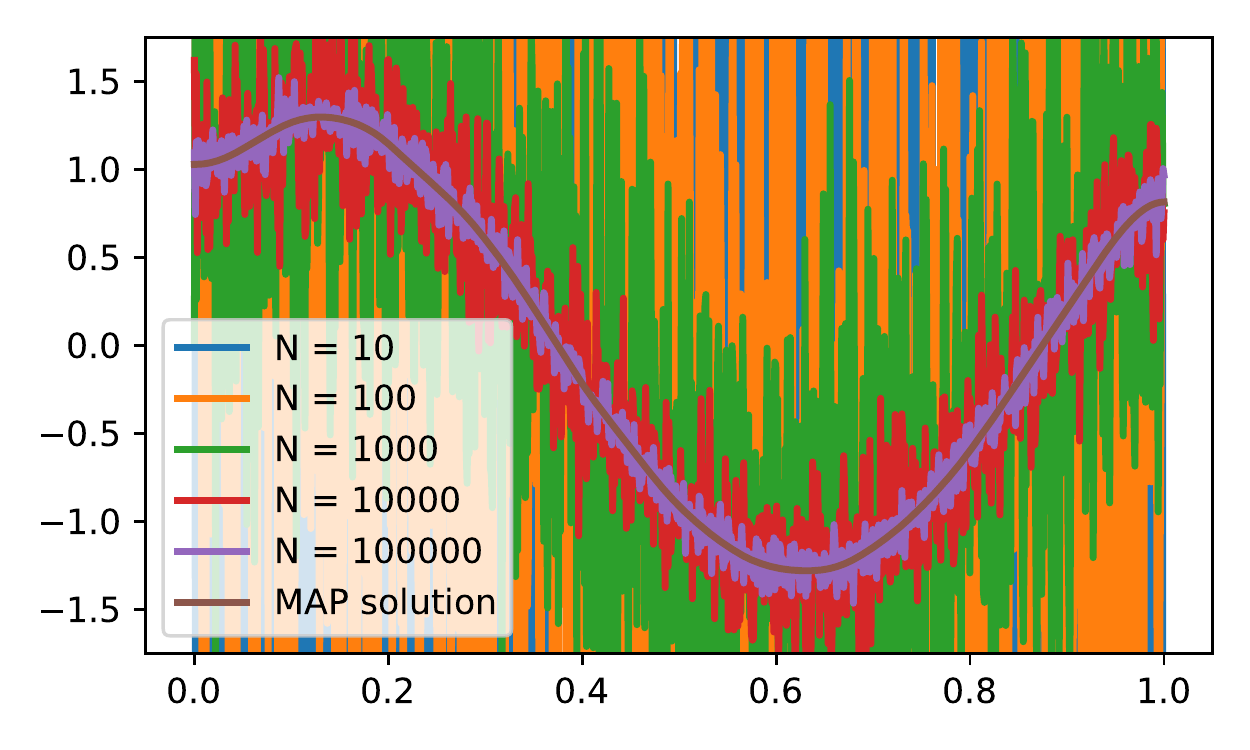}
            \caption{Right sketching}
            \figlab{1D_Decon_RS}
          \end{subfigure}
        \end{tabular}
      \end{subfigure}\\
      \begin{subfigure}[b]{0.45\textwidth}
        \begin{tabular}{c}
          \begin{subfigure}{\textwidth}
            \centering
            \includegraphics[trim=0.0cm 0.0cm 0.0cm 0.0cm,clip=true,width=\textwidth]{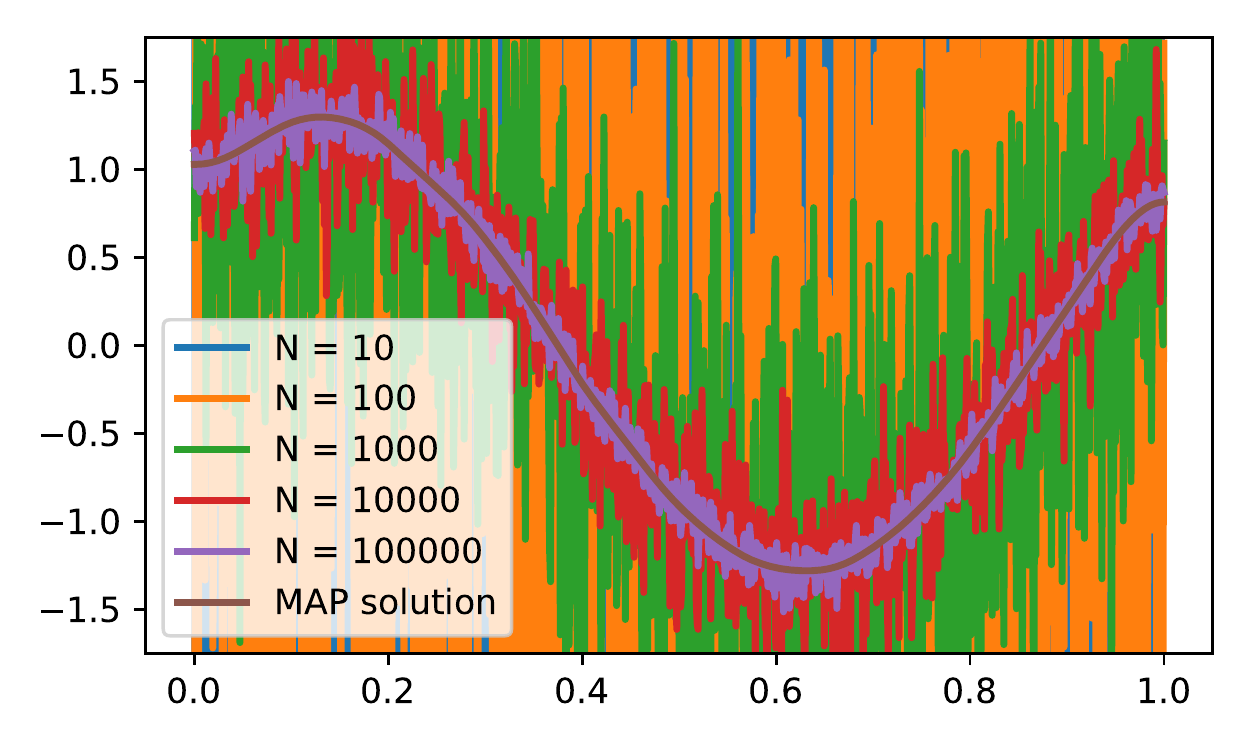}
            \caption{EnKF}
            \figlab{1D_Decon_ENKF}
          \end{subfigure}
        \end{tabular}
      \end{subfigure}%
      \hfill%
      \hspace{1cm}
      \begin{subfigure}[b]{0.45\textwidth}
        \begin{tabular}{c}
          \begin{subfigure}{\textwidth}
            \centering
            \includegraphics[trim=0.0cm 0.0cm 0.0cm 0.0cm,clip=true, width=\textwidth]{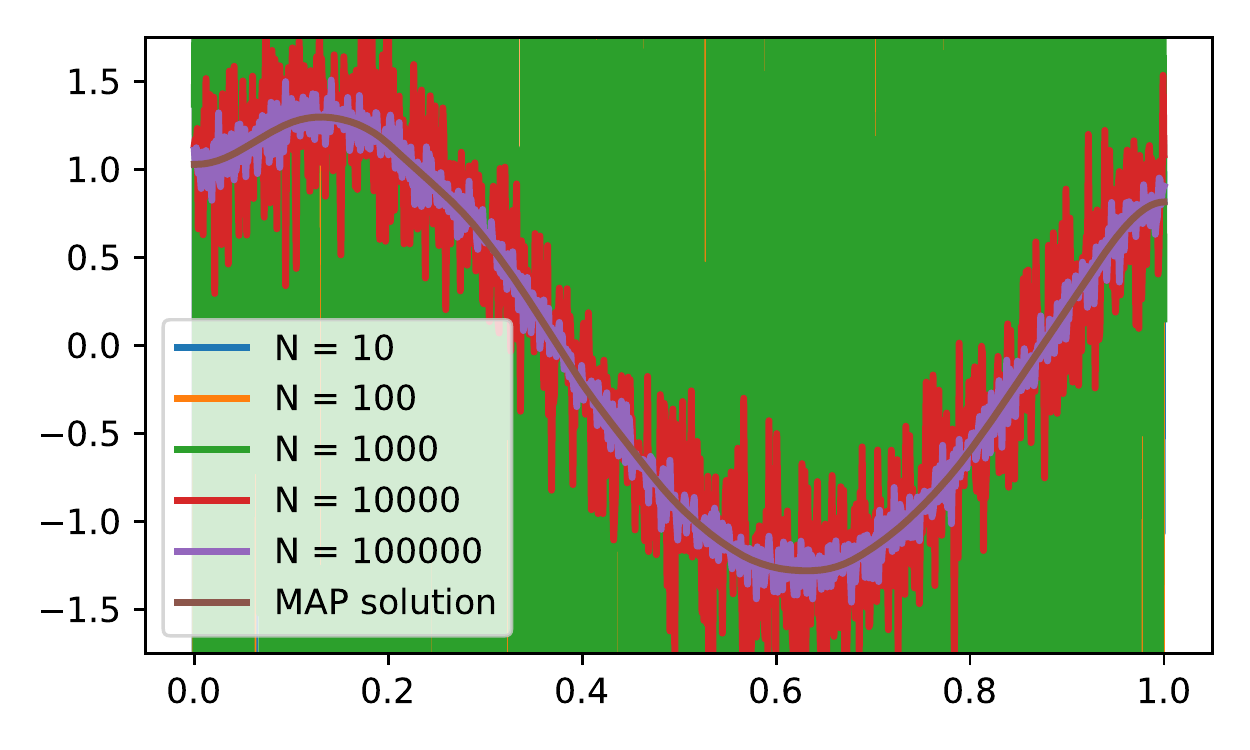}
            \caption{All}
            \figlab{1D_Decon_All}
          \end{subfigure}
        \end{tabular}
      \end{subfigure}\\
    \end{tabular}

  \end{subfigure} 
  \caption{Solutions to 1D deconvolution problem with mesh size $n = 1000$ using various randomization schemes with scaled identity prior.
    This prior works sufficiently well for those randomization schemes that do not randomize the prior covariance
    (RMAP, RMA, RMA+RMAP), but performs poorly for RS, EnKF, and ALL which randomize the prior covariance. 
    Random sampling is performed via an Achlioptas ($2/3-$sparse) random variable}
  \figlab{1D_Decon}
\end{figure}
\FloatBarrier


\newcommand{\NOne}{00010}
\newcommand{\NTwo}{00100}
\newcommand{\NThree}{01000}
\newcommand{\NFour}{50000}
\def\stripzero#1{\expandafter\stripzerohelp#1}
\def\stripzerohelp#1{\ifx 0#1\expandafter\stripzerohelp\else#1\fi}
\newcommand{\xrayWidth}{0.2}
\begin{figure}[h!t!b!]
    \begin{tabular}[h!]{c}
        \rotatebox{90}{\hspace{-0.7cm}RMAP}
        \begin{subfigure}[b]{\xrayWidth\textwidth}
            \begin{tabular}{c}
            \textbf{N = \stripzero{\NOne}}\\
                \begin{subfigure}{\textwidth}
                    \centering
                    \includegraphics[trim=1.0cm 1.0cm 1.0cm 1.0cm,clip=true,width=\textwidth]{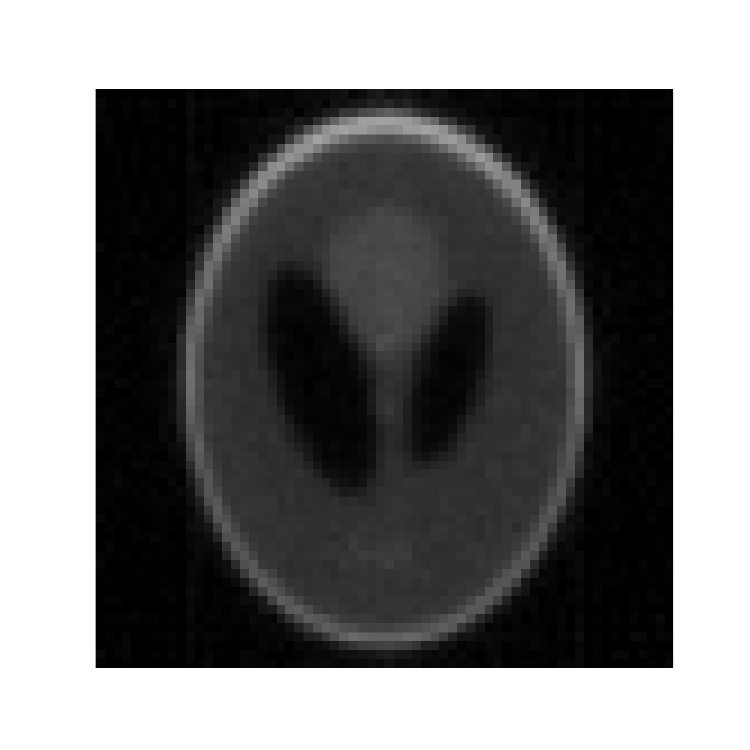}
                \end{subfigure}
            \end{tabular}
        \end{subfigure}

        \begin{subfigure}[b]{\xrayWidth\textwidth}
            \begin{tabular}{c}
            \textbf{N = \stripzero{\NTwo}}\\
                \begin{subfigure}{\textwidth}
                    \centering
                    \includegraphics[trim=1.0cm 1.0cm 1.0cm 1.0cm,clip=true,width=\textwidth]{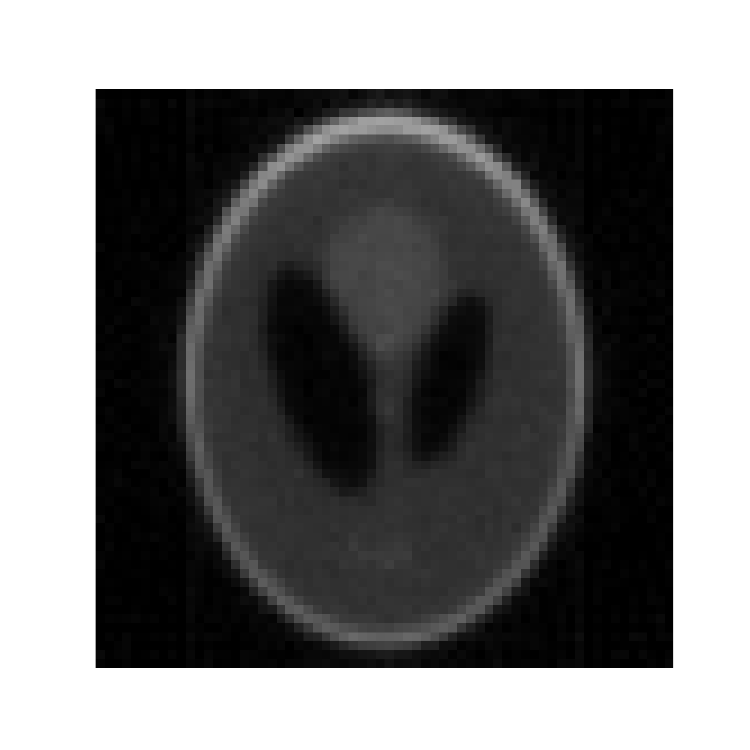}
                \end{subfigure}
            \end{tabular}
        \end{subfigure}
        
        \begin{subfigure}[b]{\xrayWidth\textwidth}
            \begin{tabular}{c}
            \textbf{N = \stripzero{\NThree}}\\
                \begin{subfigure}{\textwidth}
                    \centering
                    \includegraphics[trim=1.0cm 1.0cm 1.0cm 1.0cm,clip=true,width=\textwidth]{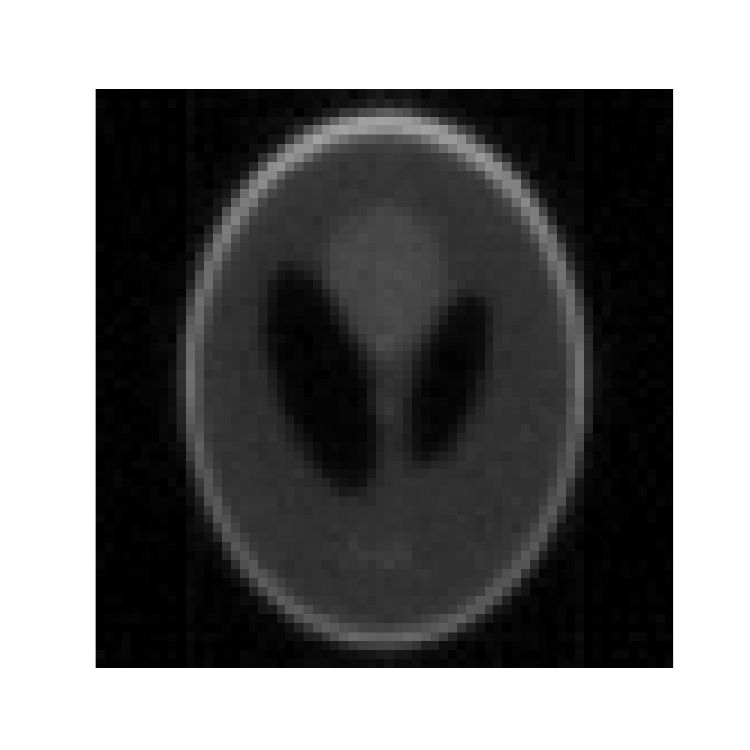}
                \end{subfigure}
            \end{tabular}
        \end{subfigure} 
        
        \begin{subfigure}[b]{\xrayWidth\textwidth}
            \begin{tabular}{c}
            \textbf{N = \stripzero{\NFour}}\\
                \begin{subfigure}{\textwidth}
                    \centering
                    \includegraphics[trim=1.0cm 1.0cm 1.0cm 1.0cm,clip=true,width=\textwidth]{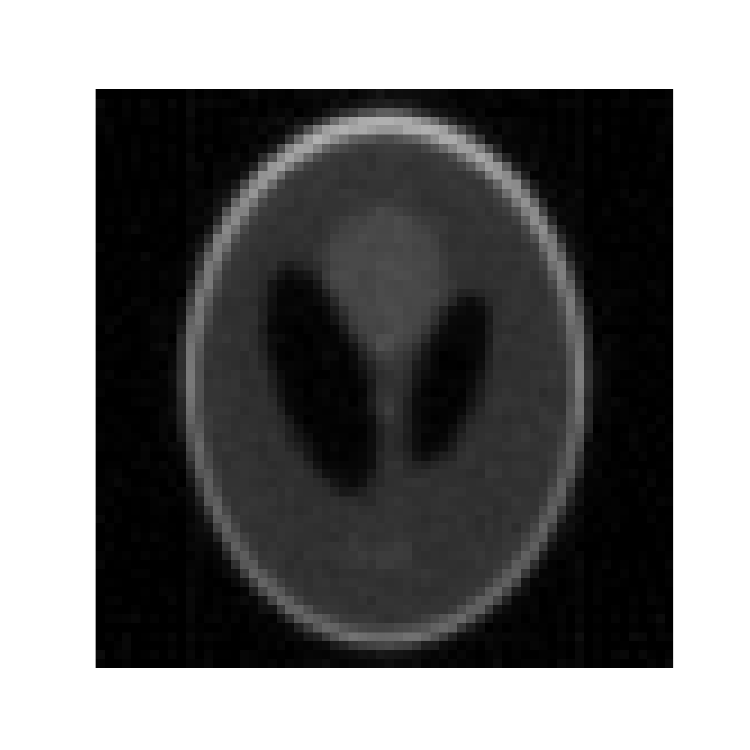}
                \end{subfigure}
            \end{tabular}
        \end{subfigure}\\
        
        \rotatebox{90}{\hspace{-0.4cm}RMA}
        \begin{subfigure}[b]{\xrayWidth\textwidth}
            \begin{tabular}{c}
                \begin{subfigure}{\textwidth}
                    \centering
                    \includegraphics[trim=1.0cm 1.0cm 1.0cm 1.0cm,clip=true,width=\textwidth]{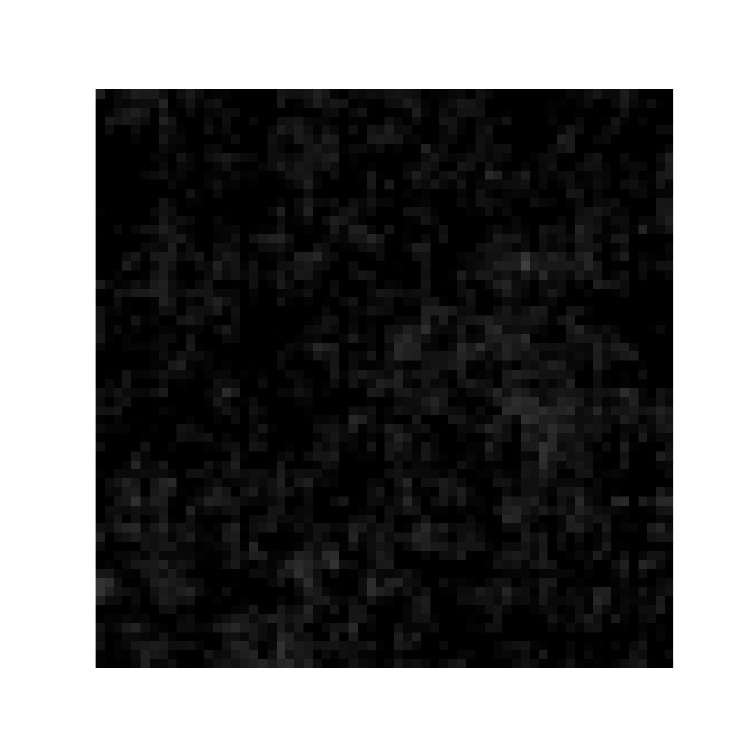}
                \end{subfigure}
            \end{tabular}
        \end{subfigure}

        \begin{subfigure}[b]{\xrayWidth\textwidth}
            \begin{tabular}{c}
                \begin{subfigure}{\textwidth}
                    \centering
                    \includegraphics[trim=1.0cm 1.0cm 1.0cm 1.0cm,clip=true,width=\textwidth]{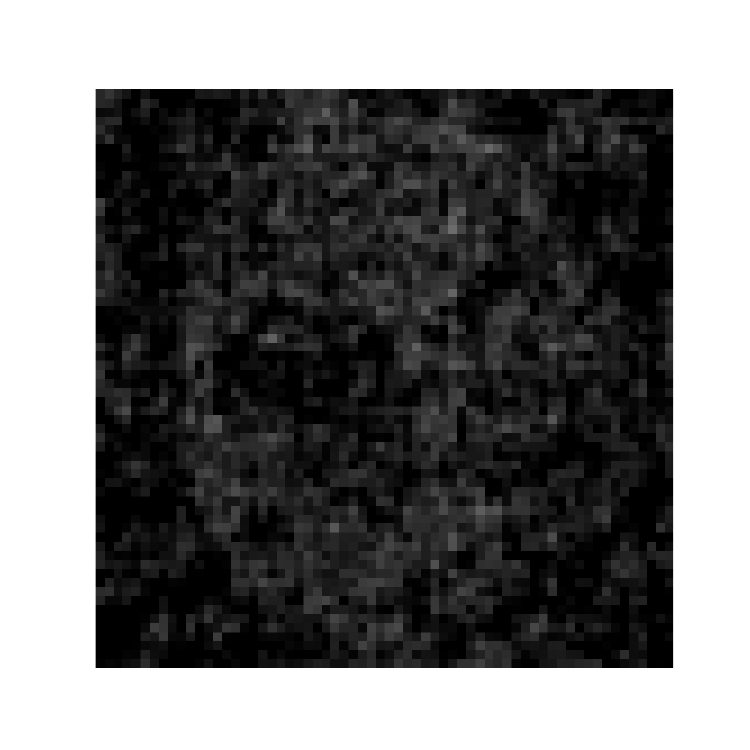}
                \end{subfigure}
            \end{tabular}
        \end{subfigure}
        
        \begin{subfigure}[b]{\xrayWidth\textwidth}
            \begin{tabular}{c}
                \begin{subfigure}{\textwidth}
                    \centering
                    \includegraphics[trim=1.0cm 1.0cm 1.0cm 1.0cm,clip=true,width=\textwidth]{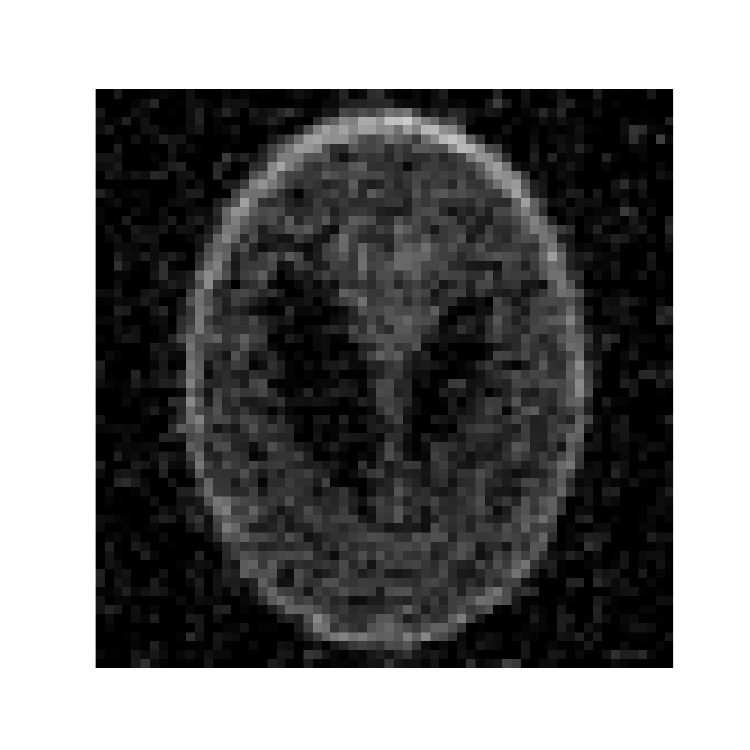}
                \end{subfigure}
            \end{tabular}
        \end{subfigure} 
        
        \begin{subfigure}[b]{\xrayWidth\textwidth}
            \begin{tabular}{c}
                \begin{subfigure}{\textwidth}
                    \centering
                    \includegraphics[trim=1.0cm 1.0cm 1.0cm 1.0cm,clip=true,width=\textwidth]{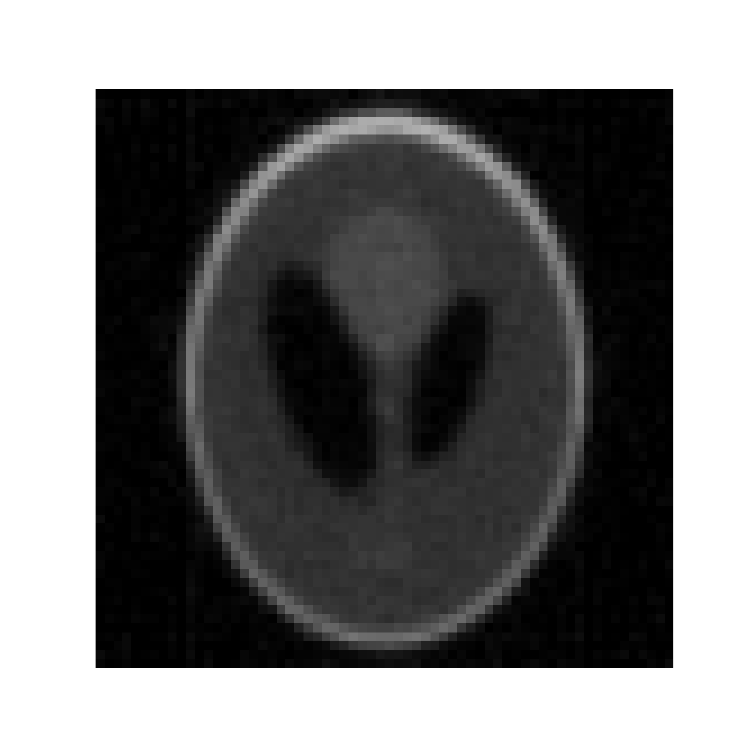}
                \end{subfigure}
            \end{tabular}
        \end{subfigure}\\
        
        \rotatebox{90}{\hspace{-0.9cm}RMA+RMAP}
        \begin{subfigure}[b]{\xrayWidth\textwidth}
            \begin{tabular}{c}
                \begin{subfigure}{\textwidth}
                    \centering
                    \includegraphics[trim=1.0cm 1.0cm 1.0cm 1.0cm,clip=true,width=\textwidth]{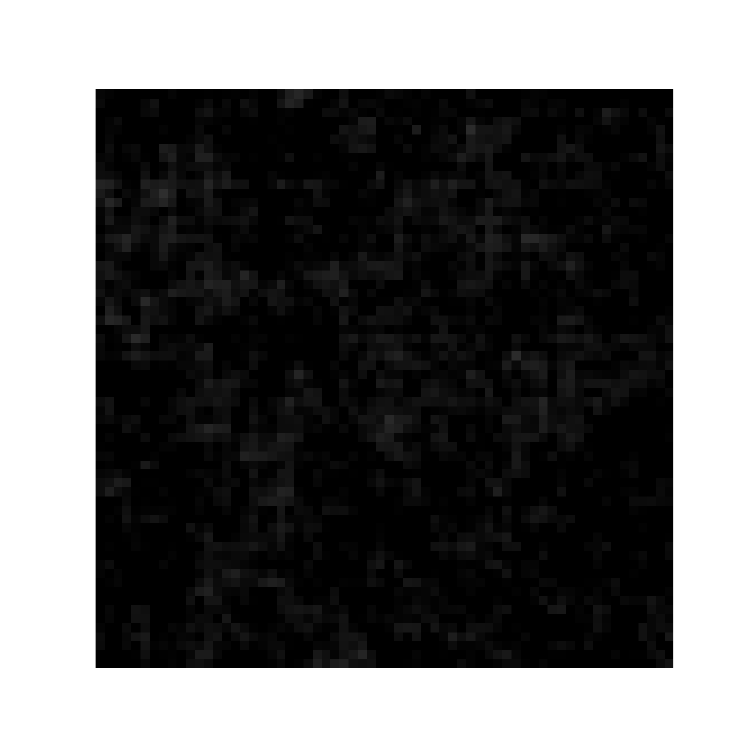}
                \end{subfigure}
            \end{tabular}
        \end{subfigure}

        \begin{subfigure}[b]{\xrayWidth\textwidth}
            \begin{tabular}{c}
                \begin{subfigure}{\textwidth}
                    \centering
                    \includegraphics[trim=1.0cm 1.0cm 1.0cm 1.0cm,clip=true,width=\textwidth]{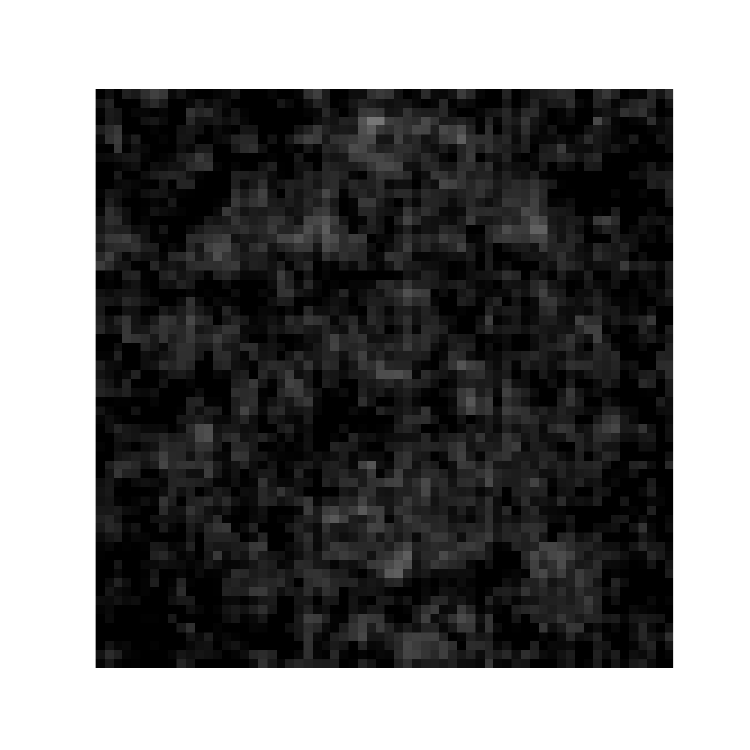}
                \end{subfigure}
            \end{tabular}
        \end{subfigure}
        
        \begin{subfigure}[b]{\xrayWidth\textwidth}
            \begin{tabular}{c}
                \begin{subfigure}{\textwidth}
                    \centering
                    \includegraphics[trim=1.0cm 1.0cm 1.0cm 1.0cm,clip=true,width=\textwidth]{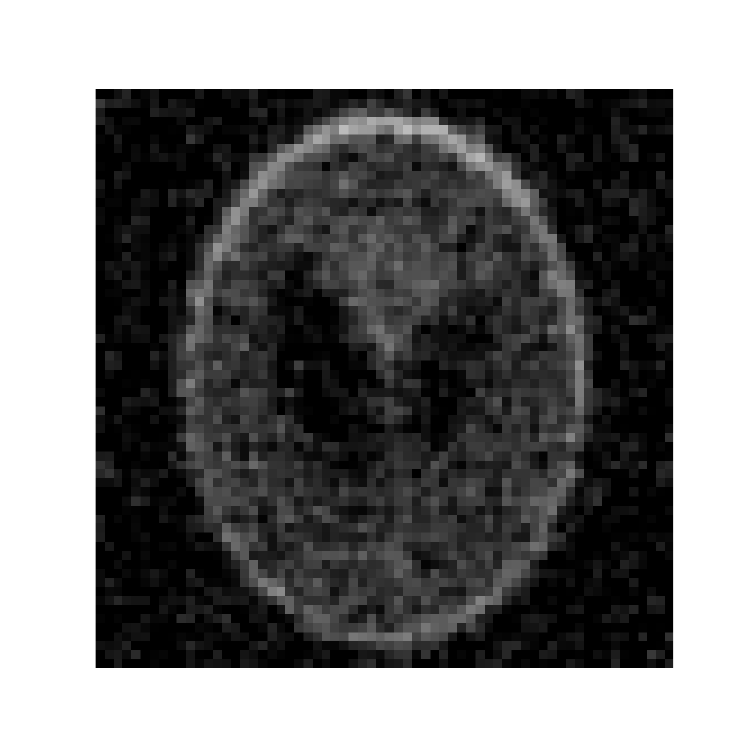}
                \end{subfigure}
            \end{tabular}
        \end{subfigure} 
        
        \begin{subfigure}[b]{\xrayWidth\textwidth}
            \begin{tabular}{c}
                \begin{subfigure}{\textwidth}
                    \centering
                    \includegraphics[trim=1.0cm 1.0cm 1.0cm 1.0cm,clip=true,width=\textwidth]{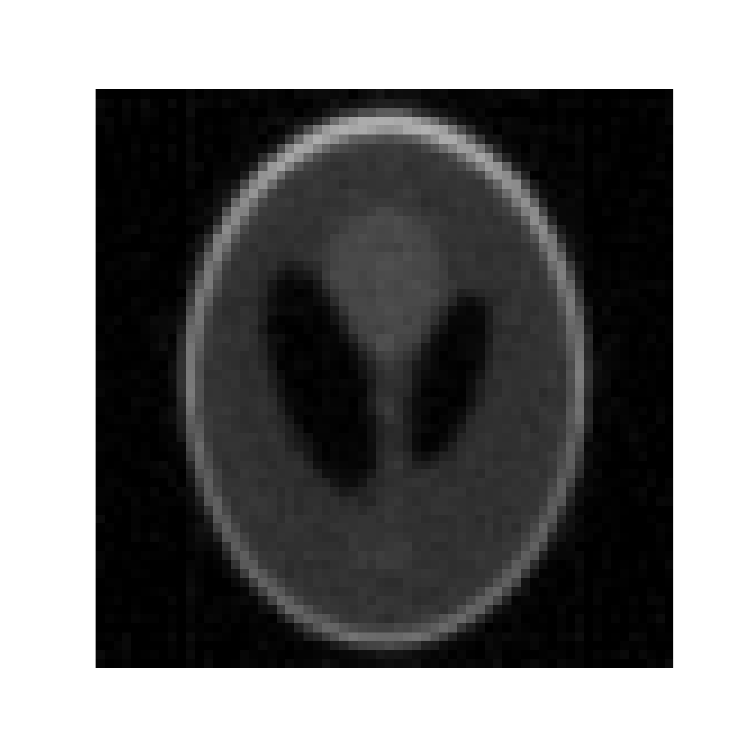}
                \end{subfigure}
            \end{tabular}
        \end{subfigure} \\

        \rotatebox{90}{\hspace{-0.2cm}RS}
        \begin{subfigure}[b]{\xrayWidth\textwidth}
            \begin{tabular}{c}
                \begin{subfigure}{\textwidth}
                    \centering
                    \includegraphics[trim=1.0cm 1.0cm 1.0cm 1.0cm,clip=true,width=\textwidth]{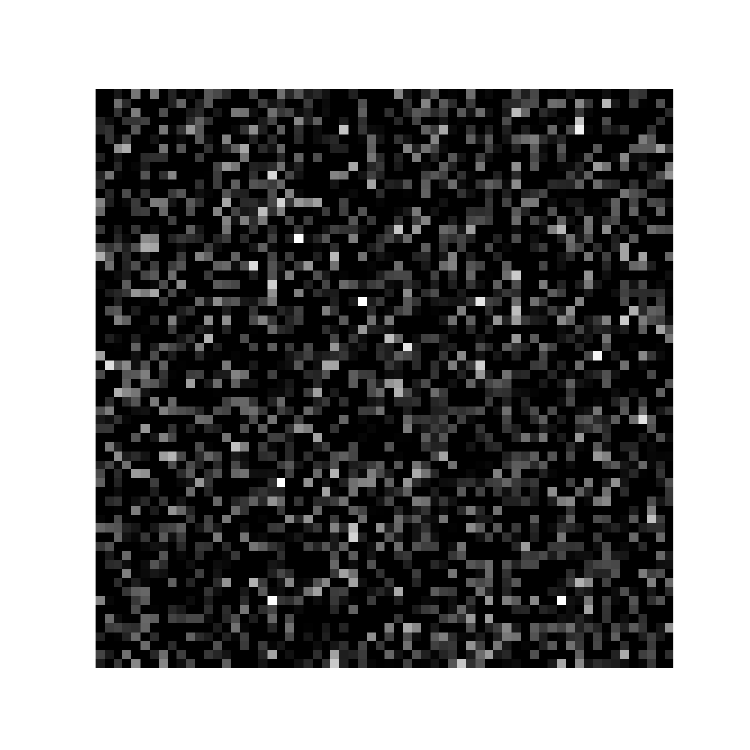}
                \end{subfigure}
            \end{tabular}
        \end{subfigure}

        \begin{subfigure}[b]{\xrayWidth\textwidth}
            \begin{tabular}{c}
                \begin{subfigure}{\textwidth}
                    \centering
                    \includegraphics[trim=1.0cm 1.0cm 1.0cm 1.0cm,clip=true,width=\textwidth]{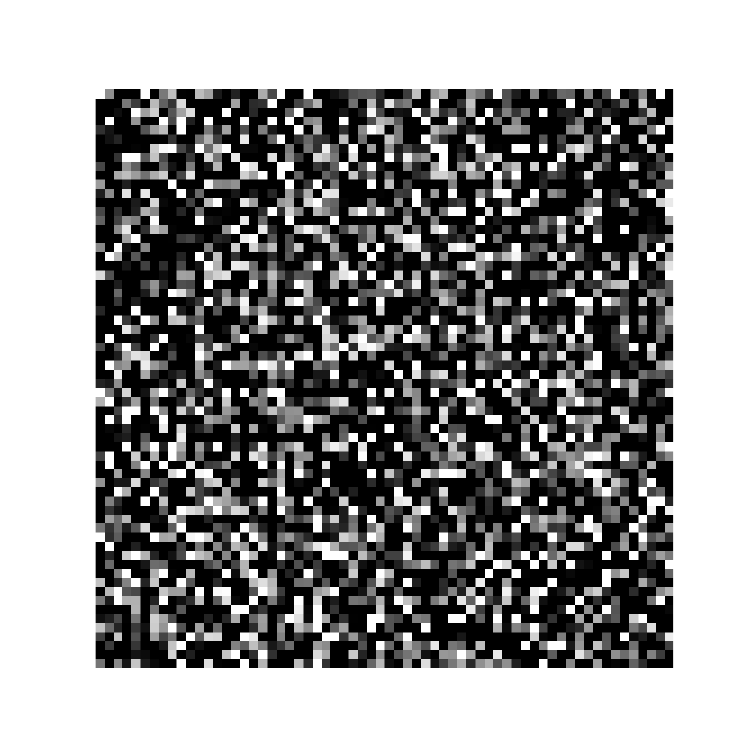}
                \end{subfigure}
            \end{tabular}
        \end{subfigure}
        
        \begin{subfigure}[b]{\xrayWidth\textwidth}
            \begin{tabular}{c}
                \begin{subfigure}{\textwidth}
                    \centering
                    \includegraphics[trim=1.0cm 1.0cm 1.0cm 1.0cm,clip=true,width=\textwidth]{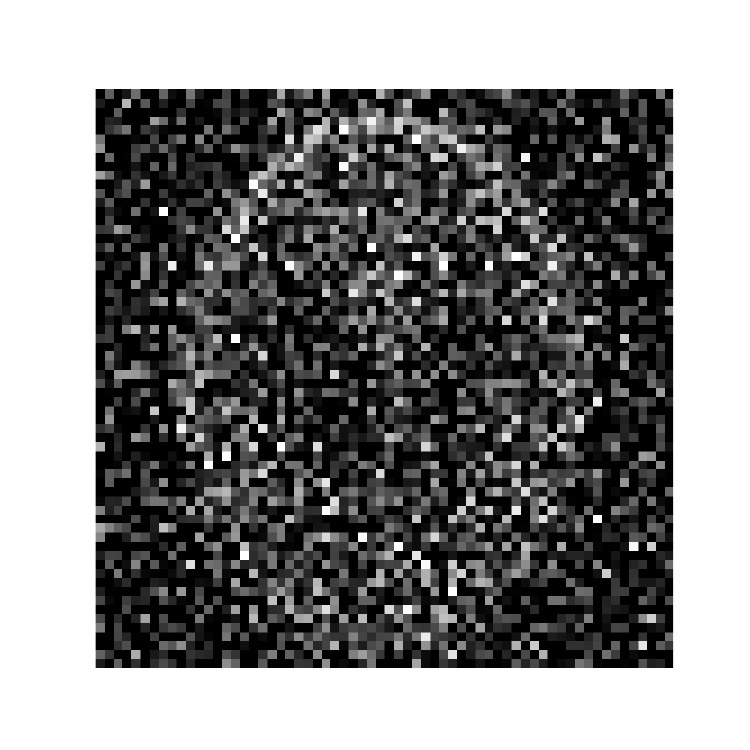}
                \end{subfigure}
            \end{tabular}
        \end{subfigure} 
        
        \begin{subfigure}[b]{\xrayWidth\textwidth}
            \begin{tabular}{c}
                \begin{subfigure}{\textwidth}
                    \centering
                    \includegraphics[trim=1.0cm 1.0cm 1.0cm 1.0cm,clip=true,width=\textwidth]{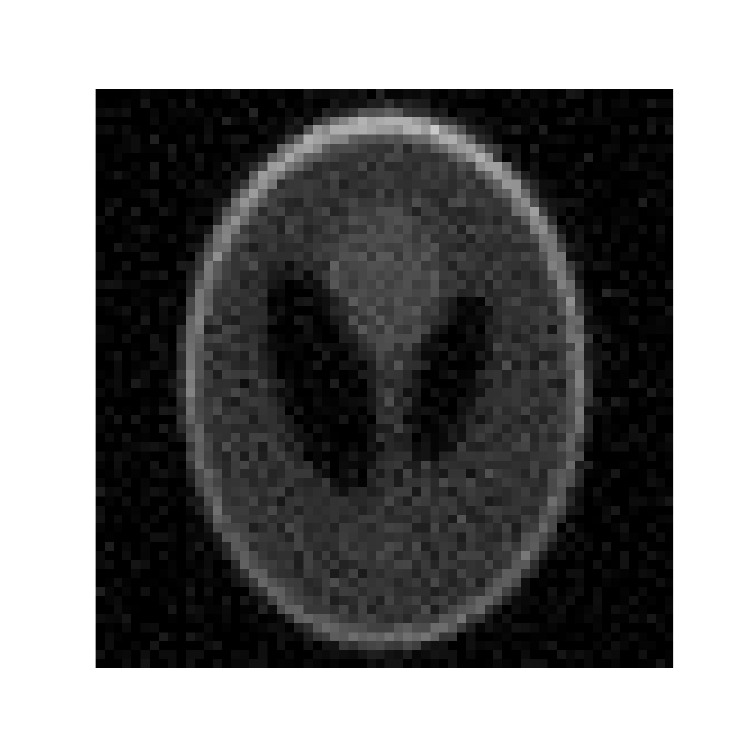}
                \end{subfigure}
            \end{tabular}
        \end{subfigure}\\

        \rotatebox{90}{\hspace{-0.5cm}ENKF}
        \begin{subfigure}[b]{\xrayWidth\textwidth}
            \begin{tabular}{c}
                \begin{subfigure}{\textwidth}
                    \centering
                    \includegraphics[trim=1.0cm 1.0cm 1.0cm 1.0cm,clip=true,width=\textwidth]{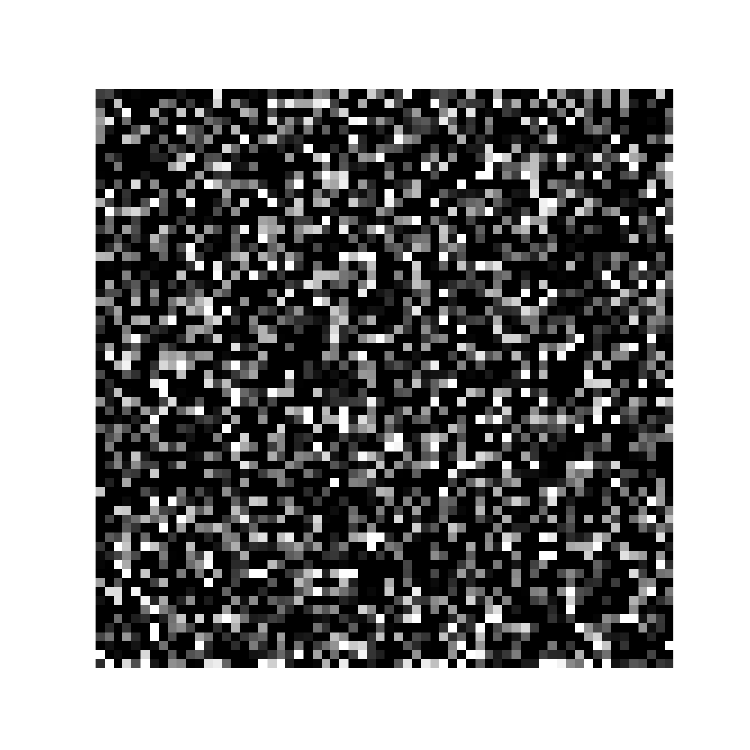}
                \end{subfigure}
            \end{tabular}
        \end{subfigure}

        \begin{subfigure}[b]{\xrayWidth\textwidth}
            \begin{tabular}{c}
                \begin{subfigure}{\textwidth}
                    \centering
                    \includegraphics[trim=1.0cm 1.0cm 1.0cm 1.0cm,clip=true,width=\textwidth]{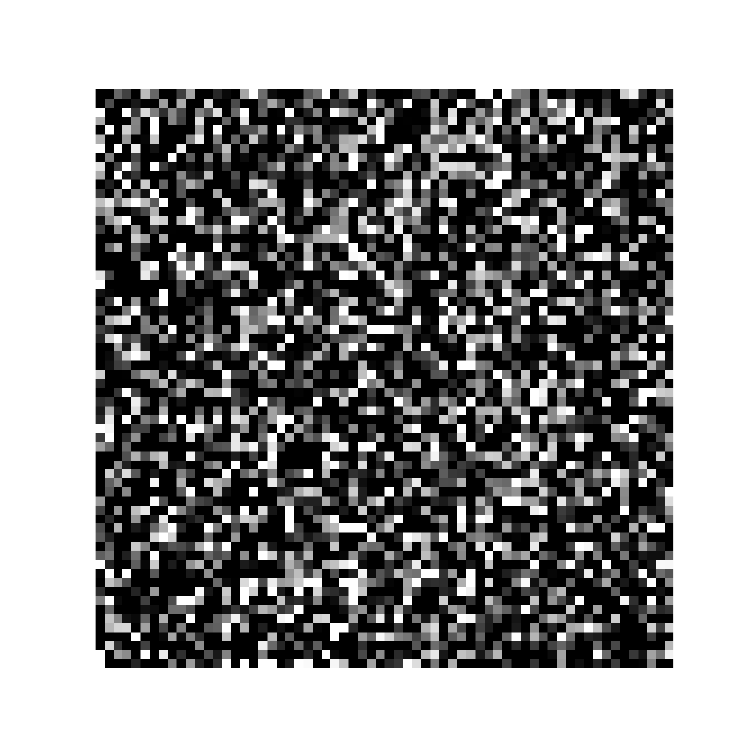}
                \end{subfigure}
            \end{tabular}
        \end{subfigure}
        
        \begin{subfigure}[b]{\xrayWidth\textwidth}
            \begin{tabular}{c}
                \begin{subfigure}{\textwidth}
                    \centering
                    \includegraphics[trim=1.0cm 1.0cm 1.0cm 1.0cm,clip=true,width=\textwidth]{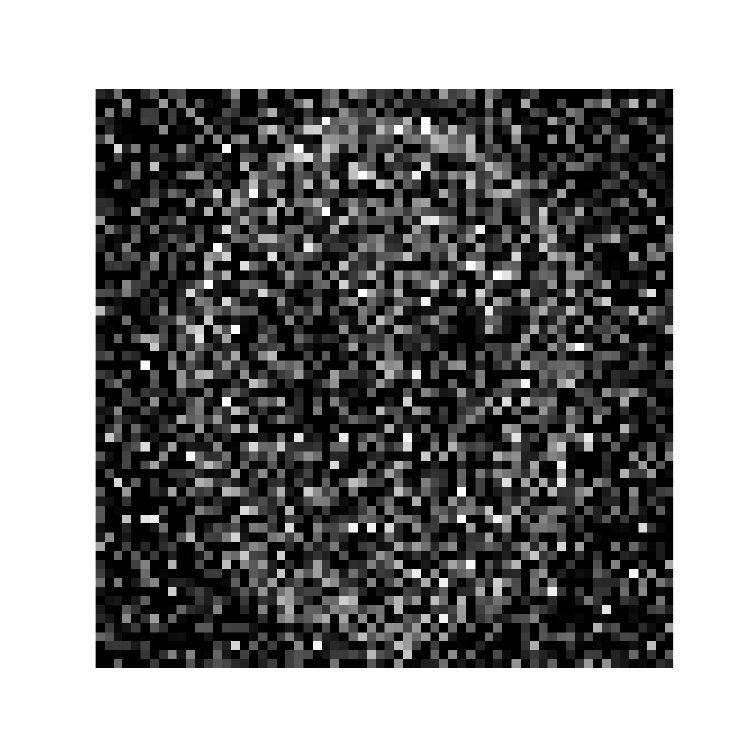}
                \end{subfigure}
            \end{tabular}
        \end{subfigure} 
        
        \begin{subfigure}[b]{\xrayWidth\textwidth}
            \begin{tabular}{c}
                \begin{subfigure}{\textwidth}
                    \centering
                    \includegraphics[trim=1.0cm 1.0cm 1.0cm 1.0cm,clip=true,width=\textwidth]{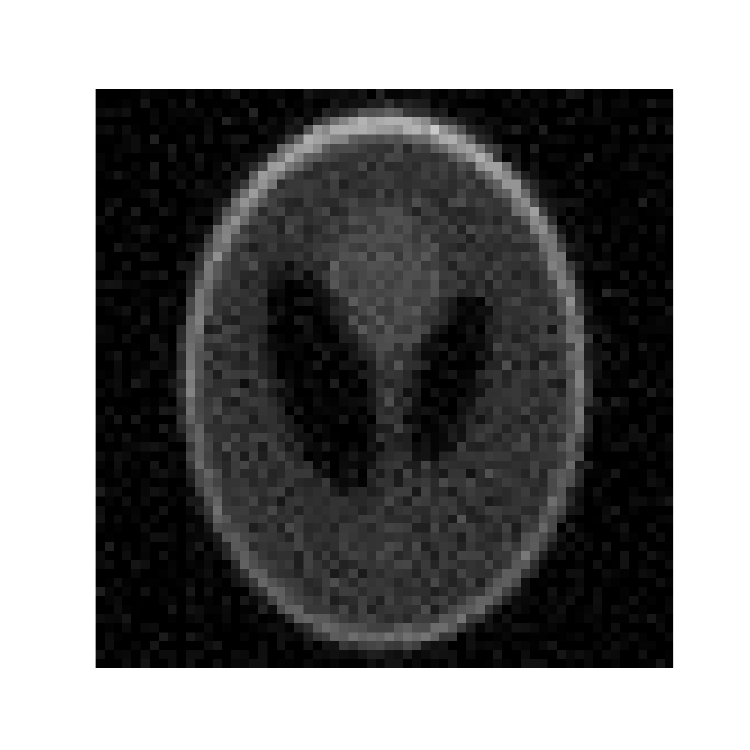}
                \end{subfigure}
            \end{tabular}
        \end{subfigure}\\
        
        \rotatebox{90}{\hspace{-0.3cm}ALL}
        \begin{subfigure}[b]{\xrayWidth\textwidth}
            \begin{tabular}{c}
                \begin{subfigure}{\textwidth}
                    \centering
                    \includegraphics[trim=1.0cm 1.0cm 1.0cm 1.0cm,clip=true,width=\textwidth]{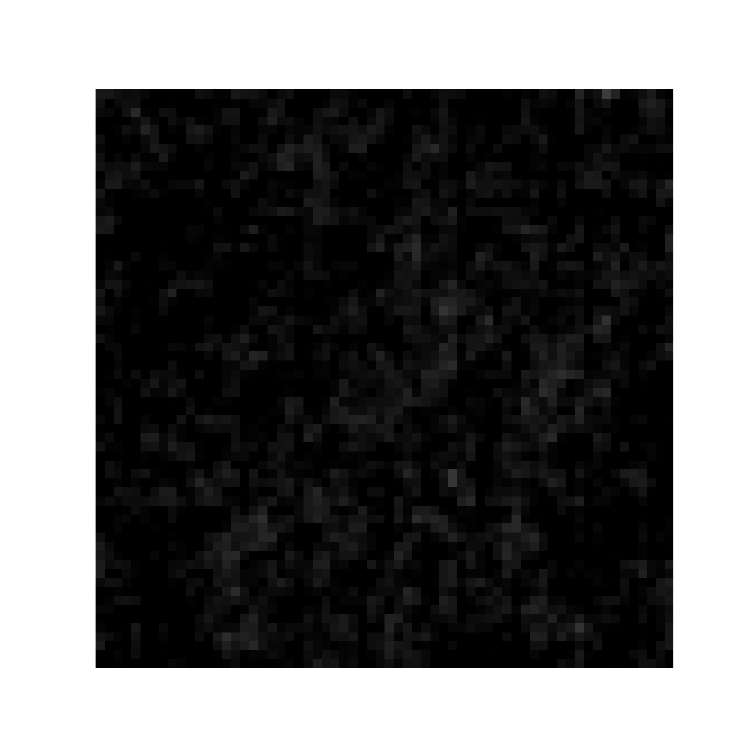}
                \end{subfigure}
            \end{tabular}
        \end{subfigure}

        \begin{subfigure}[b]{\xrayWidth\textwidth}
            \begin{tabular}{c}
                \begin{subfigure}{\textwidth}
                    \centering
                    \includegraphics[trim=1.0cm 1.0cm 1.0cm 1.0cm,clip=true,width=\textwidth]{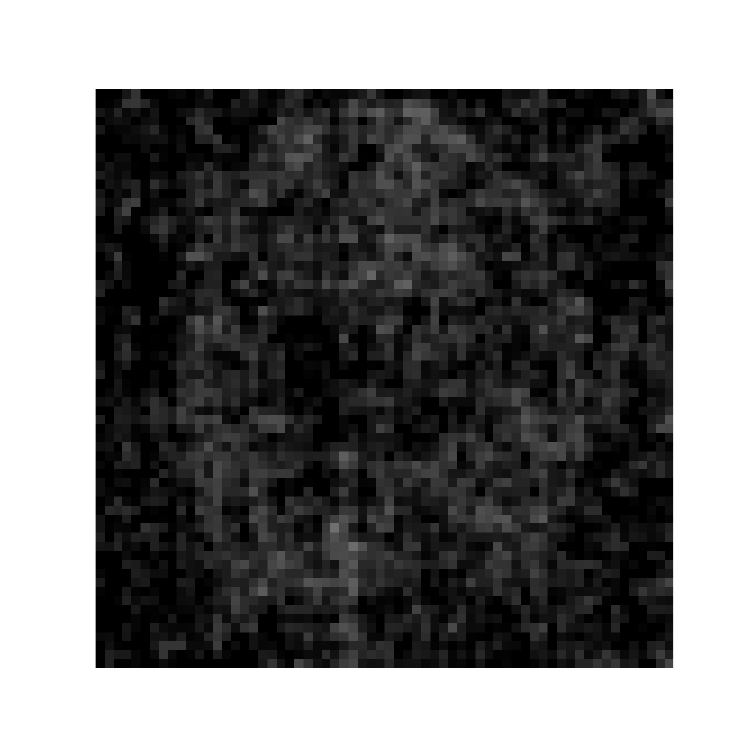}
                \end{subfigure}
            \end{tabular}
        \end{subfigure}
        
        \begin{subfigure}[b]{\xrayWidth\textwidth}
            \begin{tabular}{c}
                \begin{subfigure}{\textwidth}
                    \centering
                    \includegraphics[trim=1.0cm 1.0cm 1.0cm 1.0cm,clip=true,width=\textwidth]{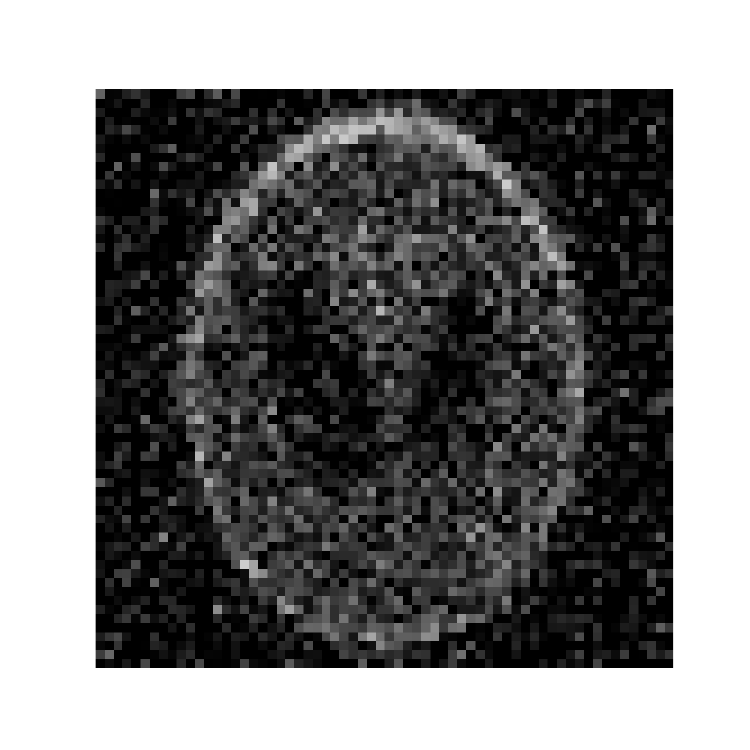}
                \end{subfigure}
            \end{tabular}
        \end{subfigure} 
        
        \begin{subfigure}[b]{\xrayWidth\textwidth}
            \begin{tabular}{c}
                \begin{subfigure}{\textwidth}
                    \centering
                    \includegraphics[trim=1.0cm 1.0cm 1.0cm 1.0cm,clip=true,width=\textwidth]{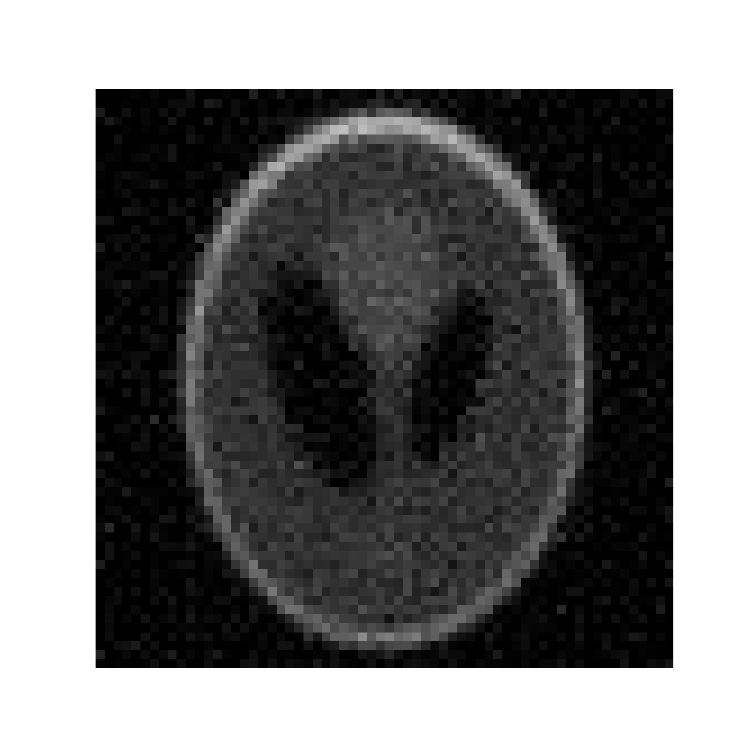}
                \end{subfigure}
            \end{tabular}
        \end{subfigure}\\

  \end{tabular}
    \caption{Solutions for various randomization approaches for an x-ray tomography problem with Gaussian random variables.}
    \label{X_ray}
\end{figure}

\newcommand{\DNOne}{00010}
\newcommand{\DNTwo}{00100}
\newcommand{\DNThree}{01000}
\newcommand{\DNFour}{10000}
\def\stripzero#1{\expandafter\stripzerohelp#1}
\def\stripzerohelp#1{\ifx 0#1\expandafter\stripzerohelp\else#1\fi}
\begin{figure}[h!t!b!]
    \begin{tabular}[h!]{c}
        \rotatebox{90}{\hspace{-0.7cm}RMAP}
        \begin{subfigure}[b]{0.22\textwidth}
            \begin{tabular}{c}
            \textbf{N = \stripzero{\DNOne}}\\
                \begin{subfigure}{\textwidth}
                    \centering
                    \includegraphics[trim=2.5cm 2.5cm 7.0cm 1.5cm,clip=true,width=\textwidth]{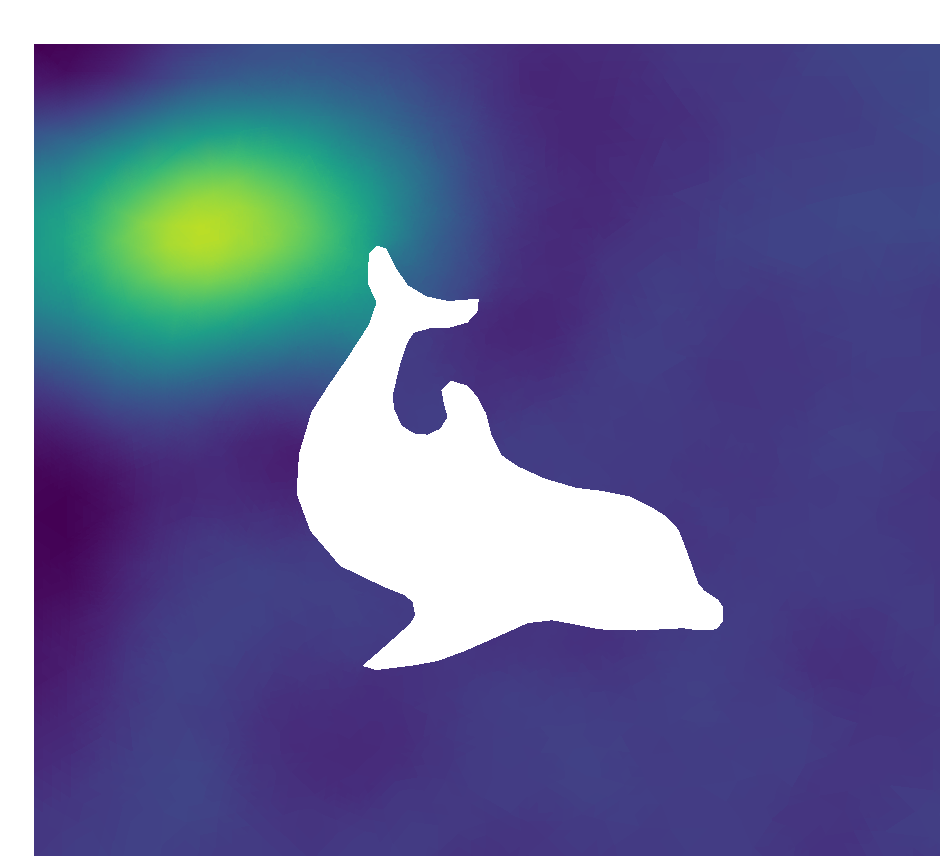}
                \end{subfigure}
            \end{tabular}
        \end{subfigure}

        \begin{subfigure}[b]{0.22\textwidth}
            \begin{tabular}{c}
            \textbf{N = \stripzero{\DNTwo}}\\
                \begin{subfigure}{\textwidth}
                    \centering
                    \includegraphics[trim=2.5cm 2.5cm 7.0cm 1.5cm,clip=true,width=\textwidth]{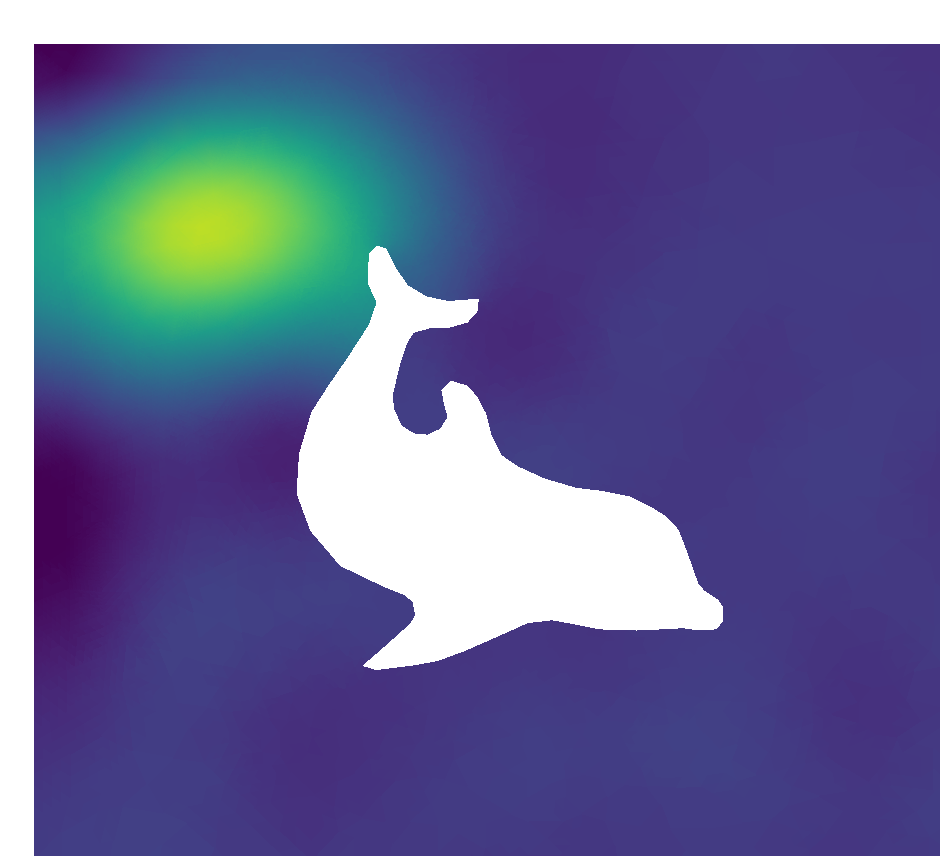}
                \end{subfigure}
            \end{tabular}
        \end{subfigure}
        
        \begin{subfigure}[b]{0.22\textwidth}
            \begin{tabular}{c}
            \textbf{N = \stripzero{\DNThree}}\\
                \begin{subfigure}{\textwidth}
                    \centering
                    \includegraphics[trim=2.5cm 2.5cm 7.0cm 1.5cm,clip=true,width=\textwidth]{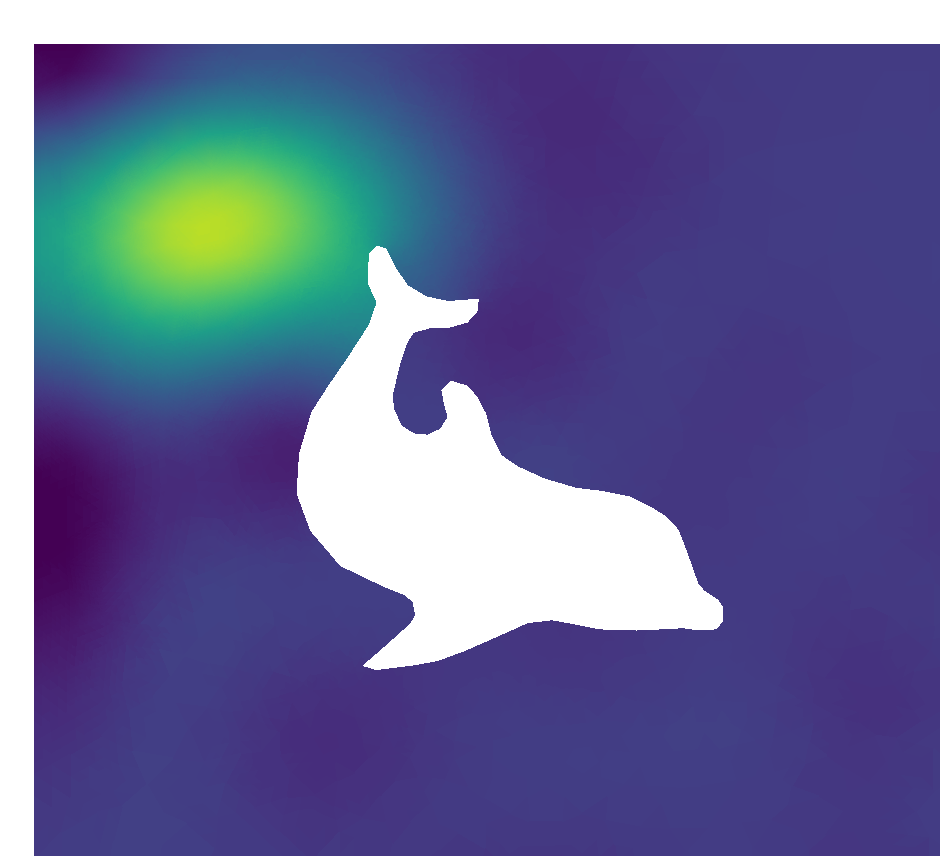}
                \end{subfigure}
            \end{tabular}
        \end{subfigure} 
        
        \begin{subfigure}[b]{0.22\textwidth}
            \begin{tabular}{c}
            \textbf{N = \stripzero{\DNFour}}\\
                \begin{subfigure}{\textwidth}
                    \centering
                    \includegraphics[trim=2.5cm 2.5cm 7.0cm 1.5cm,clip=true,width=\textwidth]{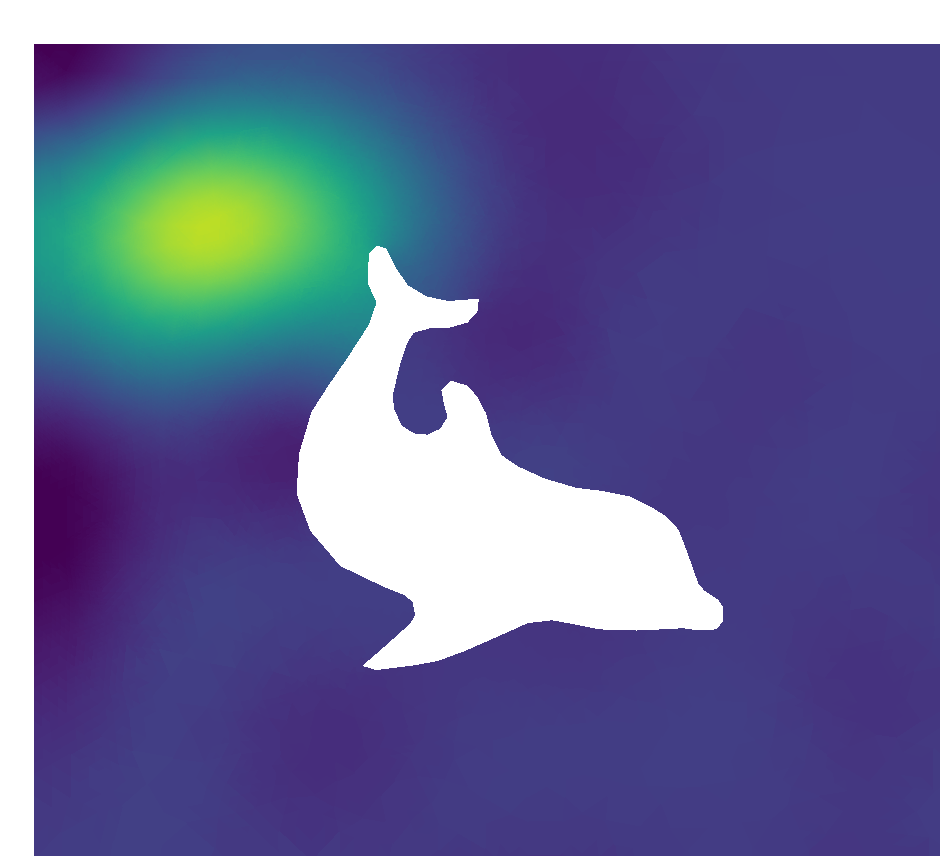}
                \end{subfigure}
            \end{tabular}
        \end{subfigure}\\
        
        \rotatebox{90}{\hspace{-0.4cm}RMA}
        \begin{subfigure}[b]{0.22\textwidth}
            \begin{tabular}{c}
                \begin{subfigure}{\textwidth}
                    \centering
                    \includegraphics[trim=2.5cm 2.5cm 7.0cm 1.5cm,clip=true,width=\textwidth]{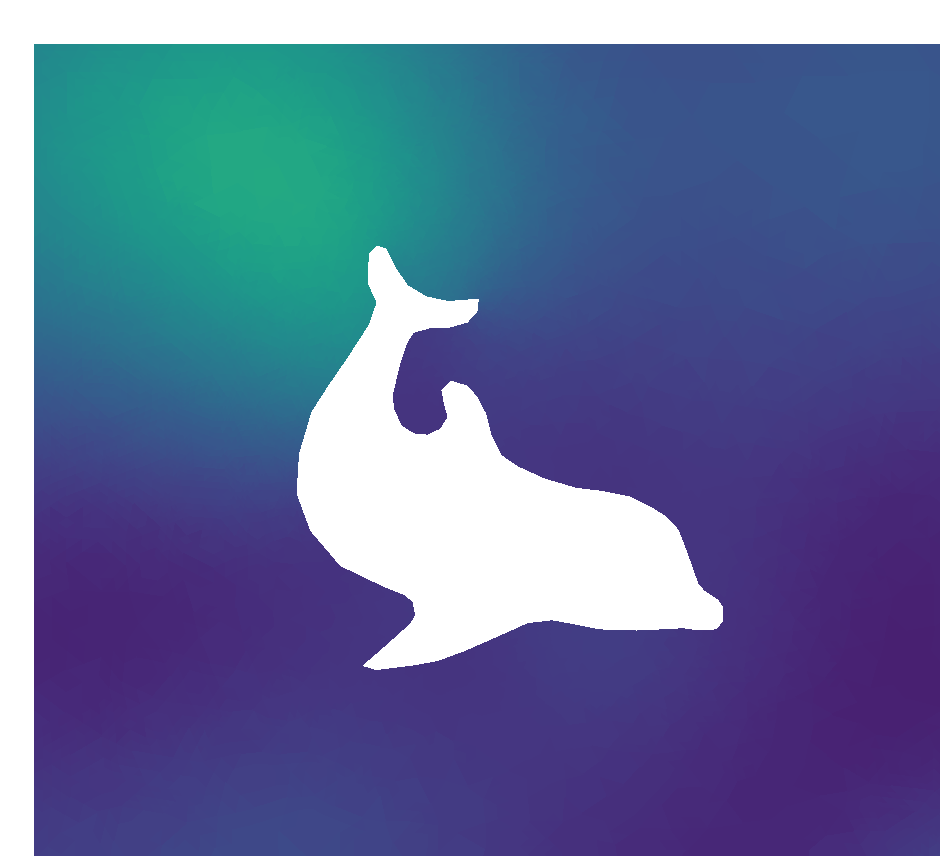}
                \end{subfigure}
            \end{tabular}
        \end{subfigure}

        \begin{subfigure}[b]{0.22\textwidth}
            \begin{tabular}{c}
                \begin{subfigure}{\textwidth}
                    \centering
                    \includegraphics[trim=2.5cm 2.5cm 7.0cm 1.5cm,clip=true,width=\textwidth]{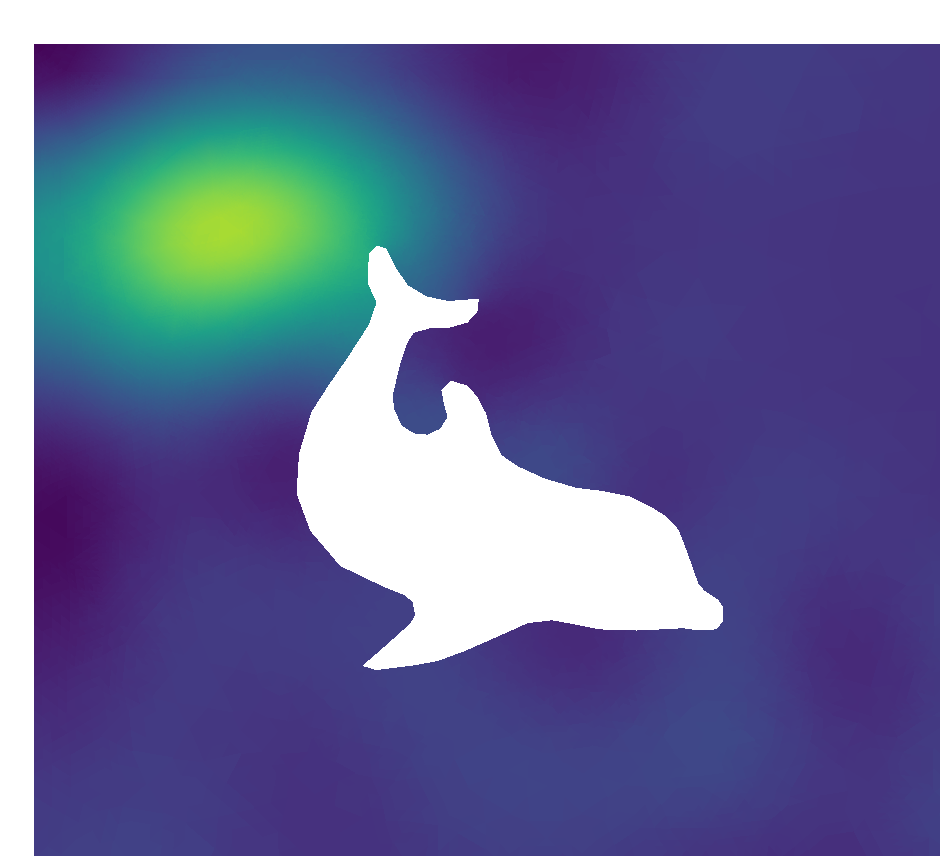}
                \end{subfigure}
            \end{tabular}
        \end{subfigure}
        
        \begin{subfigure}[b]{0.22\textwidth}
            \begin{tabular}{c}
                \begin{subfigure}{\textwidth}
                    \centering
                    \includegraphics[trim=2.5cm 2.5cm 7.0cm 1.5cm,clip=true,width=\textwidth]{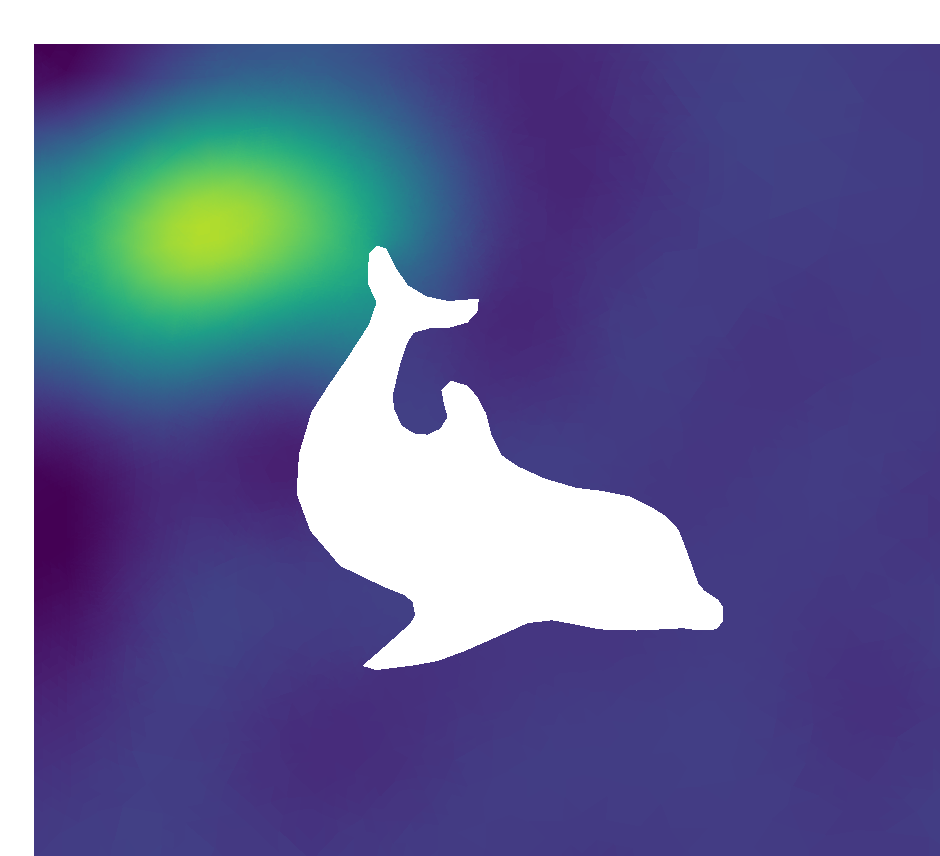}
                \end{subfigure}
            \end{tabular}
        \end{subfigure} 
        
        \begin{subfigure}[b]{0.22\textwidth}
            \begin{tabular}{c}
                \begin{subfigure}{\textwidth}
                    \centering
                    \includegraphics[trim=2.5cm 2.5cm 7.0cm 1.5cm,clip=true,width=\textwidth]{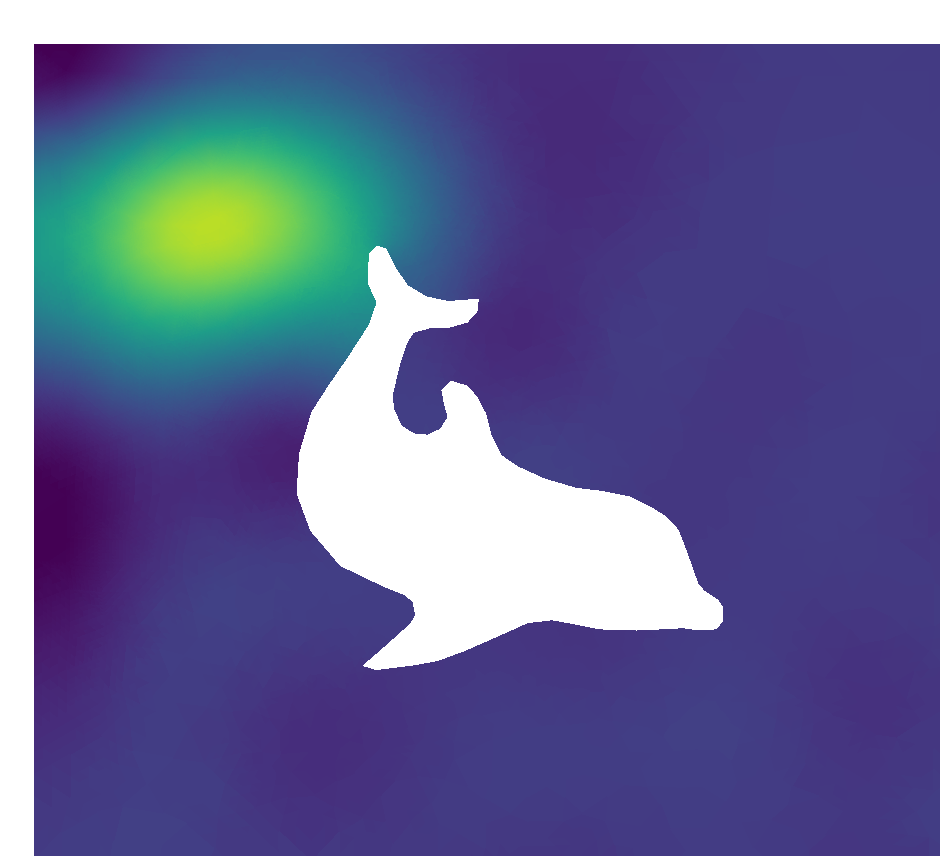}
                \end{subfigure}
            \end{tabular}
        \end{subfigure}\\
        
        \rotatebox{90}{\hspace{-0.9cm}RMA+RMAP}
        \begin{subfigure}[b]{0.22\textwidth}
            \begin{tabular}{c}
                \begin{subfigure}{\textwidth}
                    \centering
                    \includegraphics[trim=2.5cm 2.5cm 7.0cm 1.5cm,clip=true,width=\textwidth]{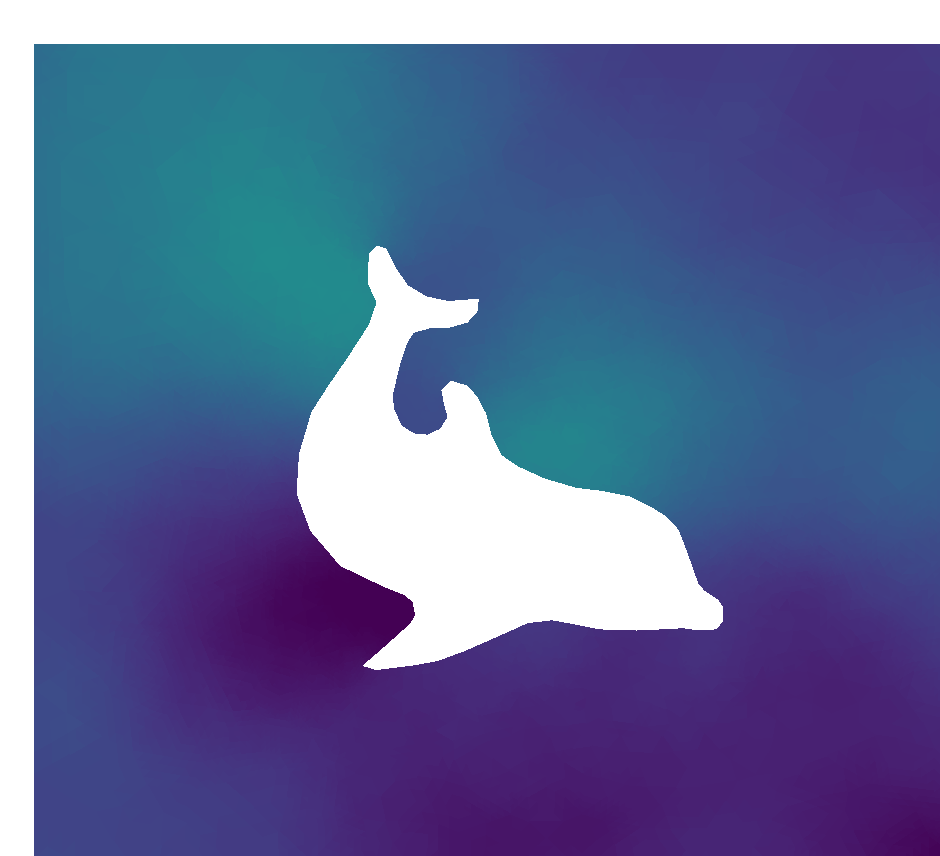}
                \end{subfigure}
            \end{tabular}
        \end{subfigure}

        \begin{subfigure}[b]{0.22\textwidth}
            \begin{tabular}{c}
                \begin{subfigure}{\textwidth}
                    \centering
                    \includegraphics[trim=2.5cm 2.5cm 7.0cm 1.5cm,clip=true,width=\textwidth]{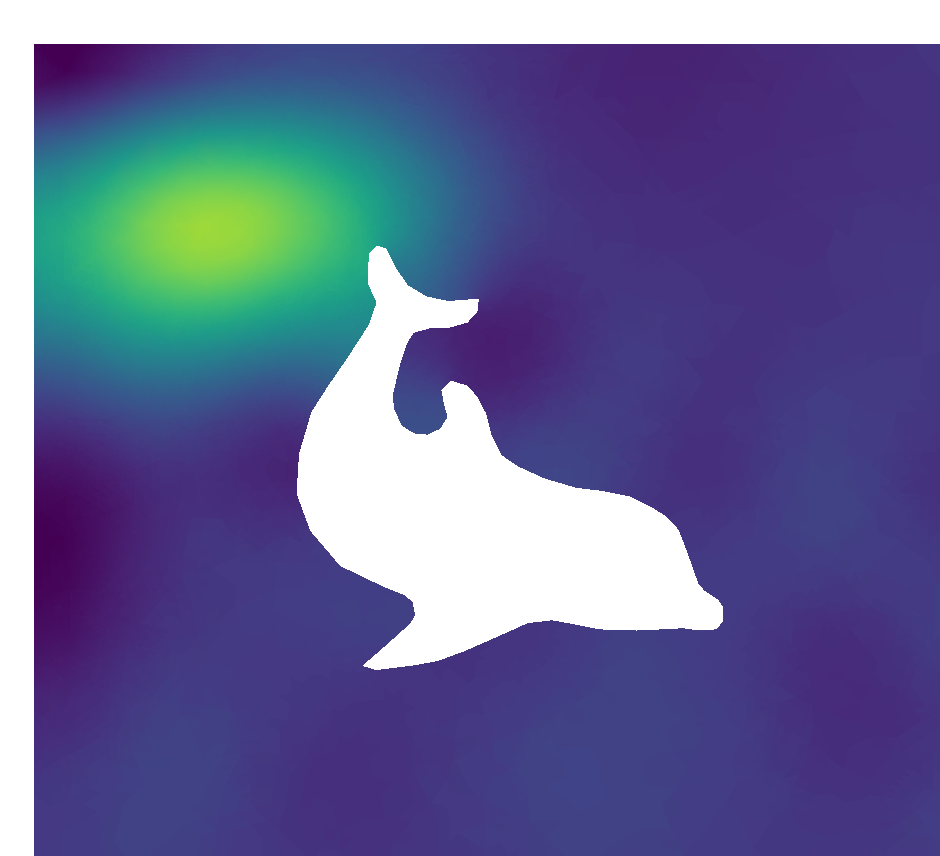}
                \end{subfigure}
            \end{tabular}
        \end{subfigure}
        
        \begin{subfigure}[b]{0.22\textwidth}
            \begin{tabular}{c}
                \begin{subfigure}{\textwidth}
                    \centering
                    \includegraphics[trim=2.5cm 2.5cm 7.0cm 1.5cm,clip=true,width=\textwidth]{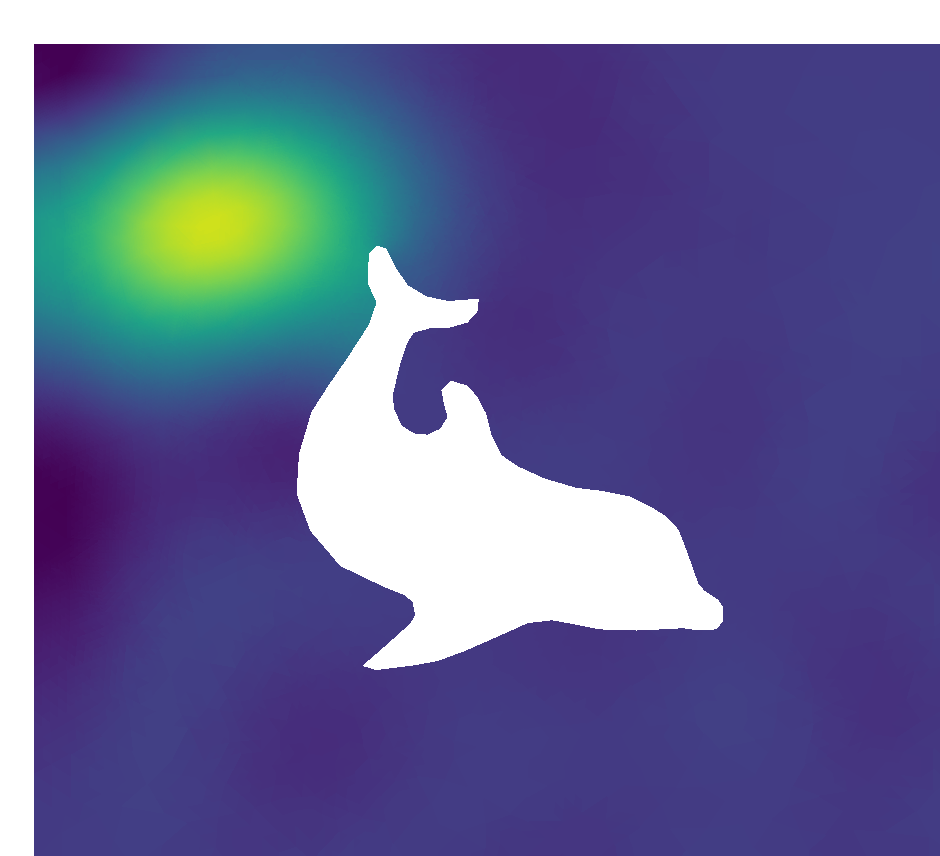}
                \end{subfigure}
            \end{tabular}
        \end{subfigure} 
        
        \begin{subfigure}[b]{0.22\textwidth}
            \begin{tabular}{c}
                \begin{subfigure}{\textwidth}
                    \centering
                    \includegraphics[trim=2.5cm 2.5cm 7.0cm 1.5cm,clip=true,width=\textwidth]{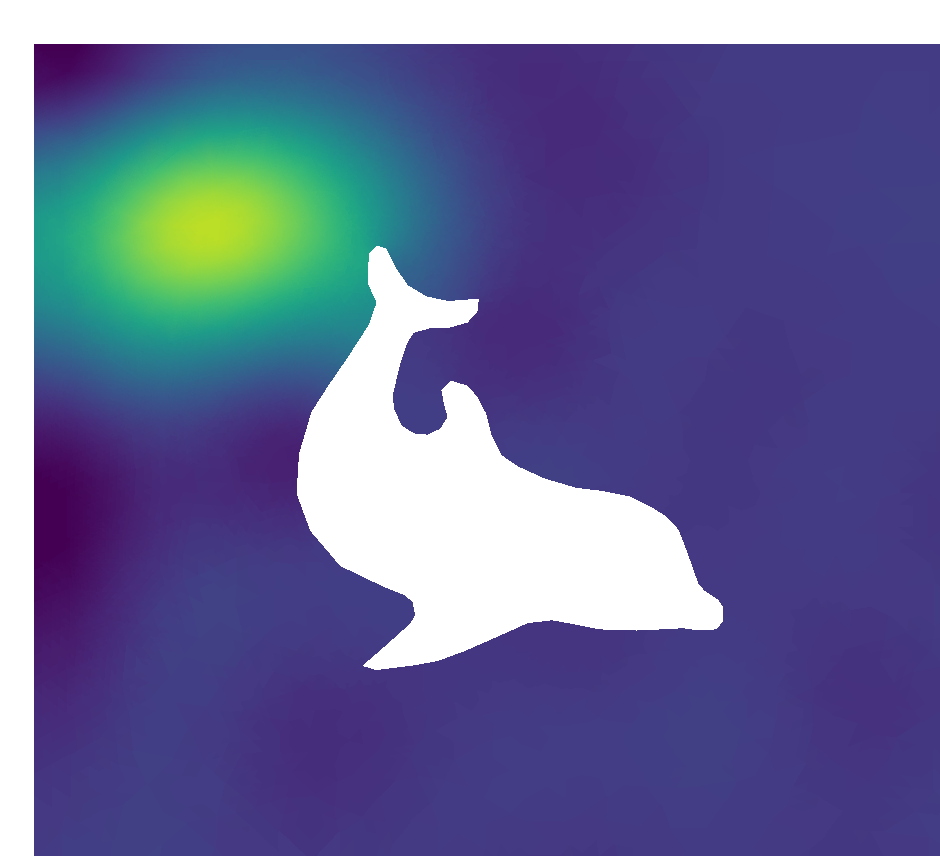}
                \end{subfigure}
            \end{tabular}
        \end{subfigure} \\

        \rotatebox{90}{\hspace{-0.2cm}RS}
        \begin{subfigure}[b]{0.22\textwidth}
            \begin{tabular}{c}
                \begin{subfigure}{\textwidth}
                    \centering
                    \includegraphics[trim=2.5cm 2.5cm 7.0cm 1.5cm,clip=true,width=\textwidth]{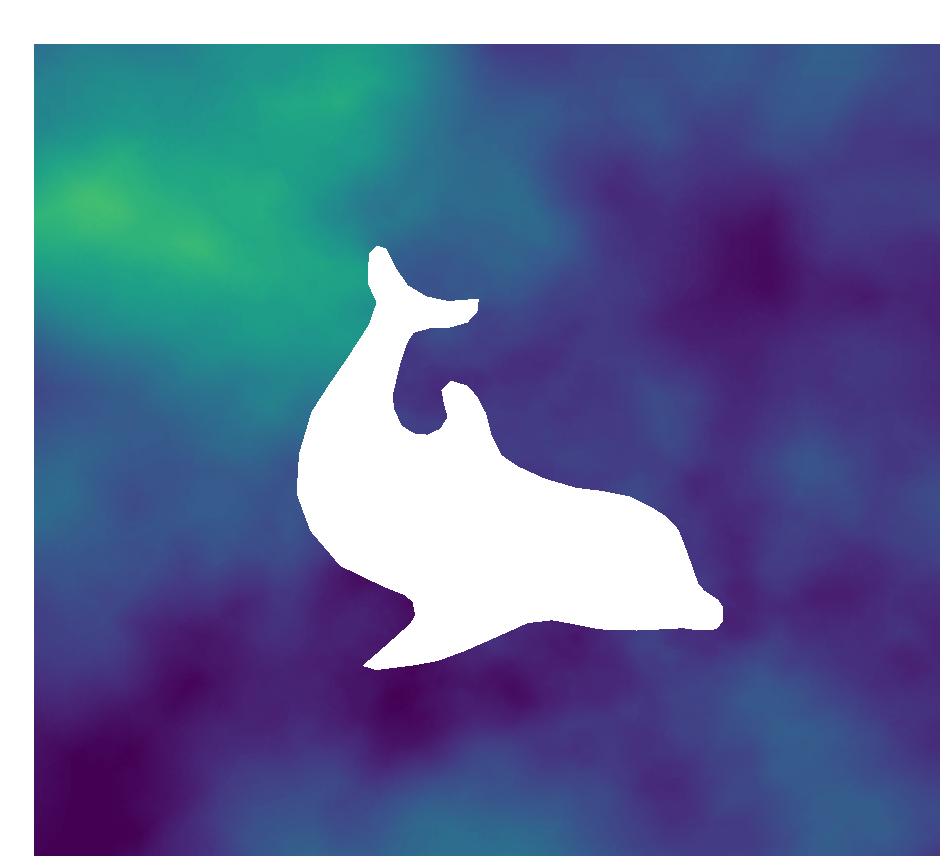}
                \end{subfigure}
            \end{tabular}
        \end{subfigure}

        \begin{subfigure}[b]{0.22\textwidth}
            \begin{tabular}{c}
                \begin{subfigure}{\textwidth}
                    \centering
                    \includegraphics[trim=2.5cm 2.5cm 7.0cm 1.5cm,clip=true,width=\textwidth]{pdf_figures/Linear_PDE/RS_\DNTwo.pdf}
                \end{subfigure}
            \end{tabular}
        \end{subfigure}
        
        \begin{subfigure}[b]{0.22\textwidth}
            \begin{tabular}{c}
                \begin{subfigure}{\textwidth}
                    \centering
                    \includegraphics[trim=2.5cm 2.5cm 7.0cm 1.5cm,clip=true,width=\textwidth]{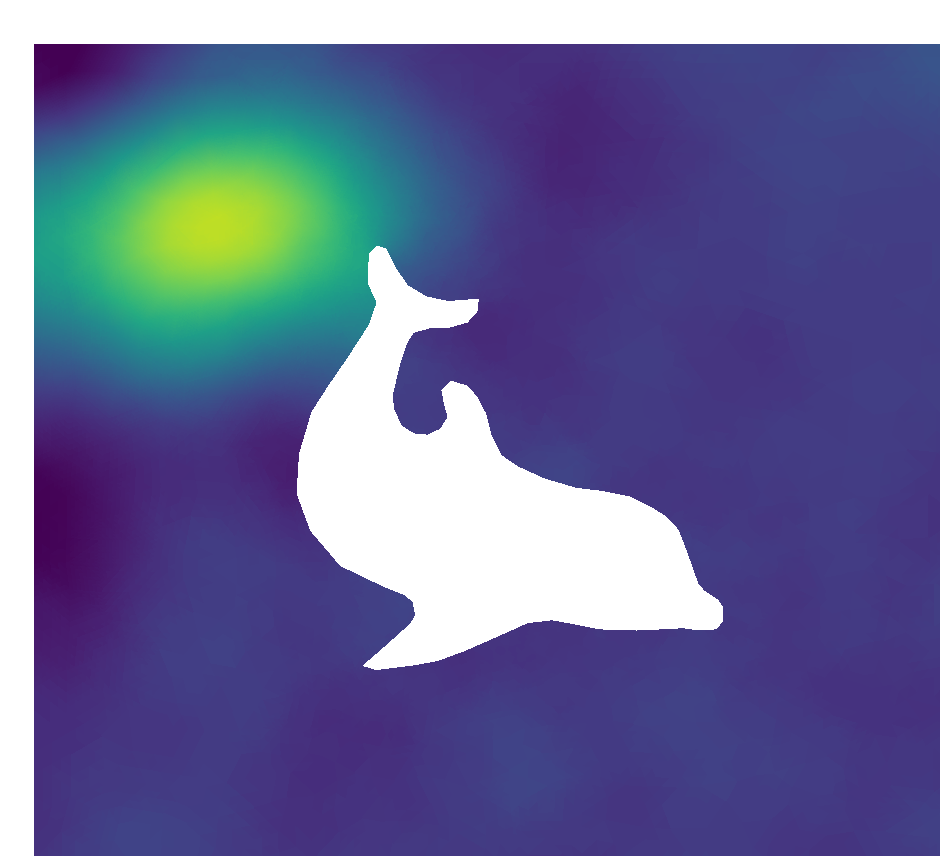}
                \end{subfigure}
            \end{tabular}
        \end{subfigure} 
        
        \begin{subfigure}[b]{0.22\textwidth}
            \begin{tabular}{c}
                \begin{subfigure}{\textwidth}
                    \centering
                    \includegraphics[trim=2.5cm 2.5cm 7.0cm 1.5cm,clip=true,width=\textwidth]{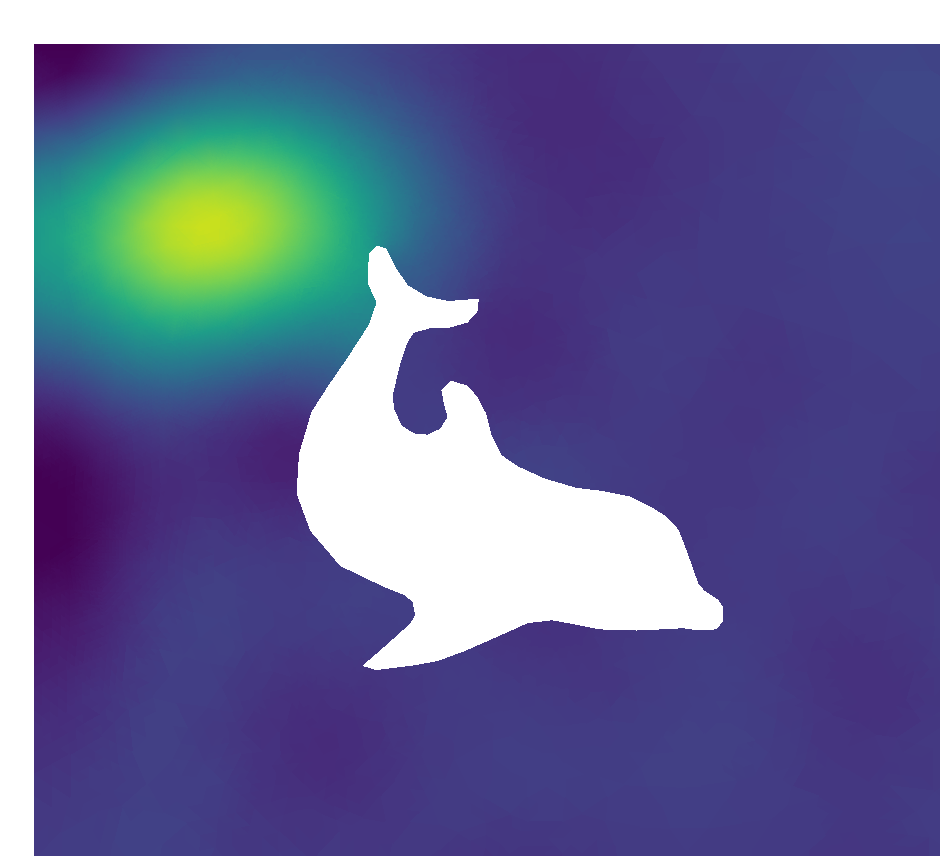}
                \end{subfigure}
            \end{tabular}
        \end{subfigure}\\

        \rotatebox{90}{\hspace{-0.5cm}ENKF}
        \begin{subfigure}[b]{0.22\textwidth}
            \begin{tabular}{c}
                \begin{subfigure}{\textwidth}
                    \centering
                    \includegraphics[trim=2.5cm 2.5cm 7.0cm 1.5cm,clip=true,width=\textwidth]{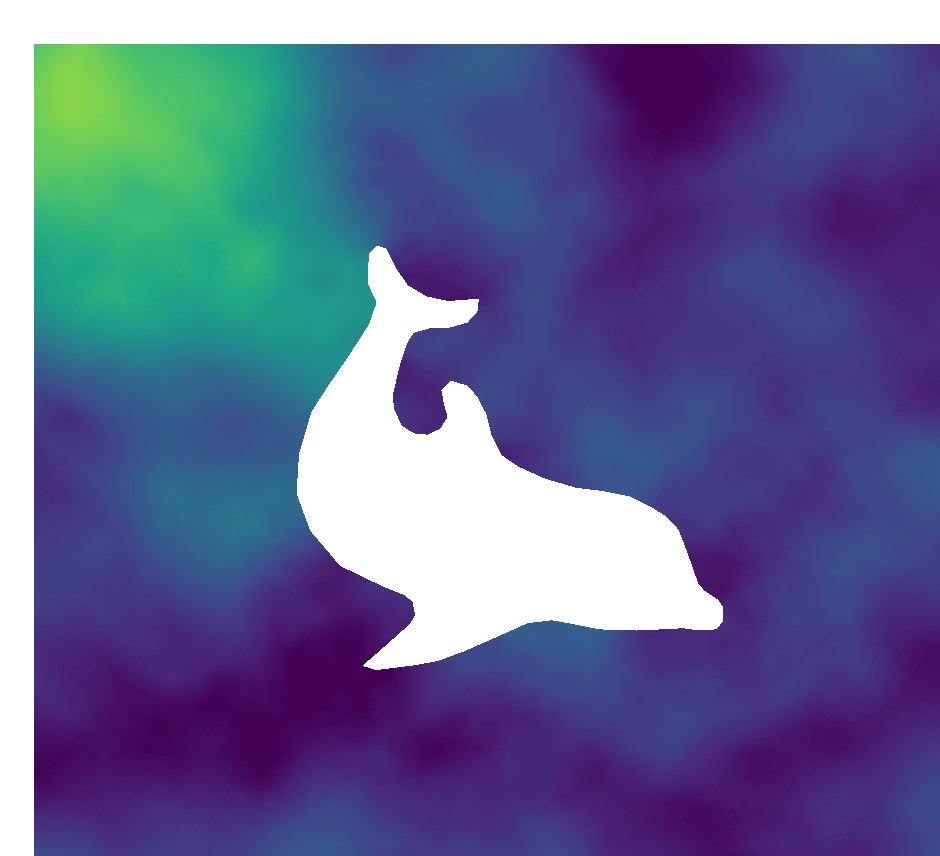}
                \end{subfigure}
            \end{tabular}
        \end{subfigure}

        \begin{subfigure}[b]{0.22\textwidth}
            \begin{tabular}{c}
                \begin{subfigure}{\textwidth}
                    \centering
                    \includegraphics[trim=2.5cm 2.5cm 7.0cm 1.5cm,clip=true,width=\textwidth]{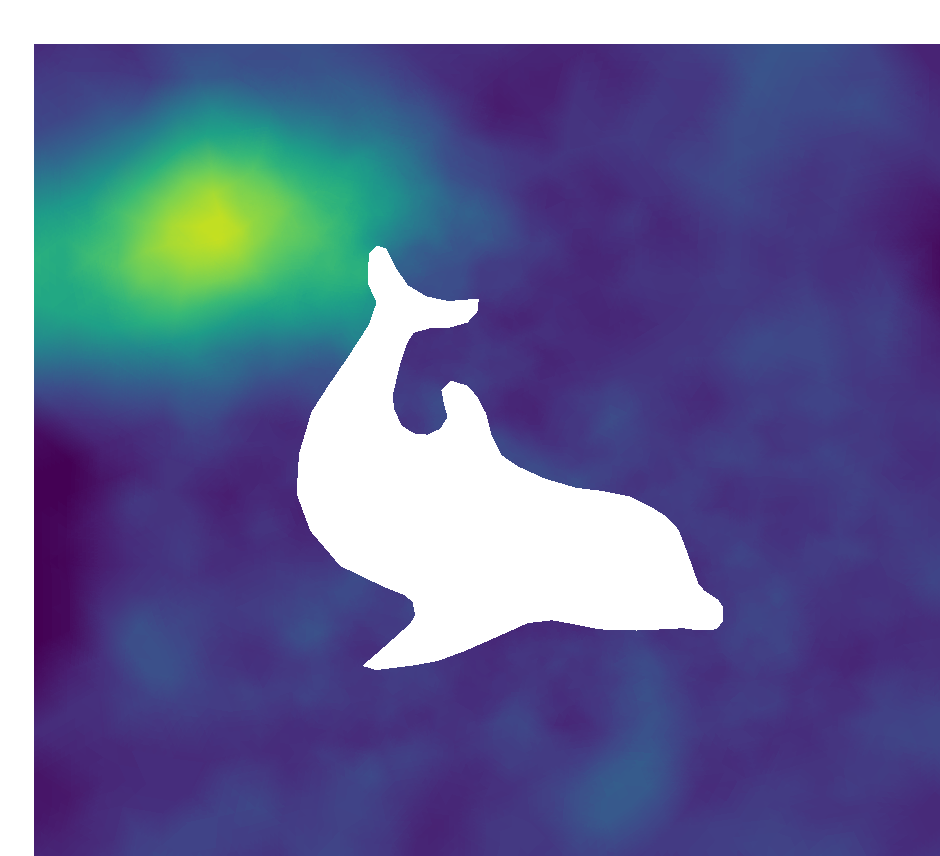}
                \end{subfigure}
            \end{tabular}
        \end{subfigure}
        
        \begin{subfigure}[b]{0.22\textwidth}
            \begin{tabular}{c}
                \begin{subfigure}{\textwidth}
                    \centering
                    \includegraphics[trim=2.5cm 2.5cm 7.0cm 1.5cm,clip=true,width=\textwidth]{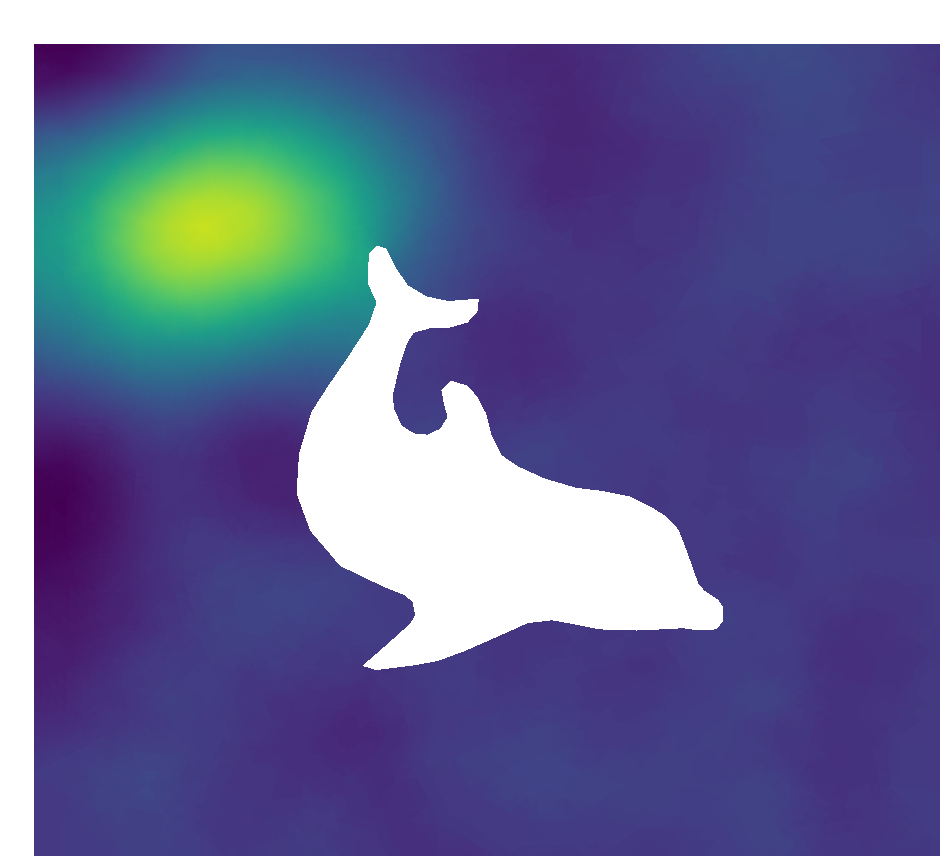}
                \end{subfigure}
            \end{tabular}
        \end{subfigure} 
        
        \begin{subfigure}[b]{0.22\textwidth}
            \begin{tabular}{c}
                \begin{subfigure}{\textwidth}
                    \centering
                    \includegraphics[trim=2.5cm 2.5cm 7.0cm 1.5cm,clip=true,width=\textwidth]{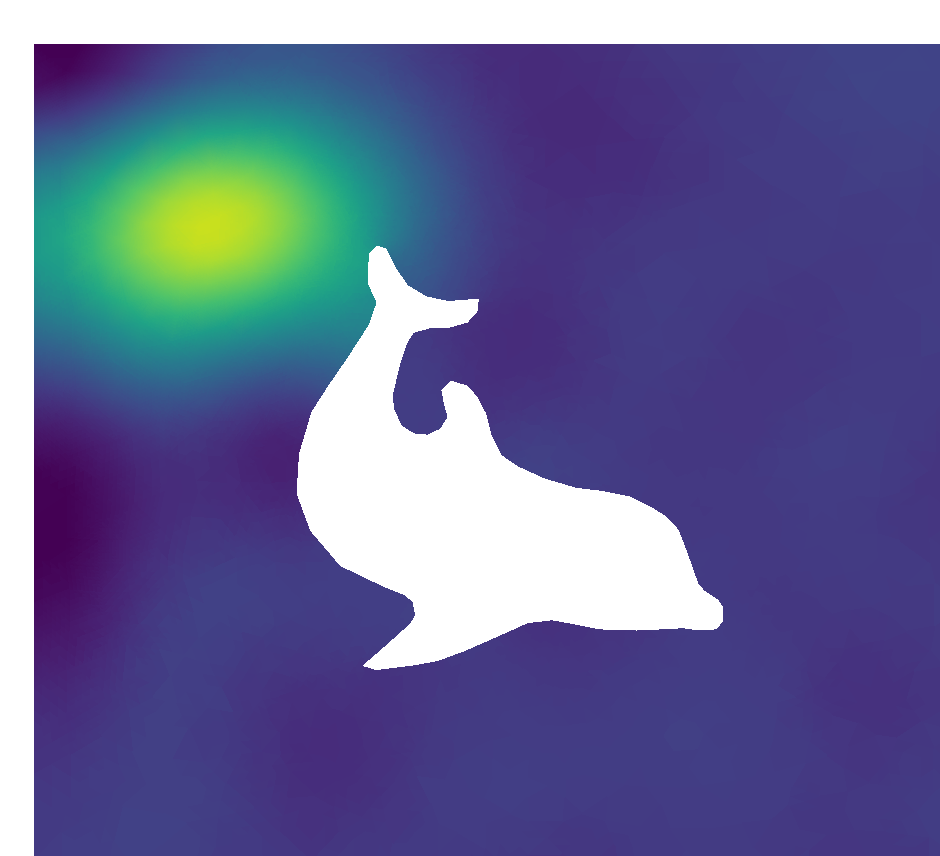}
                \end{subfigure}
            \end{tabular}
        \end{subfigure}\\
        
        \rotatebox{90}{\hspace{-0.3cm}ALL}
        \begin{subfigure}[b]{0.22\textwidth}
            \begin{tabular}{c}
                \begin{subfigure}{\textwidth}
                    \centering
                    \includegraphics[trim=2.5cm 2.5cm 7.0cm 1.5cm,clip=true,width=\textwidth]{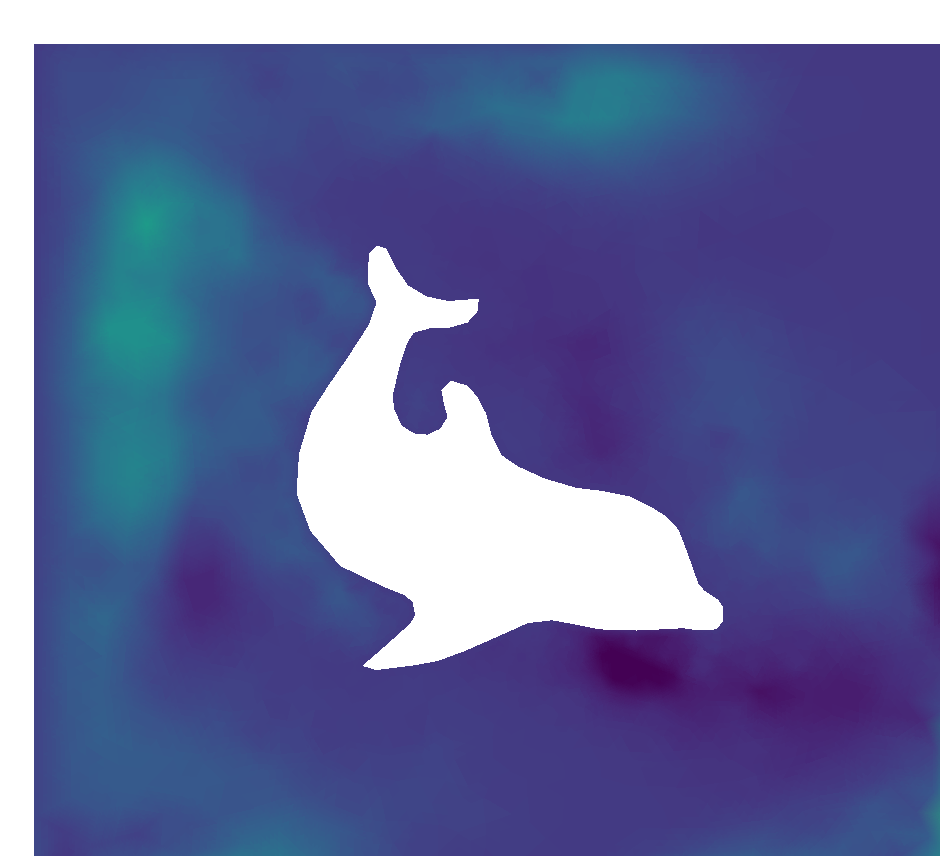}
                \end{subfigure}
            \end{tabular}
        \end{subfigure}

        \begin{subfigure}[b]{0.22\textwidth}
            \begin{tabular}{c}
                \begin{subfigure}{\textwidth}
                    \centering
                    \includegraphics[trim=2.5cm 2.5cm 7.0cm 1.5cm,clip=true,width=\textwidth]{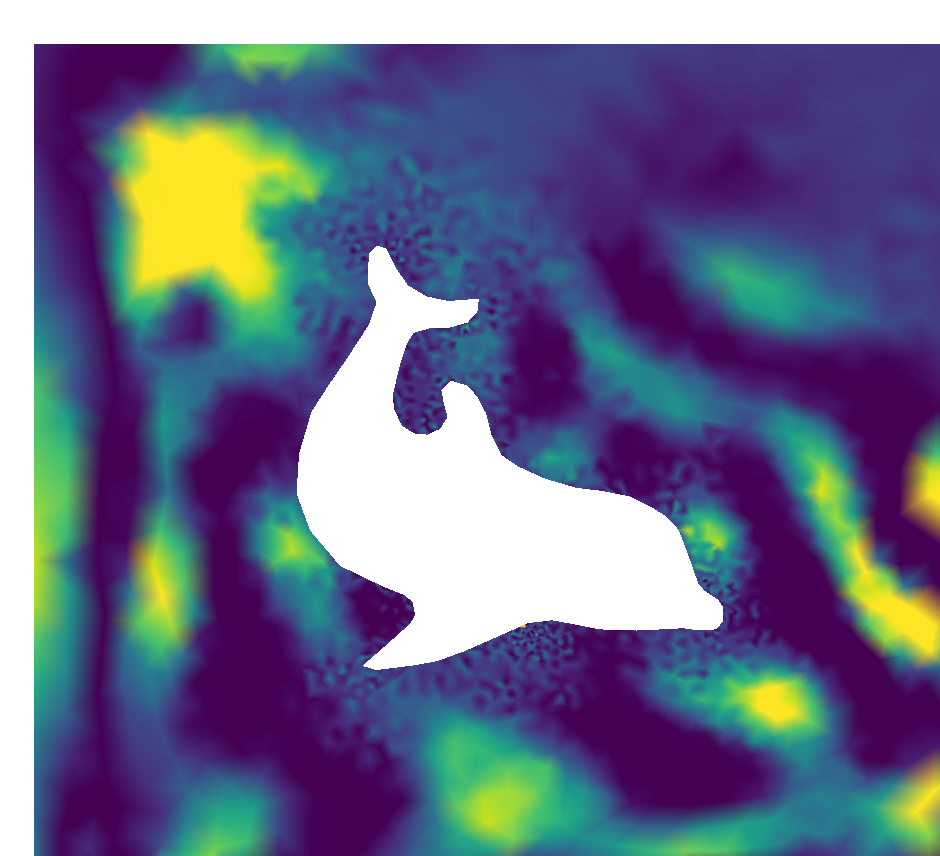}
                \end{subfigure}
            \end{tabular}
        \end{subfigure}
        
        \begin{subfigure}[b]{0.22\textwidth}
            \begin{tabular}{c}
                \begin{subfigure}{\textwidth}
                    \centering
                    \includegraphics[trim=2.5cm 2.5cm 7.0cm 1.5cm,clip=true,width=\textwidth]{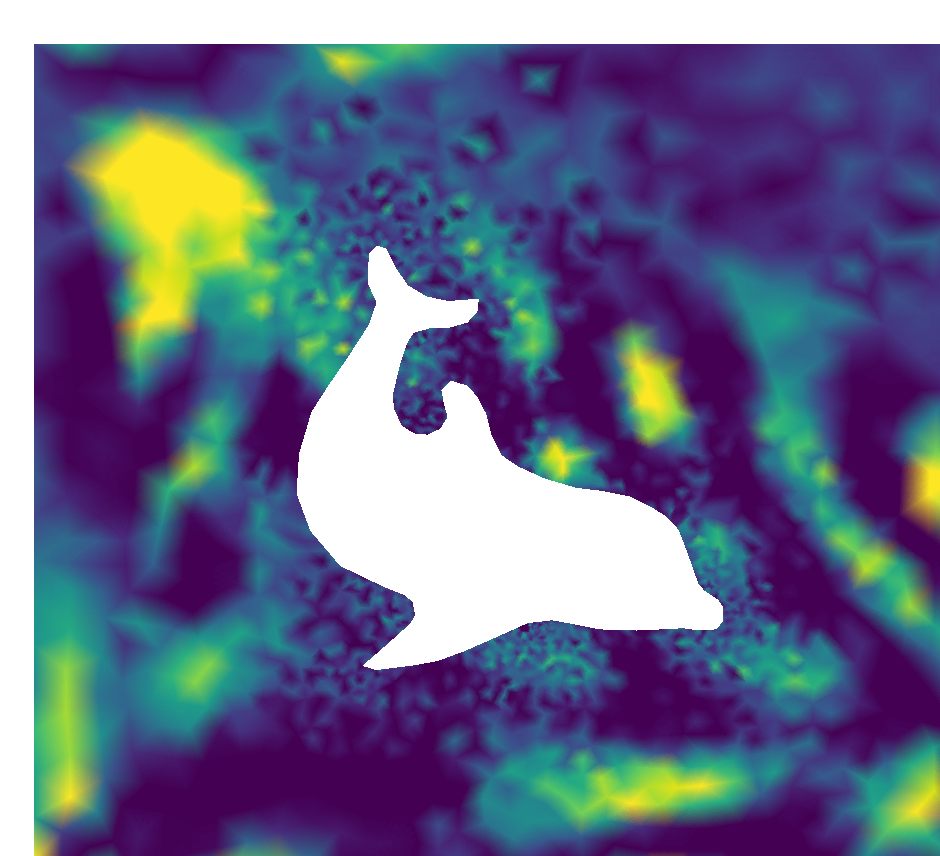}
                \end{subfigure}
            \end{tabular}
        \end{subfigure} 
        
        \begin{subfigure}[b]{0.22\textwidth}
            \begin{tabular}{c}
                \begin{subfigure}{\textwidth}
                    \centering
                    \includegraphics[trim=2.5cm 2.5cm 7.0cm 1.5cm,clip=true,width=\textwidth]{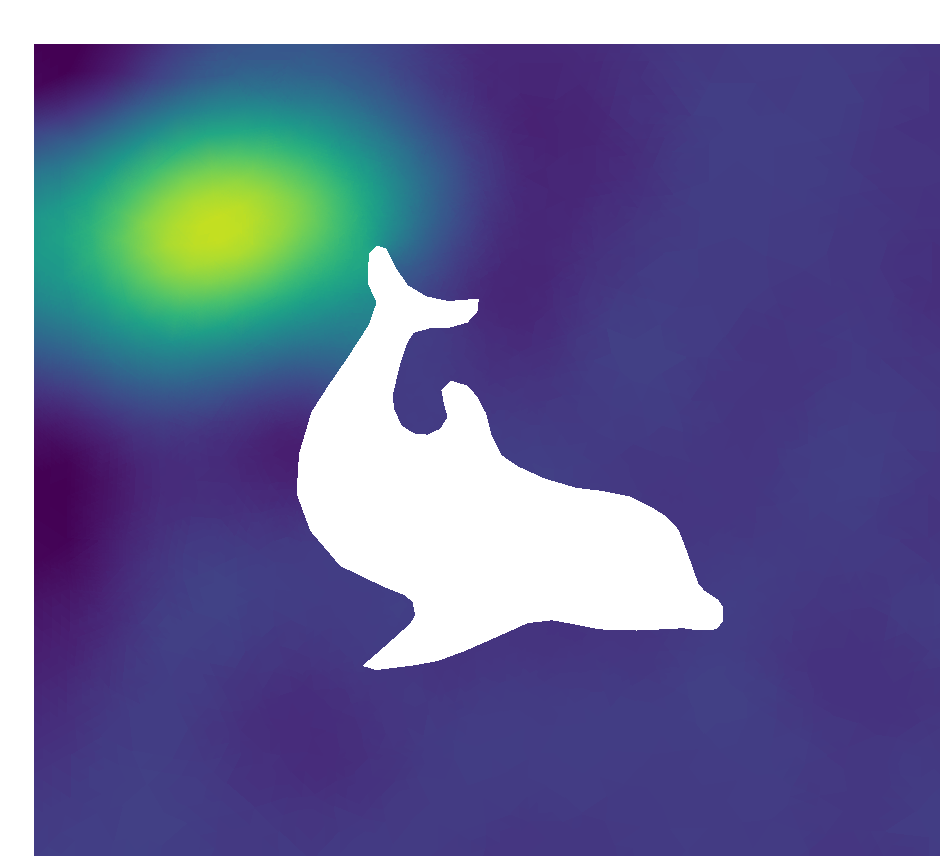}
                \end{subfigure}
            \end{tabular}
        \end{subfigure}\\

  \end{tabular}
    \caption{Solutions for various randomization approaches for a linear advection-diffusion initial condition inverse problem with Gaussian random variables.}
    \label{Linear_PDE}
\end{figure}

\newcommand{\POne}{10}
\newcommand{\PTwo}{100}
\newcommand{\PThree}{1000}
\newcommand{\PFour}{5000}
\def\stripzero#1{\expandafter\stripzerohelp#1}
\def\stripzerohelp#1{\ifx 0#1\expandafter\stripzerohelp\else#1\fi}
\newcommand{\linearPDEWidth}{0.2}
\begin{figure}[h!t!b!]
    \begin{tabular}[h!]{c}
        \rotatebox{90}{\hspace{-0.7cm}RMAP}
        \begin{subfigure}[b]{\linearPDEWidth\textwidth}
            \begin{tabular}{c}
            \textbf{N = \stripzero{\POne}}\\
                \begin{subfigure}{\textwidth}
                    \centering
                    \includegraphics[trim=2.5cm 2.5cm 7.0cm 1.5cm,clip=true,width=\textwidth]{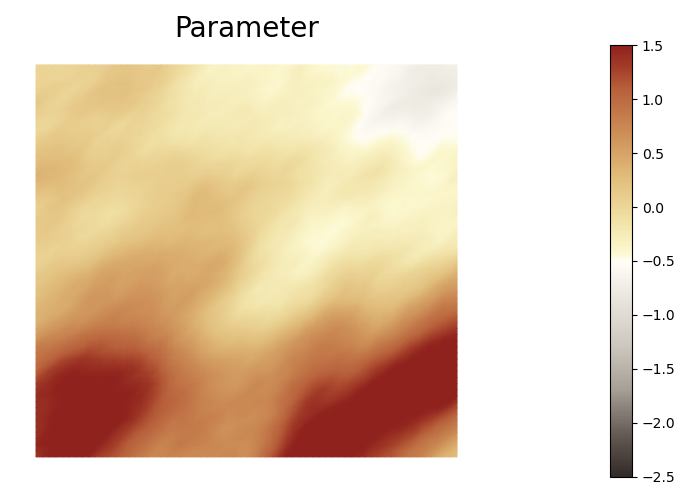}
                \end{subfigure}
            \end{tabular}
        \end{subfigure}

        \begin{subfigure}[b]{\linearPDEWidth\textwidth}
            \begin{tabular}{c}
            \textbf{N = \stripzero{\PTwo}}\\
                \begin{subfigure}{\textwidth}
                    \centering
                    \includegraphics[trim=2.5cm 2.5cm 7.0cm 1.5cm,clip=true,width=\textwidth]{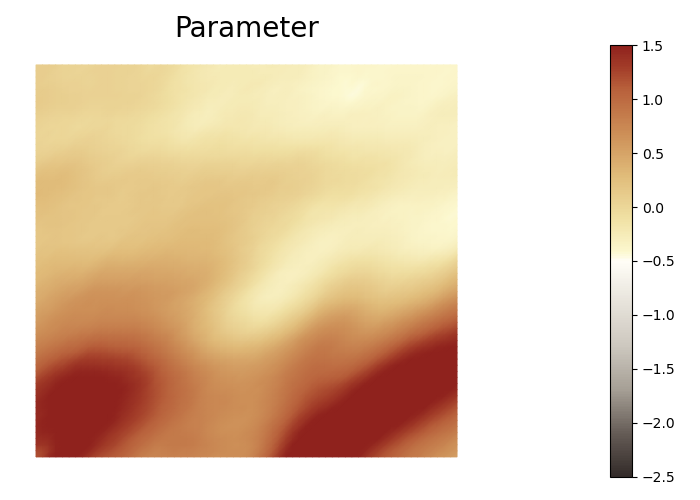}
                \end{subfigure}
            \end{tabular}
        \end{subfigure}
        
        \begin{subfigure}[b]{\linearPDEWidth\textwidth}
            \begin{tabular}{c}
            \textbf{N = \stripzero{\PThree}}\\
                \begin{subfigure}{\textwidth}
                    \centering
                    \includegraphics[trim=2.5cm 2.5cm 7.0cm 1.5cm,clip=true,width=\textwidth]{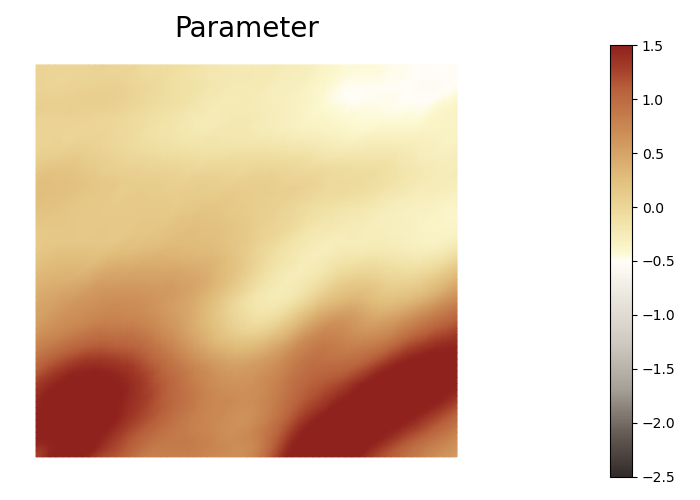}
                \end{subfigure}
            \end{tabular}
        \end{subfigure} 
        
        \begin{subfigure}[b]{\linearPDEWidth\textwidth}
            \begin{tabular}{c}
            \textbf{N = \stripzero{\PFour}}\\
                \begin{subfigure}{\textwidth}
                    \centering
                    \includegraphics[trim=2.5cm 2.5cm 7.0cm 1.5cm,clip=true,width=\textwidth]{pdf_figures/Nonlinear/final_parameter_rmap_\PFour.png}
                \end{subfigure}
            \end{tabular}
        \end{subfigure}\\
        
        \rotatebox{90}{\hspace{-0.4cm}RMA}
        \begin{subfigure}[b]{\linearPDEWidth\textwidth}
            \begin{tabular}{c}
                \begin{subfigure}{\textwidth}
                    \centering
                    \includegraphics[trim=2.5cm 2.5cm 7.0cm 1.5cm,clip=true,width=\textwidth]{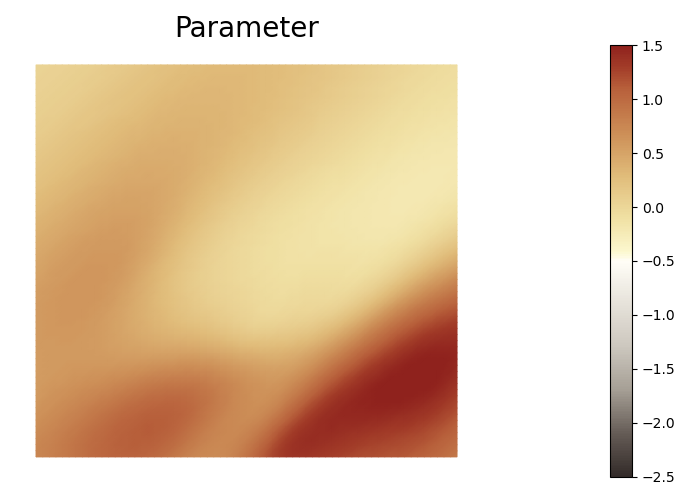}
                \end{subfigure}
            \end{tabular}
        \end{subfigure}

        \begin{subfigure}[b]{\linearPDEWidth\textwidth}
            \begin{tabular}{c}
                \begin{subfigure}{\textwidth}
                    \centering
                    \includegraphics[trim=2.5cm 2.5cm 7.0cm 1.5cm,clip=true,width=\textwidth]{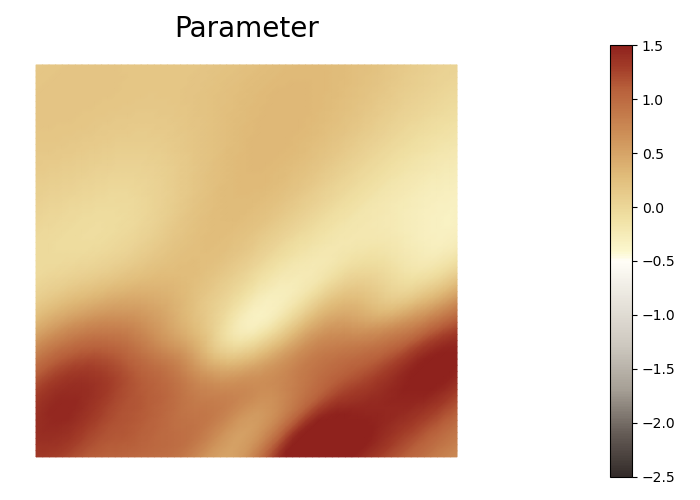}
                \end{subfigure}
            \end{tabular}
        \end{subfigure}
        
        \begin{subfigure}[b]{\linearPDEWidth\textwidth}
            \begin{tabular}{c}
                \begin{subfigure}{\textwidth}
                    \centering
                    \includegraphics[trim=2.5cm 2.5cm 7.0cm 1.5cm,clip=true,width=\textwidth]{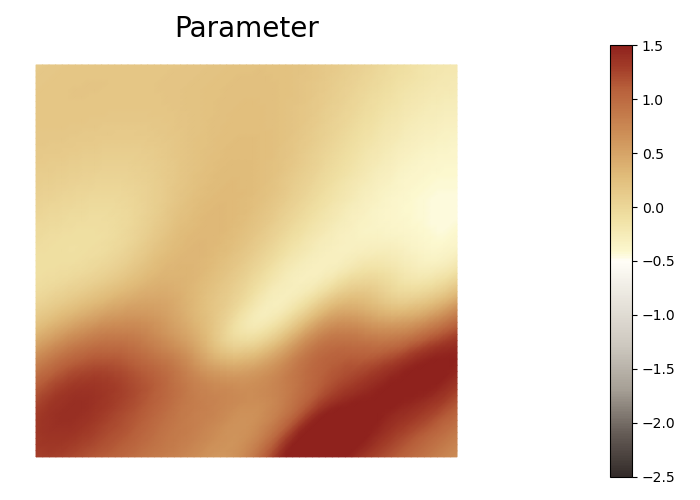}
                \end{subfigure}
            \end{tabular}
        \end{subfigure} 
        
        \begin{subfigure}[b]{\linearPDEWidth\textwidth}
            \begin{tabular}{c}
                \begin{subfigure}{\textwidth}
                    \centering
                    \includegraphics[trim=2.5cm 2.5cm 7.0cm 1.5cm,clip=true,width=\textwidth]{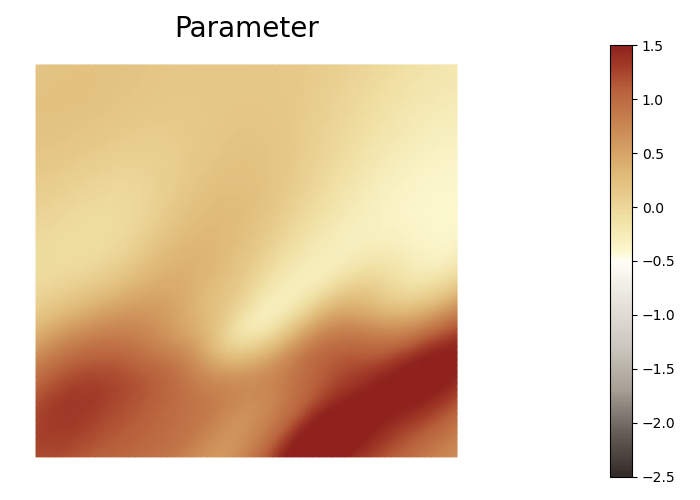}
                \end{subfigure}
            \end{tabular}
        \end{subfigure}\\
        
        \rotatebox{90}{\hspace{-0.9cm}RMA+RMAP}
        \begin{subfigure}[b]{\linearPDEWidth\textwidth}
            \begin{tabular}{c}
                \begin{subfigure}{\textwidth}
                    \centering
                    \includegraphics[trim=2.5cm 2.5cm 7.0cm 1.5cm,clip=true,width=\textwidth]{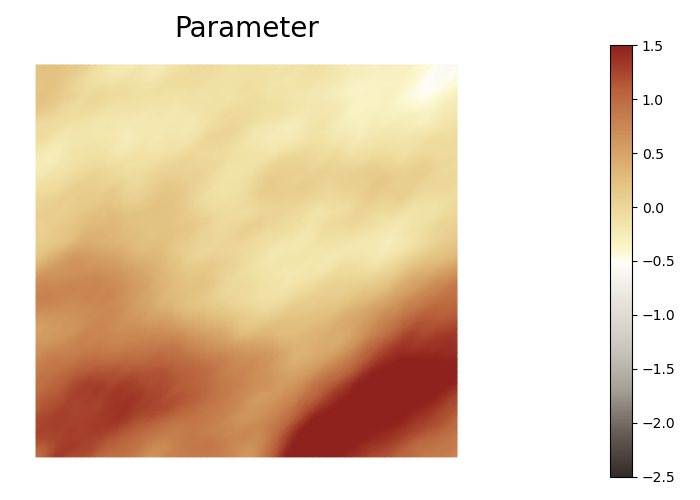}
                \end{subfigure}
            \end{tabular}
        \end{subfigure}

        \begin{subfigure}[b]{\linearPDEWidth\textwidth}
            \begin{tabular}{c}
                \begin{subfigure}{\textwidth}
                    \centering
                    \includegraphics[trim=2.5cm 2.5cm 7.0cm 1.5cm,clip=true,width=\textwidth]{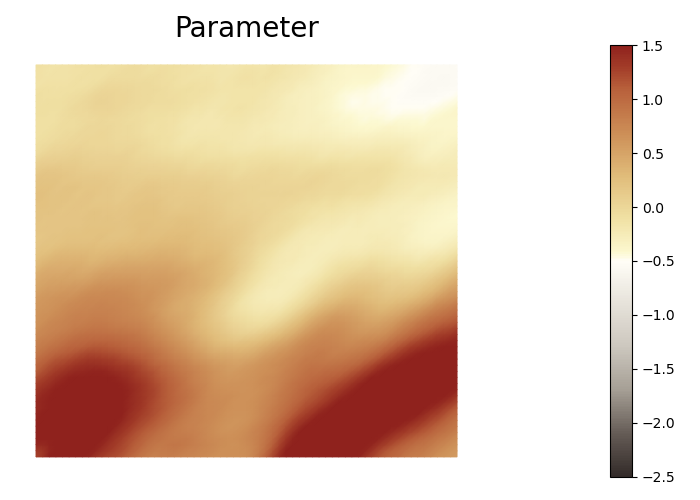}
                \end{subfigure}
            \end{tabular}
        \end{subfigure}
        
        \begin{subfigure}[b]{\linearPDEWidth\textwidth}
            \begin{tabular}{c}
                \begin{subfigure}{\textwidth}
                    \centering
                    \includegraphics[trim=2.5cm 2.5cm 7.0cm 1.5cm,clip=true,width=\textwidth]{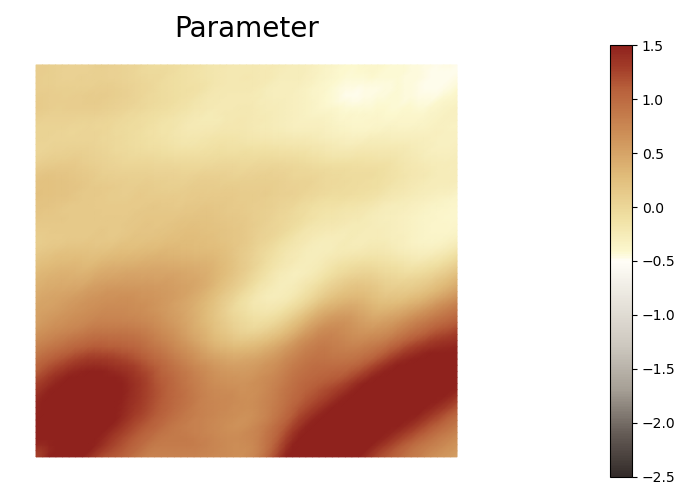}
                \end{subfigure}
            \end{tabular}
        \end{subfigure} 
        
        \begin{subfigure}[b]{\linearPDEWidth\textwidth}
            \begin{tabular}{c}
                \begin{subfigure}{\textwidth}
                    \centering
                    \includegraphics[trim=2.5cm 2.5cm 7.0cm 1.5cm,clip=true,width=\textwidth]{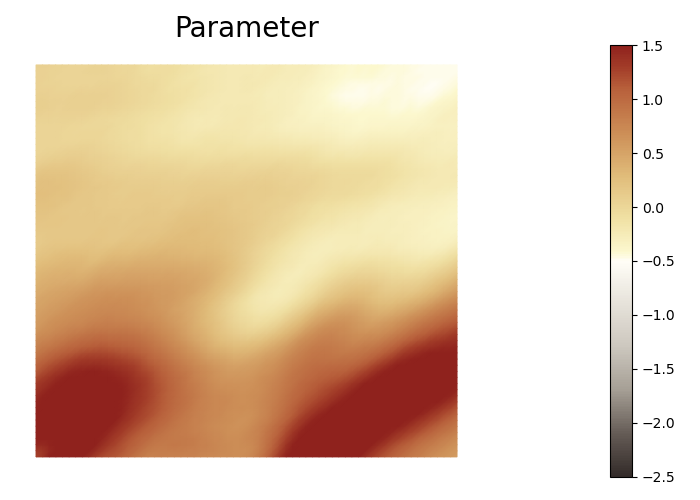}
                \end{subfigure}
            \end{tabular}
        \end{subfigure} \\

        \rotatebox{90}{\hspace{-0.2cm}RS}
        \begin{subfigure}[b]{\linearPDEWidth\textwidth}
            \begin{tabular}{c}
                \begin{subfigure}{\textwidth}
                    \centering
                    \includegraphics[trim=2.5cm 2.5cm 7.0cm 1.5cm,clip=true,width=\textwidth]{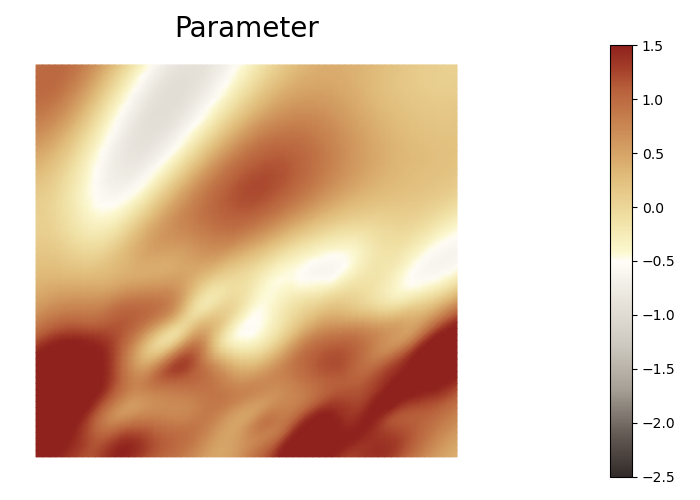}
                \end{subfigure}
            \end{tabular}
        \end{subfigure}

        \begin{subfigure}[b]{\linearPDEWidth\textwidth}
            \begin{tabular}{c}
                \begin{subfigure}{\textwidth}
                    \centering
                    \includegraphics[trim=2.5cm 2.5cm 7.0cm 1.5cm,clip=true,width=\textwidth]{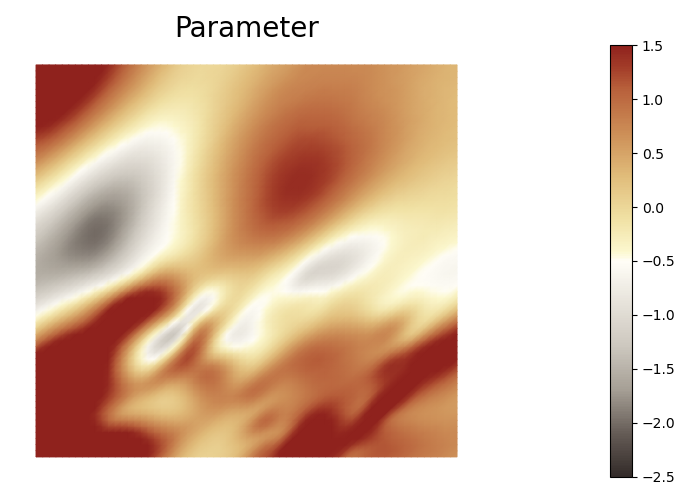}
                \end{subfigure}
            \end{tabular}
        \end{subfigure}
        
        \begin{subfigure}[b]{\linearPDEWidth\textwidth}
            \begin{tabular}{c}
                \begin{subfigure}{\textwidth}
                    \centering
                    \includegraphics[trim=2.5cm 2.5cm 7.0cm 1.5cm,clip=true,width=\textwidth]{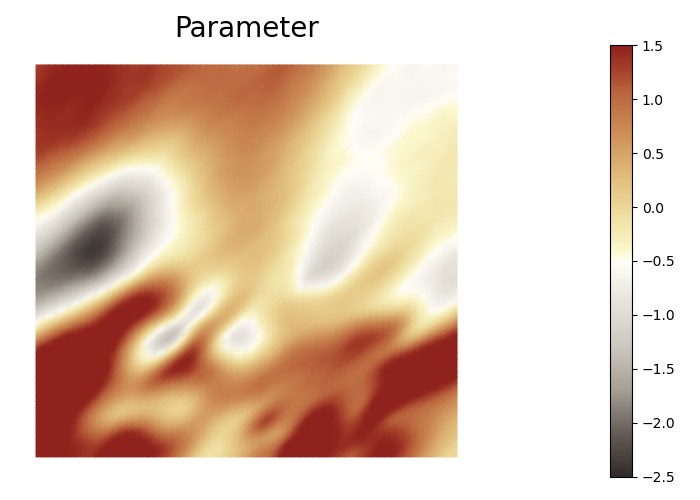}
                \end{subfigure}
            \end{tabular}
        \end{subfigure} 
        
        \begin{subfigure}[b]{\linearPDEWidth\textwidth}
            \begin{tabular}{c}
                \begin{subfigure}{\textwidth}
                    \centering
                    \includegraphics[trim=2.5cm 2.5cm 7.0cm 1.5cm,clip=true,width=\textwidth]{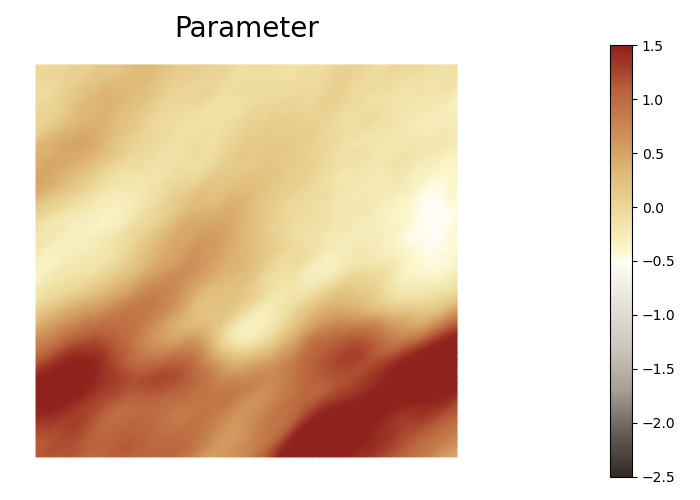}
                \end{subfigure}
            \end{tabular}
        \end{subfigure}\\

        \rotatebox{90}{\hspace{-0.5cm}ENKF}
        \begin{subfigure}[b]{\linearPDEWidth\textwidth}
            \begin{tabular}{c}
                \begin{subfigure}{\textwidth}
                    \centering
                    \includegraphics[trim=2.5cm 2.5cm 7.0cm 1.5cm,clip=true,width=\textwidth]{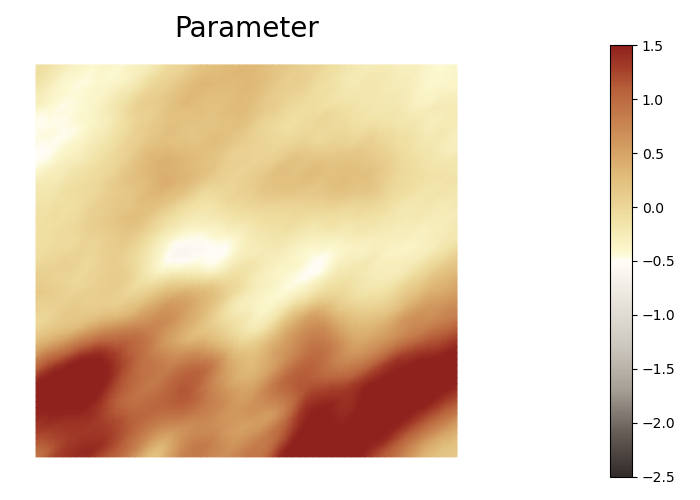}
                \end{subfigure}
            \end{tabular}
        \end{subfigure}

        \begin{subfigure}[b]{\linearPDEWidth\textwidth}
            \begin{tabular}{c}
                \begin{subfigure}{\textwidth}
                    \centering
                    \includegraphics[trim=2.5cm 2.5cm 7.0cm 1.5cm,clip=true,width=\textwidth]{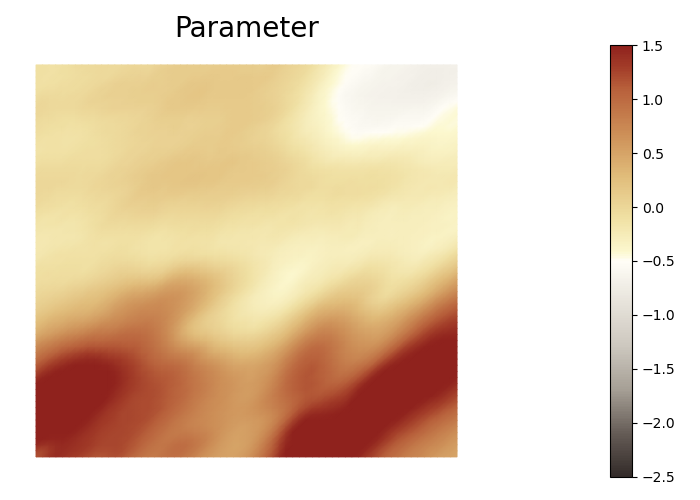}
                \end{subfigure}
            \end{tabular}
        \end{subfigure}
        
        \begin{subfigure}[b]{\linearPDEWidth\textwidth}
            \begin{tabular}{c}
                \begin{subfigure}{\textwidth}
                    \centering
                    \includegraphics[trim=2.5cm 2.5cm 7.0cm 1.5cm,clip=true,width=\textwidth]{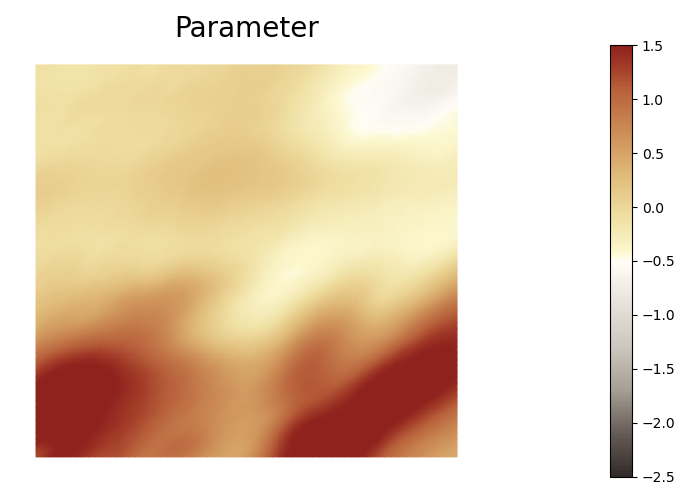}
                \end{subfigure}
            \end{tabular}
        \end{subfigure} 
        
        \begin{subfigure}[b]{\linearPDEWidth\textwidth}
            \begin{tabular}{c}
                \begin{subfigure}{\textwidth}
                    \centering
                    \includegraphics[trim=2.5cm 2.5cm 7.0cm 1.5cm,clip=true,width=\textwidth]{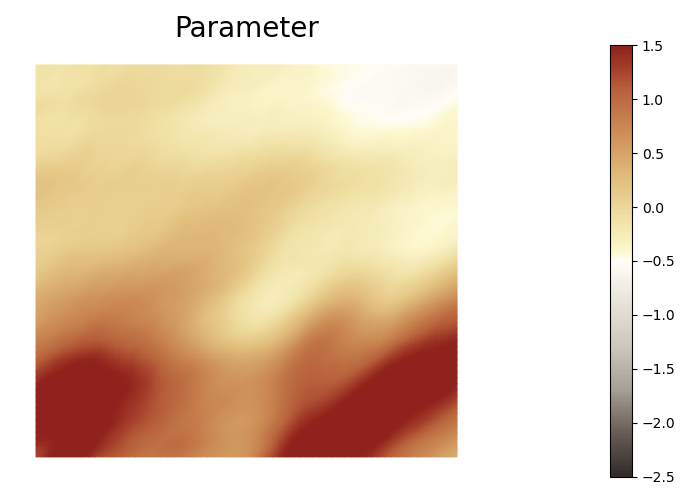}
                \end{subfigure}
            \end{tabular}
        \end{subfigure}\\
        
        \rotatebox{90}{\hspace{-0.3cm}ALL}
        \begin{subfigure}[b]{\linearPDEWidth\textwidth}
            \begin{tabular}{c}
                \begin{subfigure}{\textwidth}
                    \centering
                    \includegraphics[trim=2.5cm 2.5cm 7.0cm 1.5cm,clip=true,width=\textwidth]{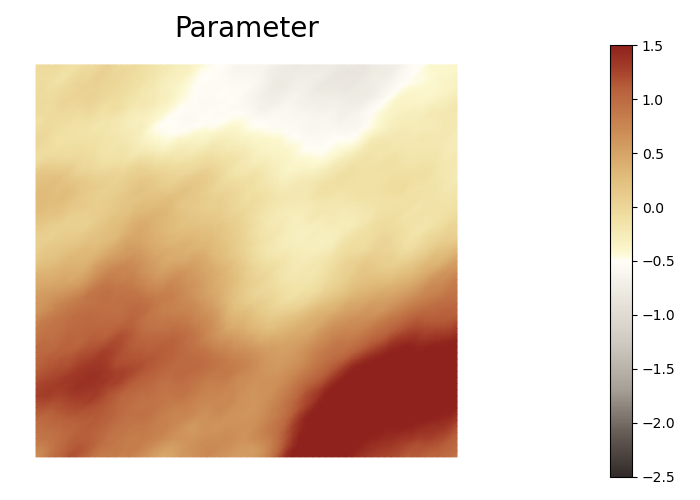}
                \end{subfigure}
            \end{tabular}
        \end{subfigure}

        \begin{subfigure}[b]{\linearPDEWidth\textwidth}
            \begin{tabular}{c}
                \begin{subfigure}{\textwidth}
                    \centering
                    \includegraphics[trim=2.5cm 2.5cm 7.0cm 1.5cm,clip=true,width=\textwidth]{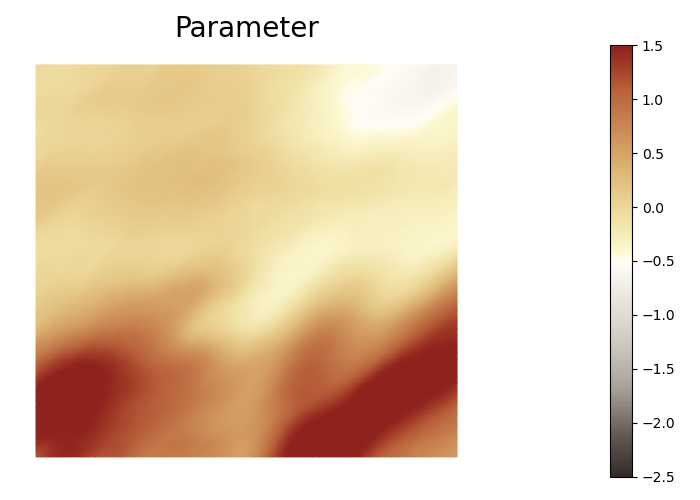}
                \end{subfigure}
            \end{tabular}
        \end{subfigure}
        
        \begin{subfigure}[b]{\linearPDEWidth\textwidth}
            \begin{tabular}{c}
                \begin{subfigure}{\textwidth}
                    \centering
                    \includegraphics[trim=2.5cm 2.5cm 7.0cm 1.5cm,clip=true,width=\textwidth]{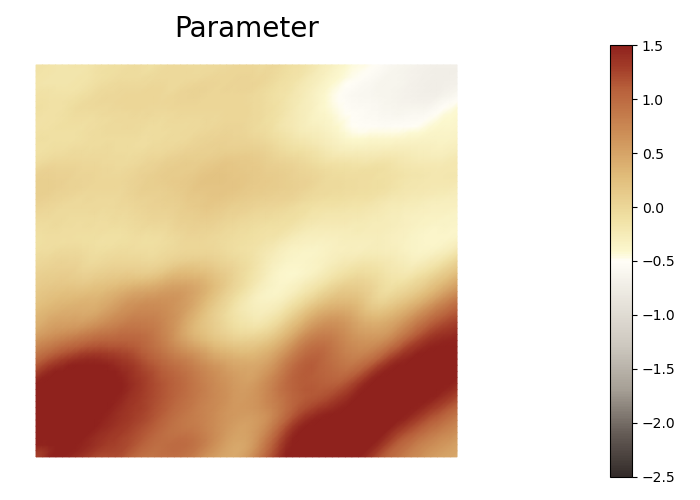}
                \end{subfigure}
            \end{tabular}
        \end{subfigure} 
        
        \begin{subfigure}[b]{\linearPDEWidth\textwidth}
            \begin{tabular}{c}
                \begin{subfigure}{\textwidth}
                    \centering
                    \includegraphics[trim=2.5cm 2.5cm 7.0cm 1.5cm,clip=true,width=\textwidth]{pdf_figures/Nonlinear/final_parameter_all_\PFour.png}
                \end{subfigure}
            \end{tabular}
        \end{subfigure}\\

  \end{tabular}
    \caption{Solutions for various randomization approaches for a nonlinear diffusion parameter inversion problem with Gaussian random variables.}
    \label{Nonlinear_PDE}
\end{figure} 

\newcommand{\nonlinearWidth}{0.2}
\begin{figure}[h!t!b!]
    \begin{tabular}[h!]{c}
        \rotatebox{90}{\hspace{-0.7cm}RMAP}
        \begin{subfigure}[b]{\nonlinearWidth\textwidth}
            \begin{tabular}{c}
            \textbf{N = \stripzero{\DNOne}}\\
                \begin{subfigure}{\textwidth}
                    \centering
                    \includegraphics[trim=2.5cm 2.5cm 7.0cm 1.5cm,clip=true,width=\textwidth]{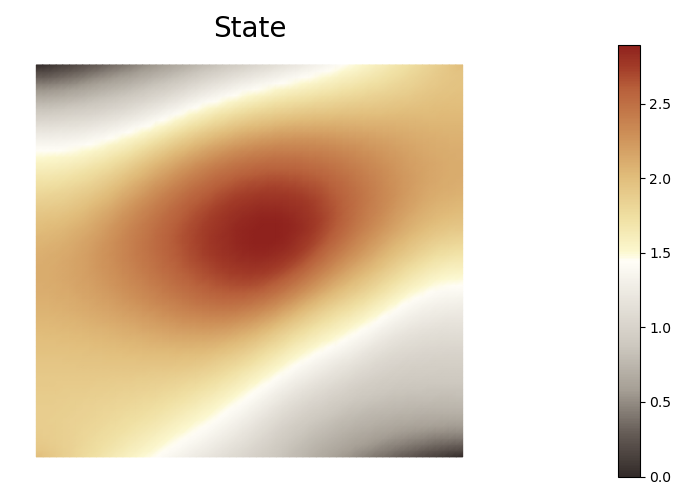}
                \end{subfigure}
            \end{tabular}
        \end{subfigure}

        \begin{subfigure}[b]{\nonlinearWidth\textwidth}
            \begin{tabular}{c}
            \textbf{N = \stripzero{\PTwo}}\\
                \begin{subfigure}{\textwidth}
                    \centering
                    \includegraphics[trim=2.5cm 2.5cm 7.0cm 1.5cm,clip=true,width=\textwidth]{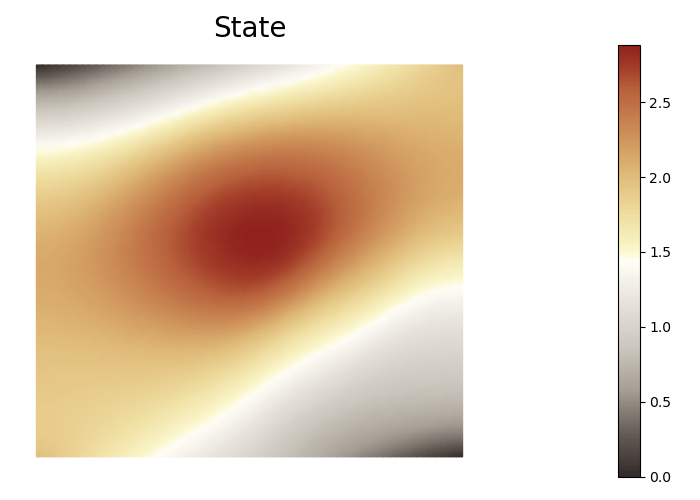}
                \end{subfigure}
            \end{tabular}
        \end{subfigure}
        
        \begin{subfigure}[b]{\nonlinearWidth\textwidth}
            \begin{tabular}{c}
            \textbf{N = \stripzero{\PThree}}\\
                \begin{subfigure}{\textwidth}
                    \centering
                    \includegraphics[trim=2.5cm 2.5cm 7.0cm 1.5cm,clip=true,width=\textwidth]{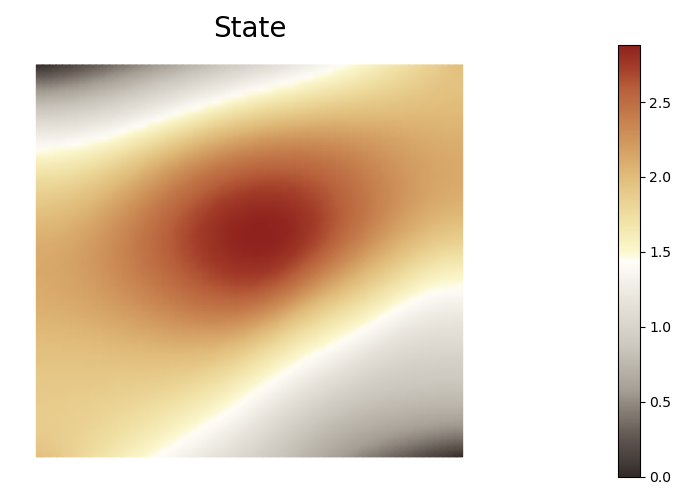}
                \end{subfigure}
            \end{tabular}
        \end{subfigure} 
        
        \begin{subfigure}[b]{\nonlinearWidth\textwidth}
            \begin{tabular}{c}
            \textbf{N = \stripzero{\DNFour}}\\
                \begin{subfigure}{\textwidth}
                    \centering
                    \includegraphics[trim=2.5cm 2.5cm 7.0cm 1.5cm,clip=true,width=\textwidth]{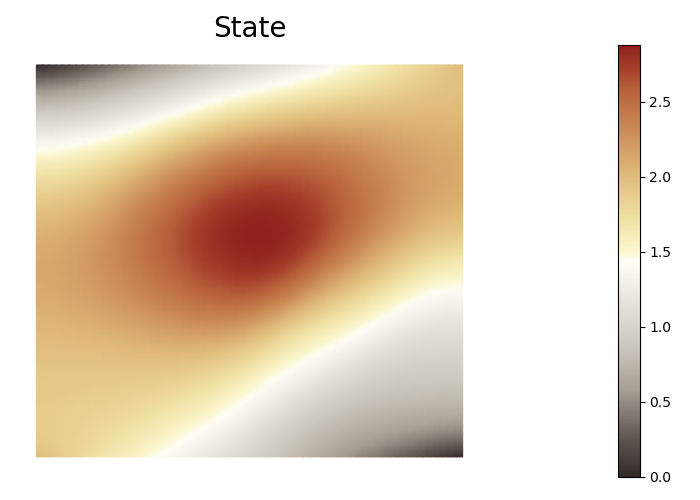}
                \end{subfigure}
            \end{tabular}
        \end{subfigure}\\
        
        \rotatebox{90}{\hspace{-0.4cm}RMA}
        \begin{subfigure}[b]{\nonlinearWidth\textwidth}
            \begin{tabular}{c}
                \begin{subfigure}{\textwidth}
                    \centering
                    \includegraphics[trim=2.5cm 2.5cm 7.0cm 1.5cm,clip=true,width=\textwidth]{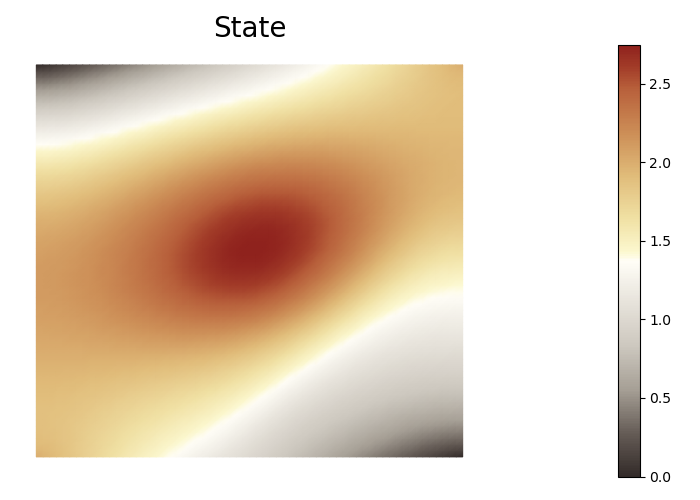}
                \end{subfigure}
            \end{tabular}
        \end{subfigure}

        \begin{subfigure}[b]{\nonlinearWidth\textwidth}
            \begin{tabular}{c}
                \begin{subfigure}{\textwidth}
                    \centering
                    \includegraphics[trim=2.5cm 2.5cm 7.0cm 1.5cm,clip=true,width=\textwidth]{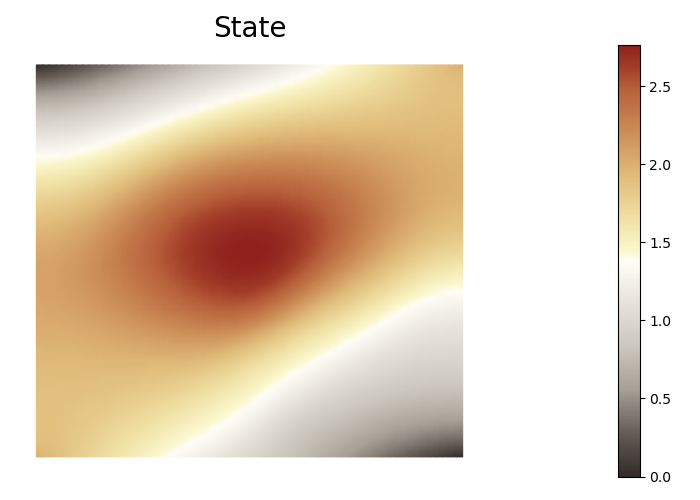}
                \end{subfigure}
            \end{tabular}
        \end{subfigure}
        
        \begin{subfigure}[b]{\nonlinearWidth\textwidth}
            \begin{tabular}{c}
                \begin{subfigure}{\textwidth}
                    \centering
                    \includegraphics[trim=2.5cm 2.5cm 7.0cm 1.5cm,clip=true,width=\textwidth]{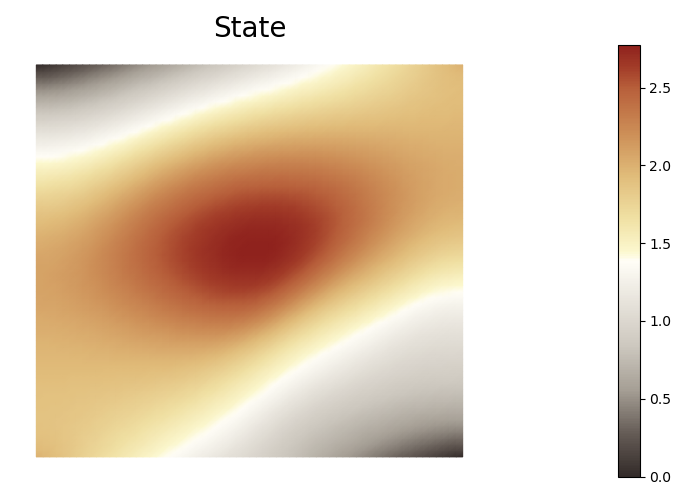}
                \end{subfigure}
            \end{tabular}
        \end{subfigure} 
        
        \begin{subfigure}[b]{\nonlinearWidth\textwidth}
            \begin{tabular}{c}
                \begin{subfigure}{\textwidth}
                    \centering
                    \includegraphics[trim=2.5cm 2.5cm 7.0cm 1.5cm,clip=true,width=\textwidth]{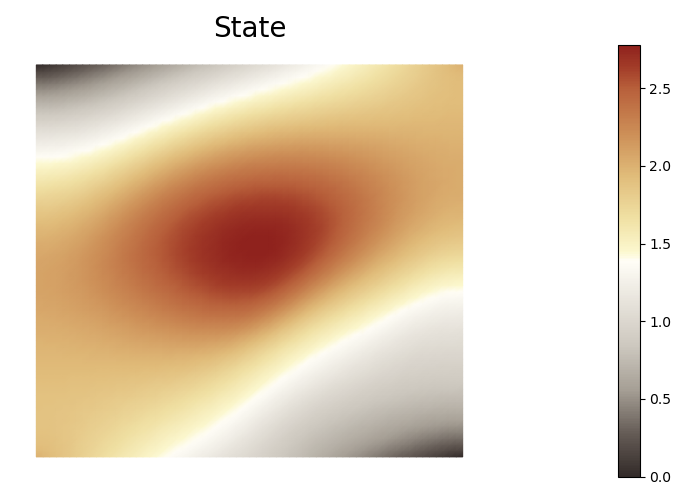}
                \end{subfigure}
            \end{tabular}
        \end{subfigure}\\
        
        \rotatebox{90}{\hspace{-0.9cm}RMA+RMAP}
        \begin{subfigure}[b]{\nonlinearWidth\textwidth}
            \begin{tabular}{c}
                \begin{subfigure}{\textwidth}
                    \centering
                    \includegraphics[trim=2.5cm 2.5cm 7.0cm 1.5cm,clip=true,width=\textwidth]{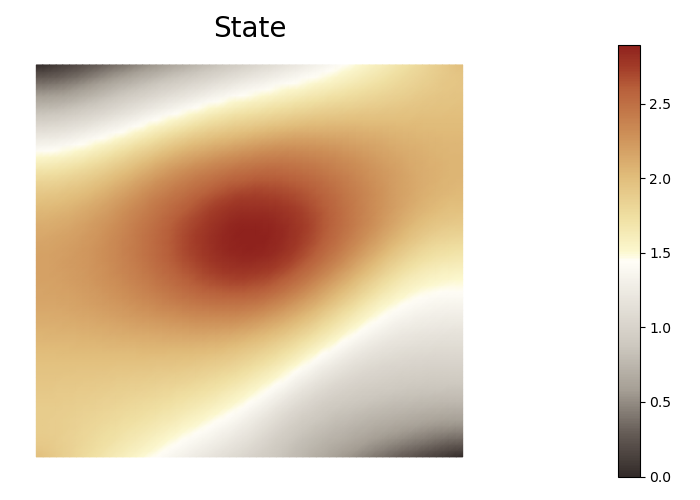}
                \end{subfigure}
            \end{tabular}
        \end{subfigure}

        \begin{subfigure}[b]{\nonlinearWidth\textwidth}
            \begin{tabular}{c}
                \begin{subfigure}{\textwidth}
                    \centering
                    \includegraphics[trim=2.5cm 2.5cm 7.0cm 1.5cm,clip=true,width=\textwidth]{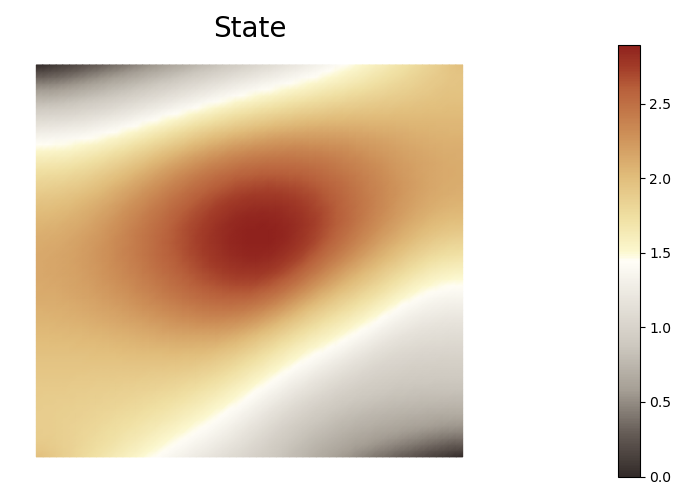}
                \end{subfigure}
            \end{tabular}
        \end{subfigure}
        
        \begin{subfigure}[b]{\nonlinearWidth\textwidth}
            \begin{tabular}{c}
                \begin{subfigure}{\textwidth}
                    \centering
                    \includegraphics[trim=2.5cm 2.5cm 7.0cm 1.5cm,clip=true,width=\textwidth]{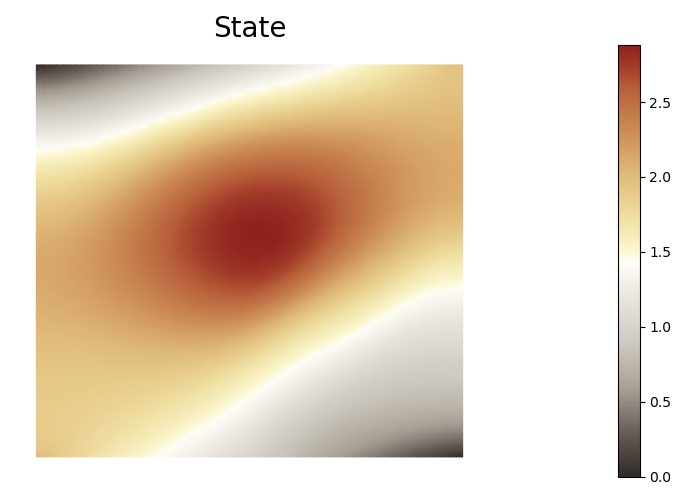}
                \end{subfigure}
            \end{tabular}
        \end{subfigure} 
        
        \begin{subfigure}[b]{\nonlinearWidth\textwidth}
            \begin{tabular}{c}
                \begin{subfigure}{\textwidth}
                    \centering
                    \includegraphics[trim=2.5cm 2.5cm 7.0cm 1.5cm,clip=true,width=\textwidth]{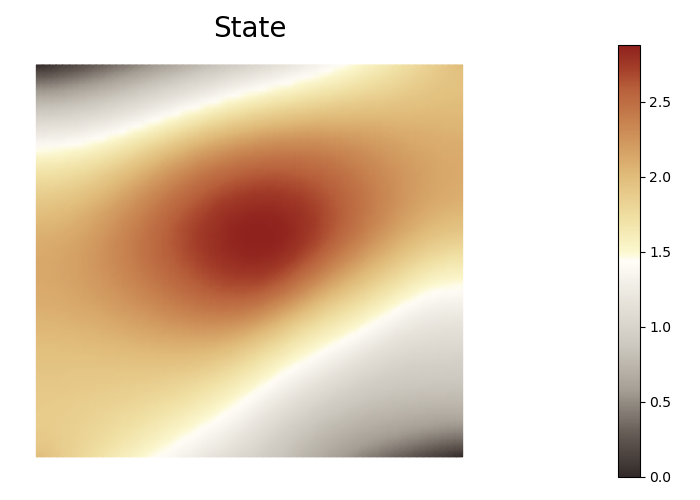}
                \end{subfigure}
            \end{tabular}
        \end{subfigure} \\

        \rotatebox{90}{\hspace{-0.2cm}RS}
        \begin{subfigure}[b]{\nonlinearWidth\textwidth}
            \begin{tabular}{c}
                \begin{subfigure}{\textwidth}
                    \centering
                    \includegraphics[trim=2.5cm 2.5cm 7.0cm 1.5cm,clip=true,width=\textwidth]{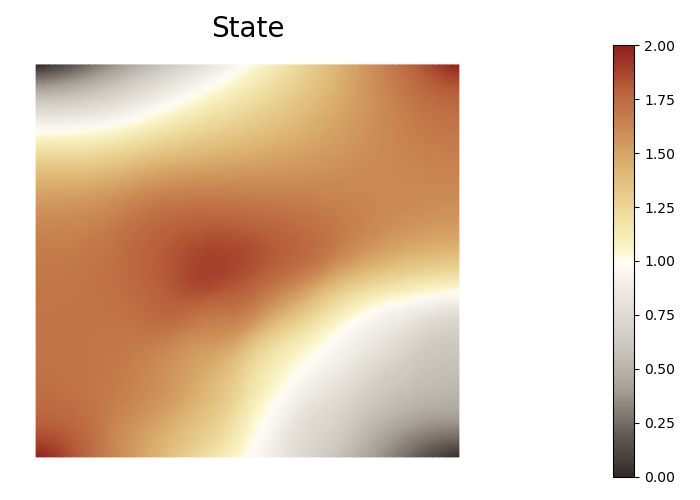}
                \end{subfigure}
            \end{tabular}
        \end{subfigure}

        \begin{subfigure}[b]{\nonlinearWidth\textwidth}
            \begin{tabular}{c}
                \begin{subfigure}{\textwidth}
                    \centering
                    \includegraphics[trim=2.5cm 2.5cm 7.0cm 1.5cm,clip=true,width=\textwidth]{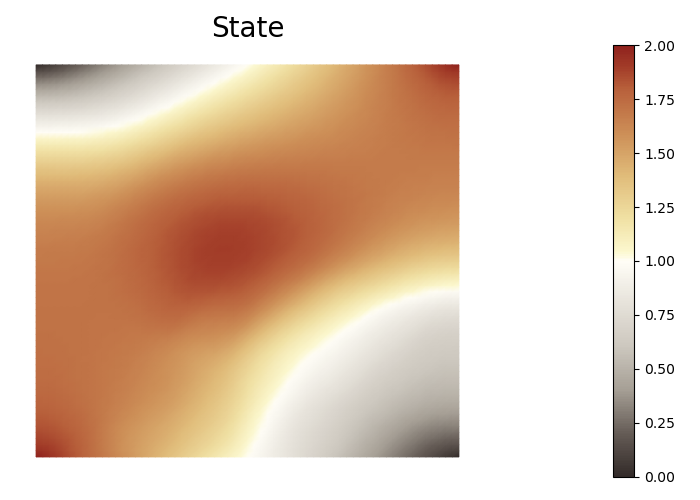}
                \end{subfigure}
            \end{tabular}
        \end{subfigure}
        
        \begin{subfigure}[b]{\nonlinearWidth\textwidth}
            \begin{tabular}{c}
                \begin{subfigure}{\textwidth}
                    \centering
                    \includegraphics[trim=2.5cm 2.5cm 7.0cm 1.5cm,clip=true,width=\textwidth]{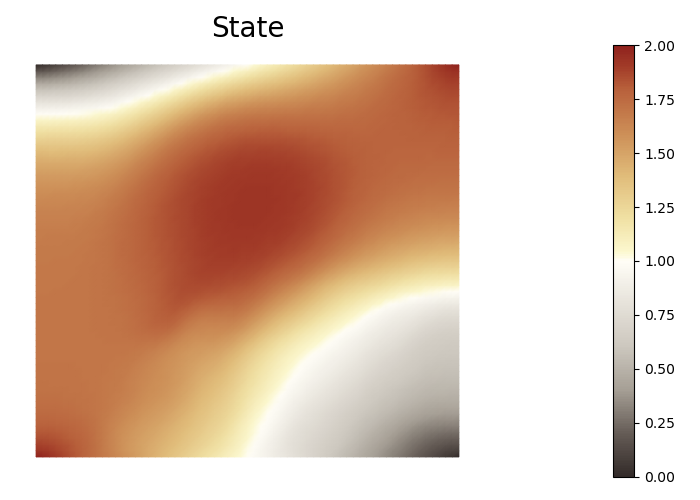}
                \end{subfigure}
            \end{tabular}
        \end{subfigure} 
        
        \begin{subfigure}[b]{\nonlinearWidth\textwidth}
            \begin{tabular}{c}
                \begin{subfigure}{\textwidth}
                    \centering
                    \includegraphics[trim=2.5cm 2.5cm 7.0cm 1.5cm,clip=true,width=\textwidth]{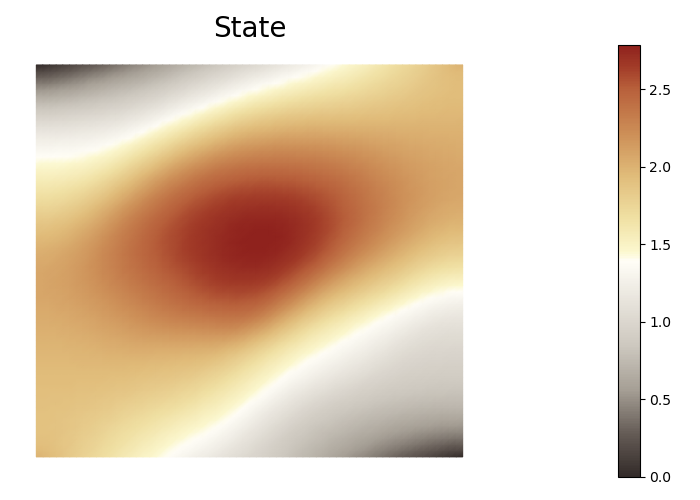}
                \end{subfigure}
            \end{tabular}
        \end{subfigure}\\

        \rotatebox{90}{\hspace{-0.5cm}ENKF}
        \begin{subfigure}[b]{\nonlinearWidth\textwidth}
            \begin{tabular}{c}
                \begin{subfigure}{\textwidth}
                    \centering
                    \includegraphics[trim=2.5cm 2.5cm 7.0cm 1.5cm,clip=true,width=\textwidth]{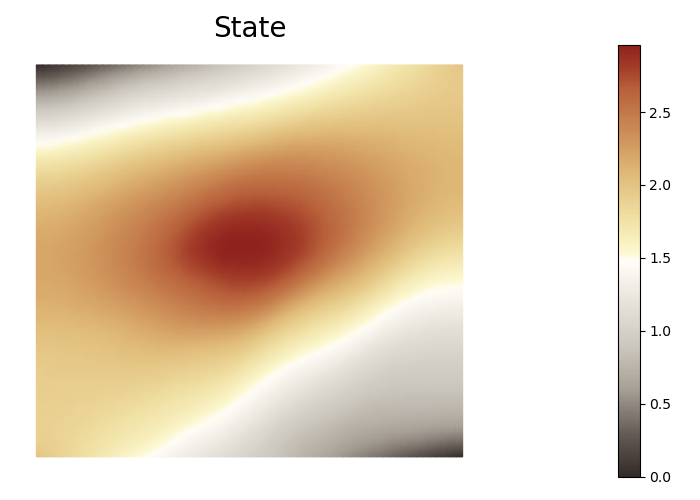}
                \end{subfigure}
            \end{tabular}
        \end{subfigure}

        \begin{subfigure}[b]{\nonlinearWidth\textwidth}
            \begin{tabular}{c}
                \begin{subfigure}{\textwidth}
                    \centering
                    \includegraphics[trim=2.5cm 2.5cm 7.0cm 1.5cm,clip=true,width=\textwidth]{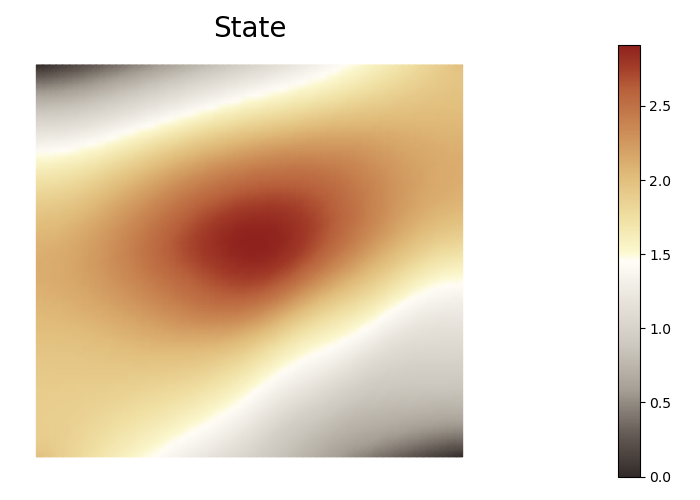}
                \end{subfigure}
            \end{tabular}
        \end{subfigure}
        
        \begin{subfigure}[b]{\nonlinearWidth\textwidth}
            \begin{tabular}{c}
                \begin{subfigure}{\textwidth}
                    \centering
                    \includegraphics[trim=2.5cm 2.5cm 7.0cm 1.5cm,clip=true,width=\textwidth]{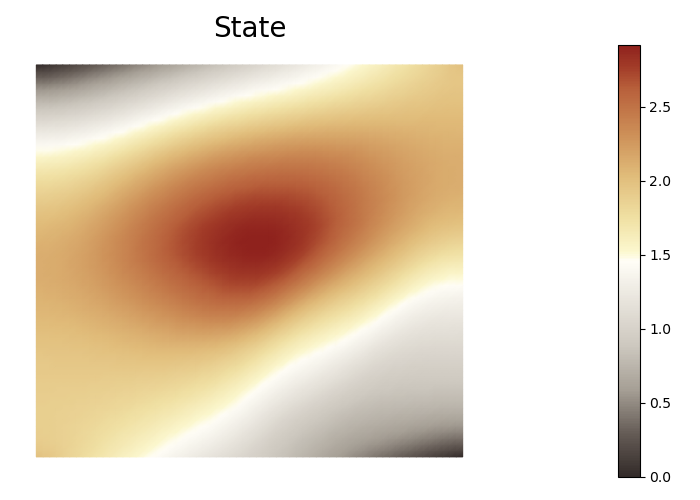}
                \end{subfigure}
            \end{tabular}
        \end{subfigure} 
        
        \begin{subfigure}[b]{\nonlinearWidth\textwidth}
            \begin{tabular}{c}
                \begin{subfigure}{\textwidth}
                    \centering
                    \includegraphics[trim=2.5cm 2.5cm 7.0cm 1.5cm,clip=true,width=\textwidth]{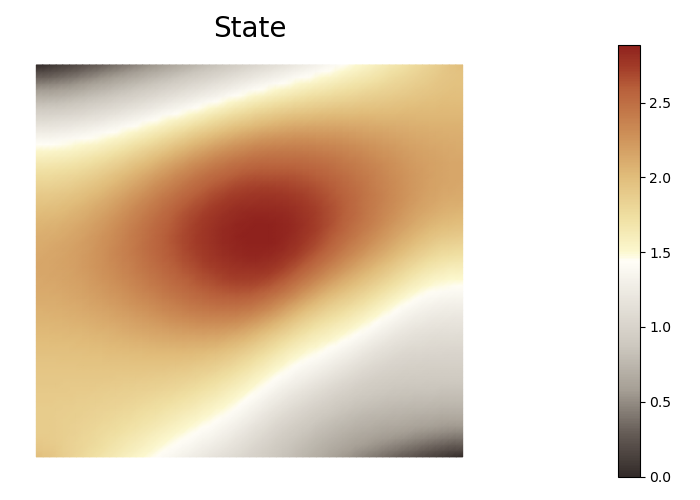}
                \end{subfigure}
            \end{tabular}
        \end{subfigure}\\
        
        \rotatebox{90}{\hspace{-0.3cm}ALL}
        \begin{subfigure}[b]{\nonlinearWidth\textwidth}
            \begin{tabular}{c}
                \begin{subfigure}{\textwidth}
                    \centering
                    \includegraphics[trim=2.5cm 2.5cm 7.0cm 1.5cm,clip=true,width=\textwidth]{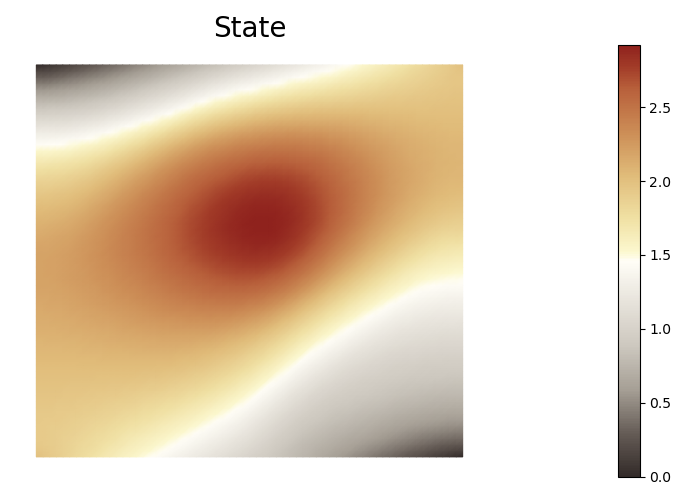}
                \end{subfigure}
            \end{tabular}
        \end{subfigure}

        \begin{subfigure}[b]{\nonlinearWidth\textwidth}
            \begin{tabular}{c}
                \begin{subfigure}{\textwidth}
                    \centering
                    \includegraphics[trim=2.5cm 2.5cm 7.0cm 1.5cm,clip=true,width=\textwidth]{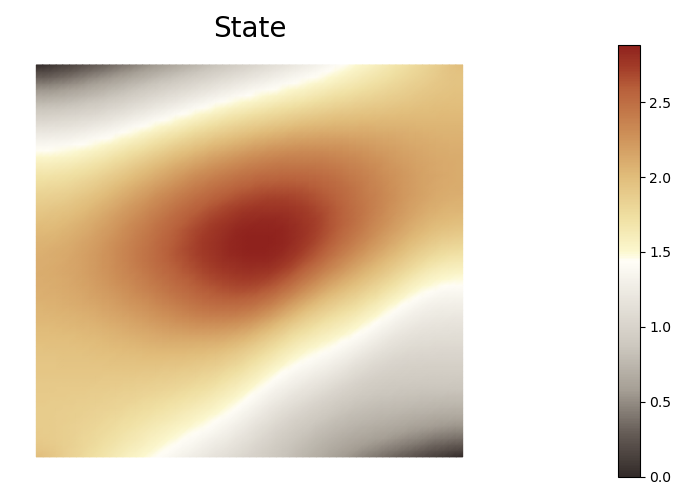}
                \end{subfigure}
            \end{tabular}
        \end{subfigure}
        
        \begin{subfigure}[b]{\nonlinearWidth\textwidth}
            \begin{tabular}{c}
                \begin{subfigure}{\textwidth}
                    \centering
                    \includegraphics[trim=2.5cm 2.5cm 7.0cm 1.5cm,clip=true,width=\textwidth]{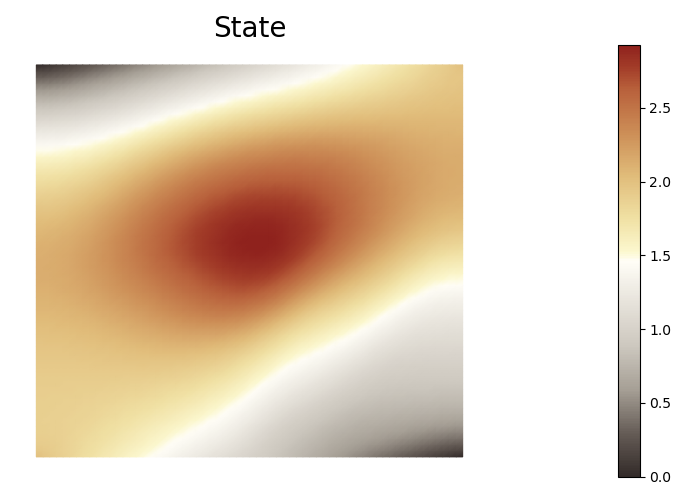}
                \end{subfigure}
            \end{tabular}
        \end{subfigure} 
        
        \begin{subfigure}[b]{\nonlinearWidth\textwidth}
            \begin{tabular}{c}
                \begin{subfigure}{\textwidth}
                    \centering
                    \includegraphics[trim=2.5cm 2.5cm 7.0cm 1.5cm,clip=true,width=\textwidth]{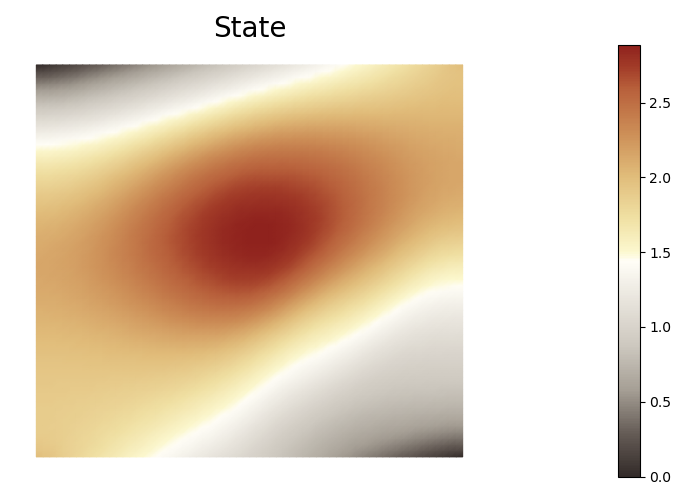}
                \end{subfigure}
            \end{tabular}
        \end{subfigure}\\

  \end{tabular}
    \caption{Reconstructed state for various randomization approaches for a nonlinear diffusion problem. Even for the right sketching 
    method which did not give good parameter reconstructions 
    until $N = 10000$ samples, the state in the lower half of the
    domain, where the 100 measurements are taken, look similar to all the other methods. }
    \label{Nonlinear_PDE_state}
\end{figure} 

\FloatBarrier


\newpage

\bibliography{ceo_new}

\end{document}